\documentclass[12pt, a4paper]{article} 
\usepackage[english]{babel} 
\usepackage[utf8]{inputenc} 
\usepackage{amsmath}		
\usepackage{amssymb}           
\usepackage{amsthm}
\usepackage{color}
\usepackage{extarrows}
\usepackage{graphicx}		
\usepackage{fancyhdr}		
\usepackage{setspace}		
\usepackage{float}
\usepackage{subfigure}
\usepackage{bm}
\usepackage{color}
\usepackage{indentfirst}
\newtheorem{theorem}{Theorem}[section]

\newtheorem{remark}{Remark}[section]
\newtheorem{corollary}{Corollary}[section]

\numberwithin{equation}{section}

\DeclareMathOperator\erfc{erfc}

\newcommand{\RNum}[1]{\uppercase\expandafter{\romannumeral #1\relax}}
\newcommand{\beq}{\begin{equation}}
\newcommand{\eeq}{\end{equation}}
\newcommand{\beqq}{\begin{equation*}}
\newcommand{\eeqq}{\end{equation*}}
\newcommand{\bal}{\begin{aligned}}
\newcommand{\eal}{\end{aligned}}

\newcommand{\mR}{\mathbb{R}}

\begin{document}
\title{
Numerical simulation of   singularity propagation modeled by linear convection equations with spatially heterogeneous nonlocal interactions}

\author{Xiaoxuan Yu\footnote{School of Mathematical Sciences, University of Science
and Technology of China, Hefei, Anhui 230026, P.R. China.
Email: xxyu@mail.ustc.edu.cn.} \and Yan Xu\footnote{School of Mathematical Sciences, University of Science
and Technology of China, Hefei, Anhui 230026, P.R. China.
Email: yxu@ustc.edu.cn. Research supported by NSFC grants 12071455. }
\; \and 
Qiang Du\footnote{
Department of Applied Physics and Applied Mathematics,  and Data Science Institute,
Columbia University, New York, NY 10027, USA. Email: qd2125@columbia.edu. Research supported in part by NSF DMS-2012562.
}}
\date{}
\maketitle
\begin{abstract}
{We study the propagation  of singularities in solutions of  linear convection equations with spatially heterogeneous nonlocal interactions.  A spatially varying nonlocal horizon parameter is adopted in the model, which measures the range of nonlocal interactions.  Via heterogeneous localization, this can lead to the seamless coupling of the local and nonlocal models. 
We are interested in understanding the impact on singularity propagation due to the heterogeneities of nonlocal horizon and the local and nonlocal transition.
We  first  analytically derive equations to characterize the propagation of different types of singularities for various forms of  nonlocal horizon parameters in the nonlocal regime. We then use 
asymptotically compatible schemes to discretize the equations and carry out numerical simulations to illustrate the propagation patterns in different scenarios.}
\end{abstract}

\section{Introduction}
Nonlocal models have been increasingly used in the simulation of physical processes involving some long-range interactions and singular or anomalous behaviors and they have also become useful  analysis tools in various applications. Examples include fracture mechanics, traffic flows, phase transitions, and image processing \cite{silling2010peridynamic,Amadori2012,Betancourt2011,Colombo2016,Goettlich2014,Huang2020,phase_tran,image_pro}.
A special class of nonlocal models are ones with a finite range of nonlocal interactions parameterized by what is called the nonlocal horizon \cite{silling2010peridynamic,du2019nonlocal}. As the horizon parameter diminishes to zero, the models could be localized to classical PDEs in some appropriate sense, particularly in regimes that the solutions are smooth. Such features also imply the possibility of coupling local PDEs with nonlocal models with a finite horizon parameter \cite{fang2019method,galvanetto2016effective,72,93,du2017,wang2019couplingofperidynamics,dispersion,
helenli}, we refer to a recent review \cite{delia2021review} for more references.   The motivation is easy to make:  while nonlocal models are better than the local models in simulating singular phenomena or long-range interaction, at the same time,  they often incur higher computational cost. Therefore, it is natural to consider the coupling so that the nonlocal model  is used only in regions that are necessary, such as the area near a fracture, while the local model can still be used in other regions. 
Among many established coupling strategy, one approach is to use a horizon parameter that vary in space. The latter is also a possible feature of complex multiscale systems involving highly heterogeneous materials properties.
An early study  on nonlocal peridynamics models with a variable horizon parameter has been made in 
\cite{silling2015variable}, though the models adopted are assumed to have  a variable horizon parameter with a positive lower bound over the domain. Later studies have also been made to consider a  variable horizon parameter with a transition from some region with positive parameter values to a region  with parameter  taking value zero,  which led to the heterogeneous localization of nonlocal models \cite{253,du2016heterogeneously}. This also enables a seamless coupling  of  a nonlocal model (with nonzero horizon parameter) and its local limit, which are obtained from the nonlocal models  as the horizon parameter approaches zero \cite{tao2019nonlocal,253}.    For a nonlocal wave equation defined in the whole space involving a variable horizon parameter, numerical studies based on asymptotically compatible discretization schemes have been carried with the help of nonlocal artificial boundary conditions \cite{zhangdu2018}.  One can also find other studies related to variable horizon  parameter in \cite{nikpayam2019variable,imachi2021smoothed}.

\par
In order to gain  further insight into the models with heterogeneous nonlocal interactions and seamless coupled local and nonlocal models, a nonlocal analog of linear first order convection equation
can be used as a prototype example. In a previous work \cite{ac_2021},  asymptotically compatible discretizations have been developed for such a seamlessly coupled local and nonlocal model with a heterogeneously defined horizon parameter function.
Numerical simulations were also presented with the focus on the propagation of waves between nonlocal and local regions when the initial conditions and horizon parameter functions are reasonably smooth.  This work serves as a follow-up, but with a new focus on the properties of solutions  corresponding to initial conditions and horizon parameter functions that  lack sufficient smoothness. 
The latter is an important subject as the flexibility of nonlocal models in allowing singular solutions has been a widely recognized feature  of nonlocal modeling  \cite{silling2010peridynamic,du2019nonlocal}. Thus, it is paramount to have a good understanding on how the propagation of solution singularities is affected by the various modeling characteristics. In this work, we pay particular attention to the impact on the solution behavior due to the different forms and the spatial variations of the nonlocal horizon parameter. 
We conduct studies in both theoretical and numerical fronts, and examine cases involving either non-smooth initial data or with horizon parameter functions that are not smoothly varying, particularly around the local and nonlocal transition region. Studies and the findings in these directions have not been systematically presented in the literature before. They provide us a more clear picture on the similarities and differences of the propagation of singularities modeled through local and nonlocal equations and lead to better understanding of nonlocal modeling.

{The paper is organized as follows. In Section 2, we introduce the models and problems studied in this paper. In Section 3, we carry out some theoretical analysis to derive analytically equations satisfied by the jump of the derivative of the solutions in the nonlocal setting. In Section 4, we first present different numerical schemes to simulate the solutions and  their possible singular features. 
We then discuss the numerical results and make comparisons. Section 5 summarizes the work.}

\section{The nonlocal convection model}
We begin with a simple example of  scalar one-dimensional  linear convection PDEs given by
 \beq\label{eq0}
\left\{\begin{array}{l}
{u_{t }(x, t)+u_x(x, t)=0}
,\quad x\in \mR, \; t\in \mR_+
 \\ {u(x, 0)=\psi_{0}(x),  \quad x\in \mR.
 }\end{array}\right.
\eeq
In this paper, we consider the linear nonlocal conservation law below as a nonlocal analog of \eqref{eq0}:
 \beq\label{eq1}
\left\{\begin{array}{l}
{u_{t }(x, t)+\mathcal{D}u(x, t)=0}
,\quad x\in \mR, \; t\in \mR_+
 \\ {u(x, 0)=\psi_{0}(x),  \quad x\in \mR,
 }\end{array}\right.
\eeq
where for a scalar valued function $u=u(x): \mR \rightarrow \mathbb{R}$, the nonlocal derivative operator $\mathcal{D}$ acting on $u$ is defined as
 \beq\label{eq:opD}
\mathcal{D}u(x)=\int_{\mR^+}[ u(x)-u(x-s)] \gamma(s,x) ds,  \quad x\in \mR .
\eeq
Here, $\gamma(s, x): \mathbb{R} \times \mathbb{R} \rightarrow \mathbb{R}$ is called a nonlocal kernel function.  More discussions of nonlocal derivative operators can be found in \cite{dyz17dcdsb,du2019nonlocal} and the references cited therein.
 Since  $\mathcal{D}u(x)$
  involves only the function value of $u=u(y)$ for $y\geq x$, it is also called an upwind nonlocal derivative. We refer to \cite{tian2015nonlocal,du2014nonlocal,delia2017nonlocal}
 for other related studies on similar nonlocal models. 

\subsection{The nonlocal kernel and nonlocal horizon parameter}
In general, we assume that the nonlocal kernel
$\gamma=\gamma(s,x)$ is nonnegative with either a compact support  or a sufficiently fast decay in $s$.
The kernel is also assumed to satisfy suitable normalization conditions so that $\mathcal{D}$ 
approaches the classical local derivative operator in the local limit when acting on smooth functions.

More specifically, 
 $\gamma=\gamma(s, x)$ is defined via a reference kernel function $H=H(s)$ that satisfies.
\begin{equation}\label{conditionH}
H(-s) = H(s)\geq 0,\; \forall s\in \mathbb{R} \quad \text{and}\quad \int_0^{+\infty}s H(s)ds=1.
\end{equation}
In this work, $H=H(s)$  is assumed to be smooth and  bounded on $\mR^+$. Moreover, it 
 either has a compact support or is a function  that itself and its derivatives
all decay sufficiently fast 
as $|s|\to \infty$. 

Further, in order to better represent the spatial variations of horizon parameter, we introduce  separately a  nonnegative, spatially variable horizon parameter $\zeta=\zeta(x)$ which controls
the range of nonlocal interactions and how horizon changes in space \cite{du2016heterogeneously}. And we denote that $\delta=\| \zeta\|_\infty<\infty$ and 
  $\Omega_{nl}=\{x\in \mR, \zeta(x)>0\}$. The latter may be referred to as  the nonlocal region.
  
 We now present a more concrete form of $\gamma(s, x)$ as
\beq\label{gamvar}
\gamma(s, x)=\frac{1}{\zeta^2(x)}H\left(\frac{s}{\zeta(x)}\right), \quad\forall x\in \Omega_{nl}.
\eeq
By the assumption \eqref{conditionH} on $H=H(s)$, we have
\begin{equation}\label{normalization}
\int_{\mR^+}s\gamma(s,x)ds=1.
\end{equation}

Associated with $\gamma= \gamma(s,x)$, we define $k=k(x)>0$ by
\begin{equation}
\label{kx}
k(x) = \int_{0}^{+\infty}\gamma(s,x)ds.
\end{equation}
It is easy to see that $k(x)$ is bounded for $x\in \Omega_{nl}$. Moreover, we have by 
 \eqref{gamvar} and the decay properties of $H=H(s)$ that
\begin{equation}
\label{kx1}
k(x) =\int_{\mR^+}
\frac{1}{\zeta(x)}H\left(\frac{s}{\zeta(x)}\right) d \frac{s}{\zeta(x)} = \frac{1}{\zeta(x)}
\int_{\mR^+} y H(y)dy <\infty.
\end{equation}
Similarly,  by the decay properties of $H=H(s)$, we have
\begin{equation}\label{eq3}
\int_{\mR^+}s^2 
\gamma(s, x)ds={\zeta(x)} 
\int_{\mR^+} y^2 H(y)dy \leq \delta  \int_{\mR^+} y^2 H(y)dy<\infty.
\end{equation}
In addition,  if $\zeta'(x)$ is well-defined at some point $ x\in\Omega_{nl}$, 
 then by differentiating \eqref{gamvar},
we get that 
\begin{equation}\label{eq4}
\frac{\partial \gamma}{\partial x}(s, x)=- \zeta^{\prime}(x)\left(\frac{2 
}{\zeta^{3}(x)} H\left(\frac{s}{\zeta(x)}\right) + \frac{s}{\zeta^{4}(x)}  H^{\prime}\left(\frac{s}{\zeta( x)}\right)\right).
\end{equation}
We can see that at such a point $x$, 
\begin{equation}\label{eq4.1}
\int_{\mR^+} \left| \frac{\partial \gamma}{\partial x}
(s, x)\right|ds \leq \frac{| \zeta'(x)|}{|\zeta(x)|^2} \int_{\mR^+}( 2 H(y) + y|H'(y)|)dy <\infty.
\end{equation}

\subsection{The local limit}
For any $ x\in \Omega_{nl}$,  as shown in \cite{ac_2021},  by Taylor expansion and \eqref{eq3}, we formally 
have that
$$
\begin{aligned}
& \left| \mathcal{D}u(x,t) - u_x(x,t)\right|  =\left|\int_{\mR^+}[ u_x(x,t)s-\frac12 u_{xx}(\theta,t)s^2] \gamma(s,x) ds- u_x(x,t)\right|\\
&\qquad \leq  \frac{1}{2} \|u_{xx}(\cdot,t)\|_\infty 
\int_{\mR^+}s^2 
\gamma(s, x)ds\\
&\qquad =
 \frac{1}{2}  \|u_{xx}(\cdot,t)\|_\infty\,  {\zeta(x)}
\int_{\mR^+} y^2 H(y)dy 
\to 0,  \quad \text{as} \quad  \zeta(x)\to 0
 \end{aligned}
$$
for smooth solutions $u=u(x,t)$. 

Consequently, we formally
recover the local limit  \eqref{eq0}
from \eqref{eq1} by taking $\delta\to 0$.  Due to this relationship, the equation \eqref{eq1} is often called a linear nonlocal convection equation (with an upwinding nonlocal space derivative).

Note that if we are at a point where $\zeta(x)=0$ at some $x\in \mR$, then we may interpret $s\gamma(s,x)$ 
via the $\zeta(x)\to 0$ limit  as
the Dirac delta measure at $s=0$ so that
$\mathcal{D}u(x,t)=u_x(x,t)$. 
We can also set $k(x)=+\infty$ at such a point if one wishes to have $k=k(x)$ defined for all cases.

In this work, we focus on the case that $\zeta(x)$ is a spatially varying  continuous function, including the cases of smooth and piecewise smooth functions.  Given that $H=H(s)$ is assumed to be smooth, the regularity of $\gamma(s,x)$ is essentially determined by the horizon parameter function $\zeta=\zeta(x)$.

\section{Propagation of singularities}

The main questions studied in this work are related with the impact of the nonlocal interactions on the 
propagation of singularities.  For linear model problems, there have been earlier studies on the subject. For example,  one can show the stationarity of the initial jump discontinuities in the linear peridynamics equation of motion when a constant horizon is taken \cite{weckner2005effect}.
For the nonlocal models studied here, in the case
of a smooth kernel and a smoothly defined variable horizon, similar studies has been made in  \cite{ac_2021}.

Concerning equation  \eqref{eq1}, it is clear that in the local region, any discontinuity or other singularity in the local region would travel along characteristic lines. Thus, we focus on potential singularities in the nonlocal region.  For this reason, we assume that in the rest of this section, $\zeta(x)>0$ for all $x\in \mR$, that is,  $\Omega_{nl}=\{x\in \mR, \zeta(x)>0\}=\mR$. To allow more precise statements of the findings, we further assume that 
$$\delta_m=\min_{x\in \mR} \zeta(x)>0,
$$
so that the solutions discussion in this section always refer to the nonlocal ones.

\subsection{Propagation of initial discontinuities}

We first recall a result discussed in  \cite{ac_2021}, which is presented below in a more precise form, concerning the propagation of 
discontinuities generated from the initial data

\begin{theorem}
\label{thm:jump-u}
Let $u=u(x, t)$ denote the solutions of \eqref{eq1}. Assume that 
 $\zeta=\zeta(x)$ is continuous and bounded on $\mR$ and $\psi_0$ is piecewise continuous and uniformly bounded with a finite number of discontinuities on $\mR$.
 Then, 
 \beq
\label{eq:jump-u}
[u(x,t)] = e^{-k(x)t}[u(x,0) ]= e^{-k(x)t}[\psi_0(x)], \quad\forall t\in \mR^+\,, x\in \mR.
\eeq
\end{theorem}

\begin{proof}
We first observe that for the given initial condition, the solution of \eqref{eq1}
 remains uniform bounded due to the maximum principle.  Moreover, by  \eqref{eq3},
the smoothness assumption on $H=H(s)$ and $\zeta=\zeta(x)$,
 and the assumption that $\delta_m=\min_{x\in \mR} \zeta(x)>0$, we see that
 $k=k(x)$ is smooth and uniformly bounded for any $ x\in\mR$.

 At a point $x_0\in\mR$, we take a sufficiently small
  $\epsilon>0$ and calculate,  as in \cite{ac_2021}, the time derivative of a weighted difference of the solution $u$:
\beqq
\begin{aligned}
&  \frac{d}{dt} \left[ e^{k(x_0+\epsilon)t}u(x_0+\epsilon,t)-e^{k(x_0-\epsilon)t}u(x_0-\epsilon,t)\right]\\
 &\;\; = k(x_0+\epsilon)e^{k(x_0+\epsilon)t}u(x_0+\epsilon,t)-k(x_0-\epsilon)e^{k(x_0-\epsilon)t}u(x_0-\epsilon,t)\\
&\qquad+ e^{k(x_0+\epsilon)t}\left(  -\int_0^{+\infty}(u(x_0+\epsilon,t)- u(x_0+\epsilon - s,t))\gamma(s, x_0+\epsilon)ds\right)\\
&\qquad+ e^{k(x_0-\epsilon)t}\left(\int_0^{+\infty}(u(x_0-\epsilon,t) - u(x_0-\epsilon - s,t))\gamma(s,x_0-\epsilon)ds\right)\\
&\;\; = e^{k(x_0+\epsilon)t} \!\int_0^{+\infty}  \! \!u(x_0\!+\!\epsilon \!- \!s,t)\gamma(s, x_0\!+\!\epsilon)ds \!\\
&\qquad -  \!e^{k(x_0-\epsilon)t} \!\int_0^{+\infty} \!\!u(x_0\!-\!\epsilon \!- \!s,t)\gamma(s,x_0\!-\!\epsilon)ds.
\end{aligned}
\eeqq
Now,  we take $\epsilon\to 0$.
With the assumptions on $\gamma$ and the boundedness of $u$, we can pass the limit inside the integrals to get
\beqq
\begin{aligned}
&  \frac{d}{dt} \left[ e^{k(x_0+\epsilon)t}u(x_0+\epsilon,t)-e^{k(x_0-\epsilon)t}u(x-\epsilon,t)\right]\\
&\quad  \to  e^{k(x_0)t} \!\int_0^{+\infty}  \! \!u(x_0\!- \!s,t)\gamma(s, x_0)ds \!-  \!e^{k(x_0)t} \!\int_0^{+\infty} \!\!u(x_0\! - \!s,t)\gamma(s,x_0 )ds\\
&\quad = 0, \quad \text{as}\quad \epsilon\to 0.
\end{aligned}
\eeqq
\begin{equation}\label{eq:jump-u1}
\frac{\partial }{\partial t} \left\{
e^{k(x_0)t}
\bm{[}u(x_0,t)\bm{]} \right\}=0.
\end{equation}
This means that $e^{k(x_0)t}[u(x_0,t) ]$ is constant in time, which implies \eqref{eq:jump-u}.
\end{proof}

From the above theorem, we also get the following corollary.

\begin{corollary}
Under the assumption of theorem \ref{thm:jump-u},
 the solution remains  piecewise continuous with only points of discontinuities inherited from  the initial data.
\end{corollary}

 The results presented above might seem to be inconsistent with those corresponding to a pure local convection equation at first sight since, for the local limit \eqref{eq0},  piecewise smooth solutions can still be defined with the discontinuities travel with a constant speed along the local characteristic lines instead of being stationary in space. However, 
by \eqref{kx1}, we see that 
$$k(x) 
=  \frac{1}{\zeta(x)}
\int_{\mR^+} y H(y)dy.
$$
Thus, as $\zeta(x)\to 0$, the factor $e^{-k(x)t}$ decays exponentially fast at the given $x$ for any given $t>0$. The factor vanishes in the local limit as it should so that these stationary discontinuities die down fast and thus consistency with the local model can still be preserved in the local limit.
Even for a finite $\zeta(x)>0$, this apparent discontinuity is expected to be short lived as it also decays exponentially fast in $t$  for $t>0$.

At the same time, we note that the above result also implies that in the nonlocal region,  an initial discontinuity does not travel along the characteristic line 
$x=x_0+t$ of the local convection equation. This can be seen as  a result of viscosity implicitly encoded in the nonlocal unwind derivative. The viscous effect again diminishes as $\delta\to 0$ so that we still expect the propagation of the {\em smoothed} transition layer, rather than a sharp discontinuity,
traveling along the local characteristics, and the transition layer gets steeper as $\delta$ gets smaller.

The above type of studies on $[u]$ can also be applied to study behavior of $[u_x]$ as we discuss next, which can help with the understanding of the numerical experiments.

\subsection{The propagation of $ [u_x]$}

Similar to theorem \ref{thm:jump-u}, we can derive analogous results on $ [u_x]$.
\begin{theorem}
\label{thm:jump-ux}
Let $u=u(x, t)$ denote the solutions of \eqref{eq1}. Assume that 
 $\zeta=\zeta(x)$ is uniformly bounded and smooth with a bounded derivative on $\mR$ and $\psi_0$ is  continuous and uniformly bounded on $\mR$. In addition, 
 assume that $\psi_0$ is piecewise smooth with a piecewise bounded local derivative  $\psi_{0x}$ that only has
  a finite number of jump discontinuities.
 Then,  \\
 \beq
 \label{eq:jump-ux}
 \bm{[}u_x\bm{]} (x,t)
= e^{-k(x)t}  \bm{[}u_x\bm{]} (x,0)= 
 e^{-k(x)t}\bm{[}\psi_{0x}\bm{]} (x), \quad\forall t\in \mR^+\,, x\in \mR.
\eeq
\end{theorem}
\begin{proof}
First, from the theorem \ref{thm:jump-u} and the assumption on the continuity of $\psi_0$, we see that $u=u(x,t)$ is continuous and uniform bounded  in $x$ for all $t\in \mR^+$.

By taking the classical derivative of $x$ on both sides of the equation \eqref{eq1}, we have 
\begin{equation}\label{eq2}
\left\{\begin{array}{l}
\displaystyle
\frac{\partial u_x(x,t)}{\partial t}+\int_{\mR^{+}}\left(
u_x(x, t) - u_x(x-s, t)\right) \gamma(s, x) d s \\
 \quad \displaystyle 
+\int_{\mR^{+}}\left(u ( x, t)-u(x-s, t)\right)
 \frac{\partial \gamma(s, x)}{\partial x} d s=0 , \;\;\forall x \in \mR\, a.e.,\\[.1cm]
u_{x}( x, 0 )=\psi_{0 x}(x),
\end{array}\right.
\end{equation}
Thus, we may  view $u_x$ as a solution to an equation of the same form as \eqref{eq1} but with a nonzero right hand side that is uniformly bounded by the assumptions on the kernel and horizon parameter functions. Hence, if  $\psi_{0x}$ is  uniformly bounded, we also see that $u_x$ stays uniformly bounded.

Similar to the proof of theorem \ref{thm:jump-u},
at  a point $x_0\in\mR$, we take a sufficiently small $\epsilon>0$ such that the equation \eqref{eq2} holds at $x_0\pm \epsilon$.
Then we use the integration factor and  subtract the first  equations of \eqref{eq2}  evaluated at $x_0\pm \epsilon$ respectively to get
\begin{eqnarray}\label{eq5}
& & \frac{\partial }{\partial t}
[ e^{k(x_0+\epsilon)t}
u_x(x_0 + \epsilon,t) - e^{k(x_0-\epsilon)t} u_x(x_0 -\epsilon,t)
] \nonumber\\
&& \;\;= e^{k(x_0-\epsilon)t} \int_{\mR^{+}}\left(u ( x_0 - \epsilon, t)-u(x_0 - \epsilon-s, t)\right) 
\frac{\partial \gamma}{\partial x}(s, x_0 - \epsilon) ds   \nonumber\\
&&\quad  - e^{k(x_0+\epsilon)t}  \int_{\mR^{+}}\left(u ( x_0 + \epsilon, t)-u(x_0 + \epsilon-s, t)\right) 
\frac{\partial \gamma}{\partial x}(s, x_0 + \epsilon) ds \nonumber\\
&& \quad 
 -  e^{k(x_0-\epsilon)t} \int_{\mR^{+}} u_x(x_0 - \epsilon-s, t)
\gamma(s, x_0 - \epsilon) ds 
\nonumber\\
&& \quad +
e^{k(x_0+\epsilon)t} \int_{\mR^{+}} u_x(x_0 + \epsilon-s, t)
\gamma(s, x_0 + \epsilon) ds .
\end{eqnarray}
Then, by the continuity of $u(x,t)$ and $k(x)$ in $x$, the uniform bound on $u_x(x,t)$ and the regularity and decay properties of $\gamma(s,x)$ and its derivative, we let $\epsilon
\to 0$ and pass limit inside the integrals to get
\begin{equation}\label{eq:jump-ux1}
\frac{\partial }{\partial t} \left\{
e^{k(x_0)t}
\bm{[}u_x\bm{]} (x_0,t)\right\}=0.
\end{equation}
This leads to equation \eqref{eq:jump-ux}.
\end{proof}
\begin{remark}
From \eqref{eq1} and \eqref{eq2}, we see that there is some difference between the equations of
$u$ and $u_x$ but
 $\bm{[}u\bm{]}$ and $\bm{[}u_x\bm{]}$ satisfy the same equation for different initial value. 
\end{remark}
\begin{remark}
Similar to the possible jump discontinuity in a solution $u$ with a discontinuous initial data,  we see that for continuous initial data with a discontinuous derivative,  the spatial position of the jump discontinuity of  $u_x$ also does not change with time and no new discontinuity is generated. Further, the jump of $u_x$ also decreases in time and also in the local limit.
\end{remark}

We now consider the case with a horizon parameter function that may have a discontinuous derivative.

\begin{theorem}
\label{thm:jump-ux-dis}
Let $u=u(x, t)$ denote the solutions of \eqref{eq1}. Assume that 
 $\zeta=\zeta(x)$ and $\psi_0$ are continuous and uniformly bounded on $\mR$. In addition, 
 assume that $\psi_0$ and $\zeta$ are piecewise smooth with  piecewise bounded local derivatives  $\psi_{0x}$ and $\zeta'$
 having only
  a finite number of jump discontinuities. 
 Then, for $t>0$ and $x\in \mR$, we have
\begin{eqnarray}
\label{eq:jump-ux-dis}
&& \bm{[}u_{x}\bm{]}(x, t) = e^{-k(x) t}
\left[\psi_{0 x}\right](x) - \frac{\left[\zeta^{\prime}(x)
\right]}{\zeta^{4}(x)}
\!\int_{0}^{t}  e^{-k(x) (t-\tau)}
\int_{\mR^+}\!   \nonumber
\\
&&\;\;  \left(u(x, \tau)\!-\!u(x-s, \tau)\right)
\left(
2 \zeta(x)
 H(\frac{s}{\zeta(x)}) + s H^{\prime}(\frac{s}{\zeta( x)})\right)  ds
d\tau.
\end{eqnarray}
\end{theorem}

\begin{proof}

The proof process of the first half is similar to Theorem \ref{thm:jump-ux}. at  a point $x_0\in\mR$, we take a sufficiently small $\epsilon>0$ so that \eqref{eq5} holds.

Then, by the continuity of $u(x,t)$ and $k(x)$ in $x$, the uniform bound on $u_x(x,t)$ and the regularity and decay properties of $\gamma(s,x)$ and its derivative, we let $\epsilon
\to 0$ and pass limit inside the integrals to get
\begin{eqnarray}
\label{eq:jump-ux2}
&& \frac{\partial }{\partial t} \left\{
e^{k(x_0)t}
\bm{[}u_x\bm{]} (x_0,t)\right\}\nonumber\\
&&\qquad =e^{k(x_0)t}  \int_{\mR^{+}}
\bm{[}\frac{\partial \gamma}{\partial x}\bm{]}(s, x_0) \left(u ( x_0 , t)-u(x_0-s, t)\right) ds .
\end{eqnarray}
By integrating the two sides of the equation with respect to $t$ and using \eqref{eq4}, we can get equation \eqref{eq:jump-ux-dis}.
\end{proof}

\begin{remark} From the
theorem \ref{thm:jump-ux-dis}, we see that, the discontinuity of $u_x$ not only results from the discontinuity of  the derivative  of  the initial data, but also from the horizon parameter function $\zeta$ if the latter is continuous and has a derivative that is only piecewise  
continuous and bounded with jump discontinuities. The locations of the discontinuity of $u_x$  do not change with time. 
Moreover, while the jump of  $u_x$ due to the initial data decays exponentially in time, the jump due to the kernel function evolves in time in a more involved fashion given by the accumulated integral.
\end{remark}

\section{Numerical schemes and simulations}
{Here we use an asymptotically compatible scheme developed in \cite{ac_2021}
to numerically simulate the solution and its spatial derivative (when appropriate. We report the experimental findings to complement our analytical investigations given in the earlier section.
}

\subsection{Numerical schemes}
\label{sec:4-1}
We first describe the  numerical schemes used in the simulations.

 \subsubsection{Discretization of the nonlocal model \eqref{eq1}} \label{method 1}
Here, we directly use the numerical scheme used in \cite{ac_2021} without much elaboration, and directly introduce the specific formulae of the numerical scheme.
Let us use ${h}$ and $\tau$ to denote the spatial and time step size, $x_j=j
{h},\, t_n = n\tau$ as the spatial and time grid points. Denote $U_j^n$ as the numerical solution of $u$ at the grid point $(x_j, t_n)$.
We denote $\mathcal{D}_h$ as the spatial discretization   of the nonlocal operator $\mathcal{D}$.
\beq\label{sc1}
\begin{aligned} \mathcal{D}_h u(x_j, t)
&=\sum_{k \in \mathbb{Z}} a_{j, j-k} u(x_{j-k}, t),
\end{aligned}
\eeq
where 
\beq
a_{j, j-k}=\left\{\begin{array}{rc}\displaystyle-\int_{\mathbb{R}^+} \phi_{x_k}(s)\gamma(s, x_j)ds, & { \quad k>0,} \\\displaystyle -{\sum_{k \neq 0}\int_{\mathbb{R}^+} \phi_{x_k}(s)\gamma(s, x_j)ds}, & { \quad k=0,}\\ 0,  & {\quad k<0,} \end{array}\right.
\label{a-defn}
\eeq
for $\zeta(x_j)>0$. $\phi_{x_k}$ is the standard hat function on the mesh $\{x_j\}$ and centered at point $x_k$.  If $\zeta(x_j)=0$, the the values of $a_{j,l}$ need be refined by
\begin{align}
a_{j,l}=\left\{\begin{array}{ccl}{-\frac1h}, & {} & {l=j-1,} \\ {\frac1h}, & {} & {l=j,} \\ {0}, & {} & {\mbox{else,}}\end{array}\right.
\label{a-defl}
\end{align}
which gives exactly the coefficients of the first order local forward difference quotient operator.

For the discretization of time derivatives, we still use the forward Euler scheme. Then, combined with the previous spatial discretization scheme \eqref{sc1}, we can obtain the complete discretization scheme for the nonlocal equation \eqref{eq1}
\beq\label{sc2}
U_j^{n+1}=U_j^{n} -\tau\sum_{k \in \mathbb{Z}} a_{j, j-k} U_{j-k}^n\;.
\eeq
The numerical scheme \eqref{sc2} is known to be asymptotically compatible and 
provides a monolithic discretization to the nonlocal model
\eqref{eq1} encompassing both the local  (where $\zeta(x)=0$) and nonlocal regimes (where $\zeta(x)>0$).

In order to show the possible discontinuity of $u_x$ numerically, we also use the numerical solutions $U_{j + 1}^n, U_{j-1}^n$ and $U_{j}^n$ to calculate the jump of the difference quotient of $u_x$ where $x_j$ lies in the discontinuous point of $\psi_{0x}(x)$ 
and $\zeta'=\zeta'(r)$. 
Note that we need to choose an appropriate grid points in the numerical simulations so that the possible discontinuities just fall on some of the spatial nodes, say $x_j$, so as to ensure the calculation accuracy of $\bm{[}u_x\bm{]}$. Then we take
\beq
\bm{[}u_x\bm{]}(x_j,t_n) \approx \frac{U_{j +1}^n - U_{j}^n}{h} -  \frac{U_{j}^n - U_{j-1}^n}{h},
\eeq
as the numerical  approximation. For the jump in $u$ itself at $x_j$, we use 
\beq
\bm{[}u\bm{]}(x_j,t_n) \approx U_{j +1}^n - U_{j-1}^n.
\eeq

\subsubsection{Evaluation jumps by equations \eqref{eq:jump-u1}, \eqref{eq:jump-ux1} and \eqref{eq:jump-ux2}}\label{method 2}

As an independent check of the numerical results, particular concerning the discontinuity of $u$ and/or  $u_x$,  instead of
direct numerical differentiation for the latter, we also  numerically evaluate the equation
\eqref{eq:jump-u1}, \eqref{eq:jump-ux1} and \eqref{eq:jump-ux2}. 
The solutions of \eqref{eq:jump-u1} and \eqref{eq:jump-ux1} are given by \eqref{eq:jump-u} and \eqref{eq:jump-ux}  directly.
To find the jump of $u_x$ using \eqref{eq:jump-ux2}, we use the formula  \eqref{eq:jump-ux-dis}. For the time integration, we adopt the Riemann sum in, which is like using the forward Euler time-stepping of the ODEs   \eqref{eq:jump-ux2}. {For the spatial integration, we adopt the composite Trapezoidal rule  using grid points as quadrature points. Thus, we can use the numerically computed values $\{U^n_j\}$ of the solution to calculate the integral.}

\subsection{Numerical results}
Now, we show some numerical results for initial data $\psi_0$ and horizon parameter function $\zeta$ with different smoothness properties. In all the following experiments, we  consider an initial data with a compact support. This implies, for the time interval of interest, say for $t$ up to 1, 2, or 5,  the solution  remains sufficiently close to zero outside a sufficient large spatial domain. We then take such a truncated computation domain to conduct the simulations and treat the solution outside as zero so that the accuracy of the numerical solution within the region of interests can be ensured. The reported numerical results are obtained with the spatial grid size $h = 0.0125$ and the time step $ \tau = 0.00625$, which are also tested through refinement to ensure that they are small enough to give a sufficiently accurate solution. 

\subsubsection{$\psi_0$ smooth, $\zeta$ smooth}
First, let us consider the case where the initial data $\psi_0$ and the horizon parameter function $\zeta$ are both smooth.
Specifically, we choose
$$H(s) = 20 e^{-10s^2}, \quad \psi_0(x)=e^{-10x^2}$$
and 
$$\zeta(x)=\erfc(-\frac{x}{2^\alpha})\quad\text{for}\; \alpha =-1,\;0 \; \text{and} \;1\; \text{respectively}.$$
In Figure \ref{fig1}, we present the horizon parameter functions selected here, which are  three smooth functions with transition layers
becoming narrower as $\alpha$ changes from $1$ to $0$ then to $-1$.

 In this case, according to the previous theorem, we expect that there will be no discontinuity in the solution nor its derivative.
 The results of the solutions on $x\in (-2,4)$ 
are shown in Figure \ref{fig2} for $t\in (0, 2)$ with a top view
 and Figure \ref{fig3}   for $t\in (0, 1)$ (a zoomed-in version of Figure \ref{fig2}, with a 3D view)
 to illustrate the wave propagation.  
The smooth contours and surface plots indicate that smooth  solutions are obtained, as expected. 
The initial Gaussian peak gets smeared as time goes on, due to the diffusive effect present in the nonlocal equations. Moreover, the faster the wave enters the nonlocal region (as $\alpha$ gets smaller, from the left plot to the right plot), the more dissipative the solution exhibits  (as indicated by more dramatic color changes shown in the plots).

Similar observations as those made for this example have been reported in the literature, see for example \cite{ac_2021}, which is not our focus here.
They are included here solely for comparison purposes to contrast
with the other cases presented later.

\begin{figure}[htb]
\centering
\label{fig1:subfig:a} 
\includegraphics[width=7.5cm]{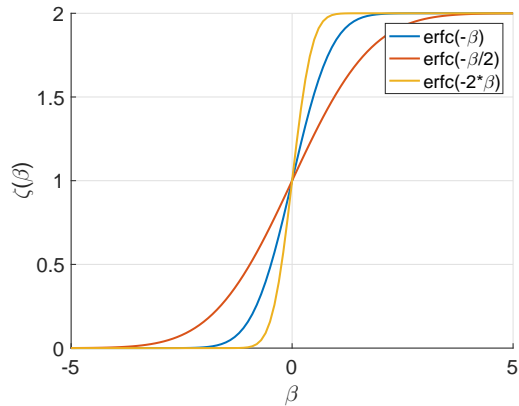}
\caption{Plot of different choices of smoothly defined $\zeta=\zeta(\beta)$.}
\label{fig1} 
\end{figure}
\begin{figure}[htb]
\centering
\subfigure[$\zeta(x)=\erfc(-\frac{x}{2})$]{
\includegraphics[width=4.3cm]{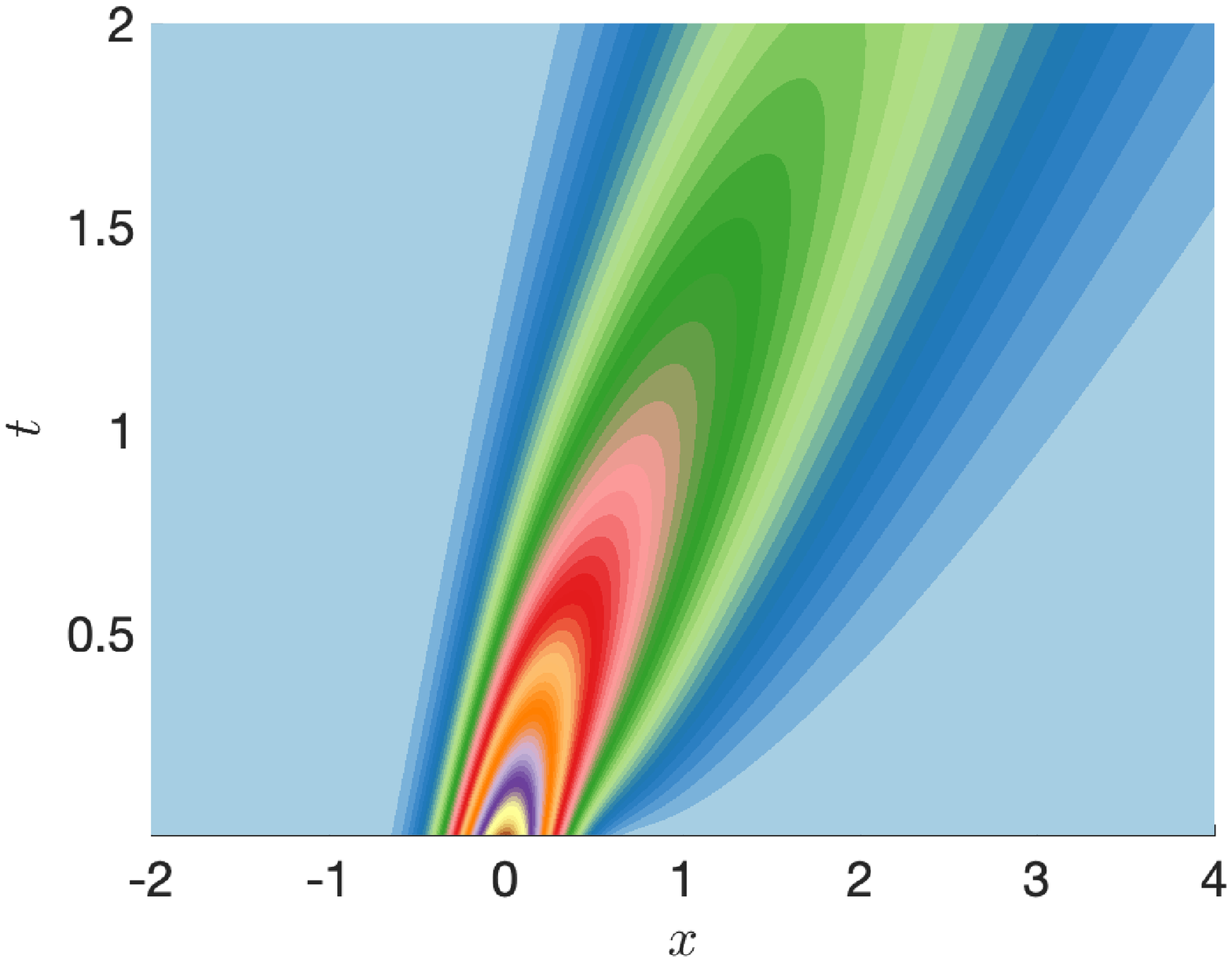}}
\subfigure[$\zeta(x)=\erfc(-x)$]{
\includegraphics[width=4.3cm]{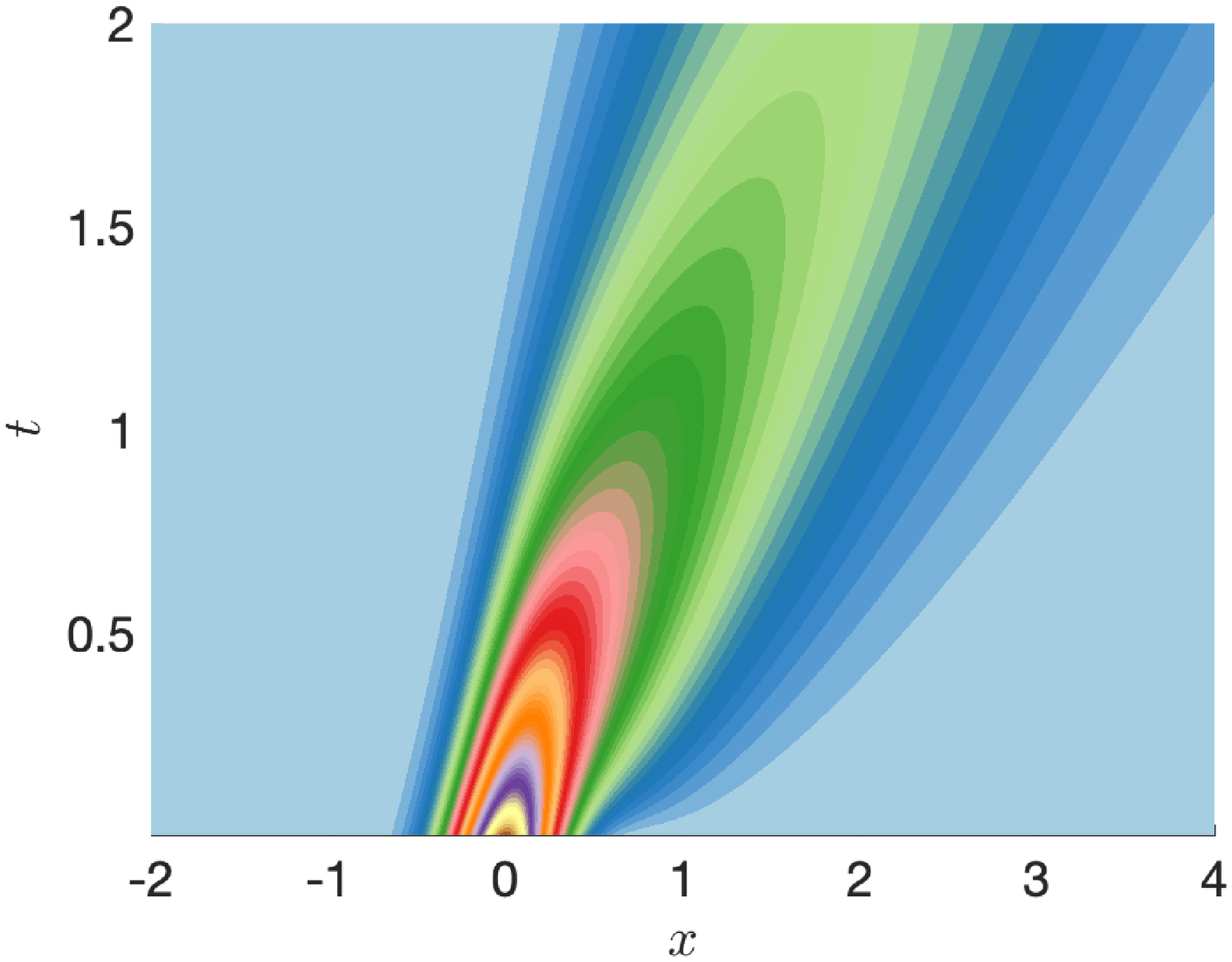}}
\subfigure[$\zeta(x)=\erfc(-2x)$]{
\includegraphics[width=4.3cm]{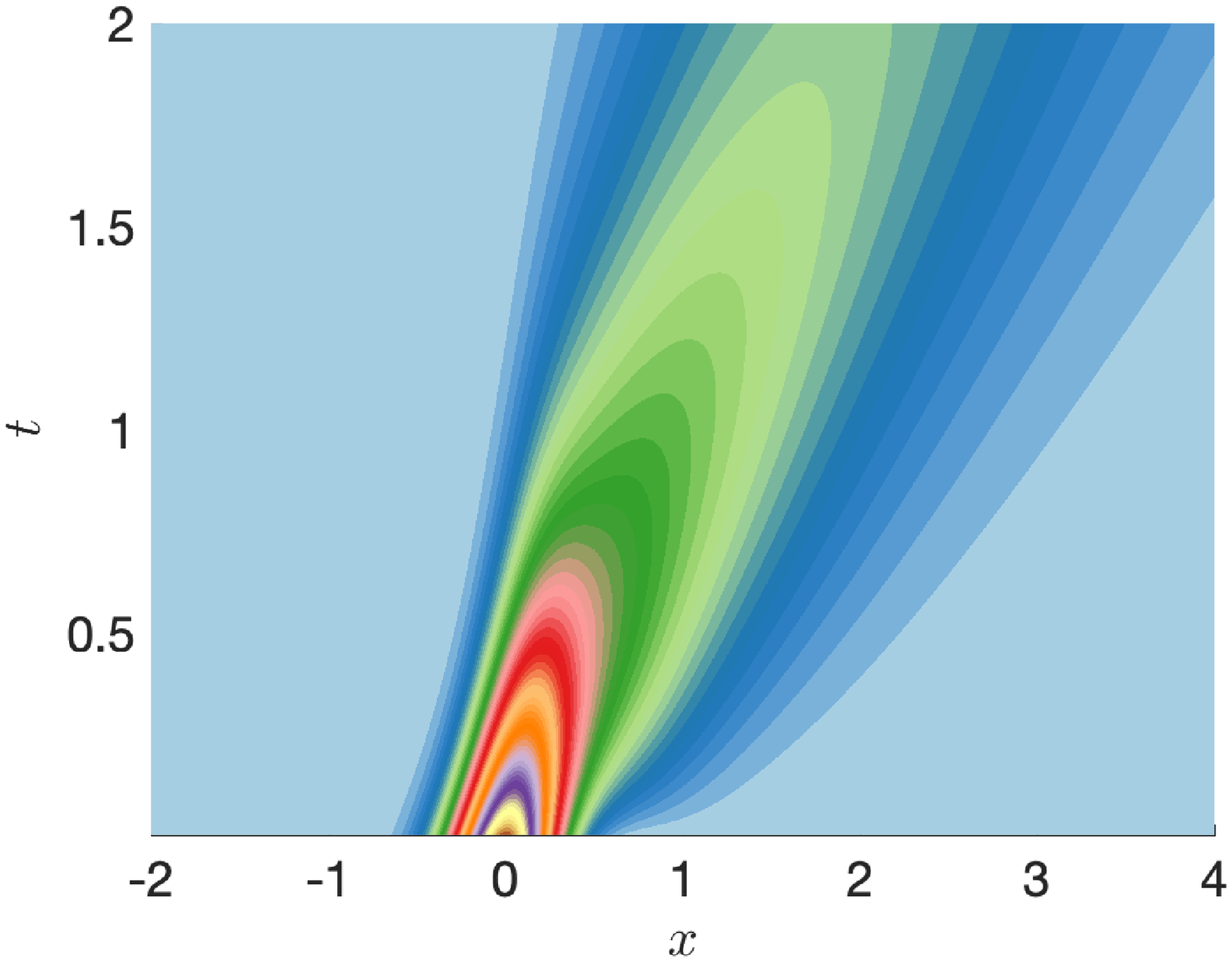}}
\caption{Wave propagation corresponding to different choices of  smooth $\zeta=\zeta(x)$ with a smooth initial data: top view.}
\label{fig2} 
\end{figure}

\begin{figure}[htb]
\centering
\subfigure[$\zeta(x)=\erfc(-\frac{x}{2})$]{
\includegraphics[width=4.3cm]{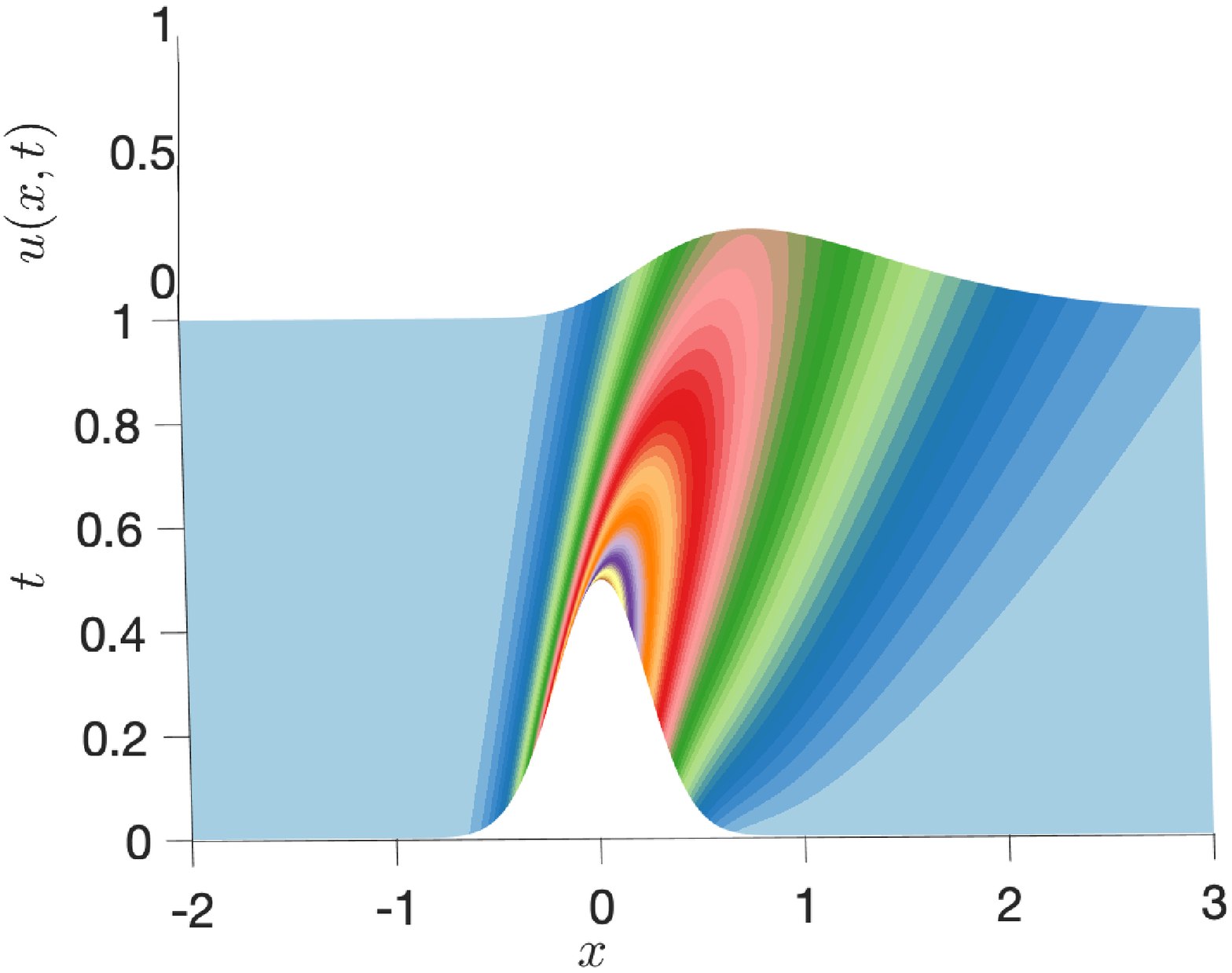}}
\subfigure[$\zeta(x)=\erfc(-x)$]{
\includegraphics[width=4.3cm]{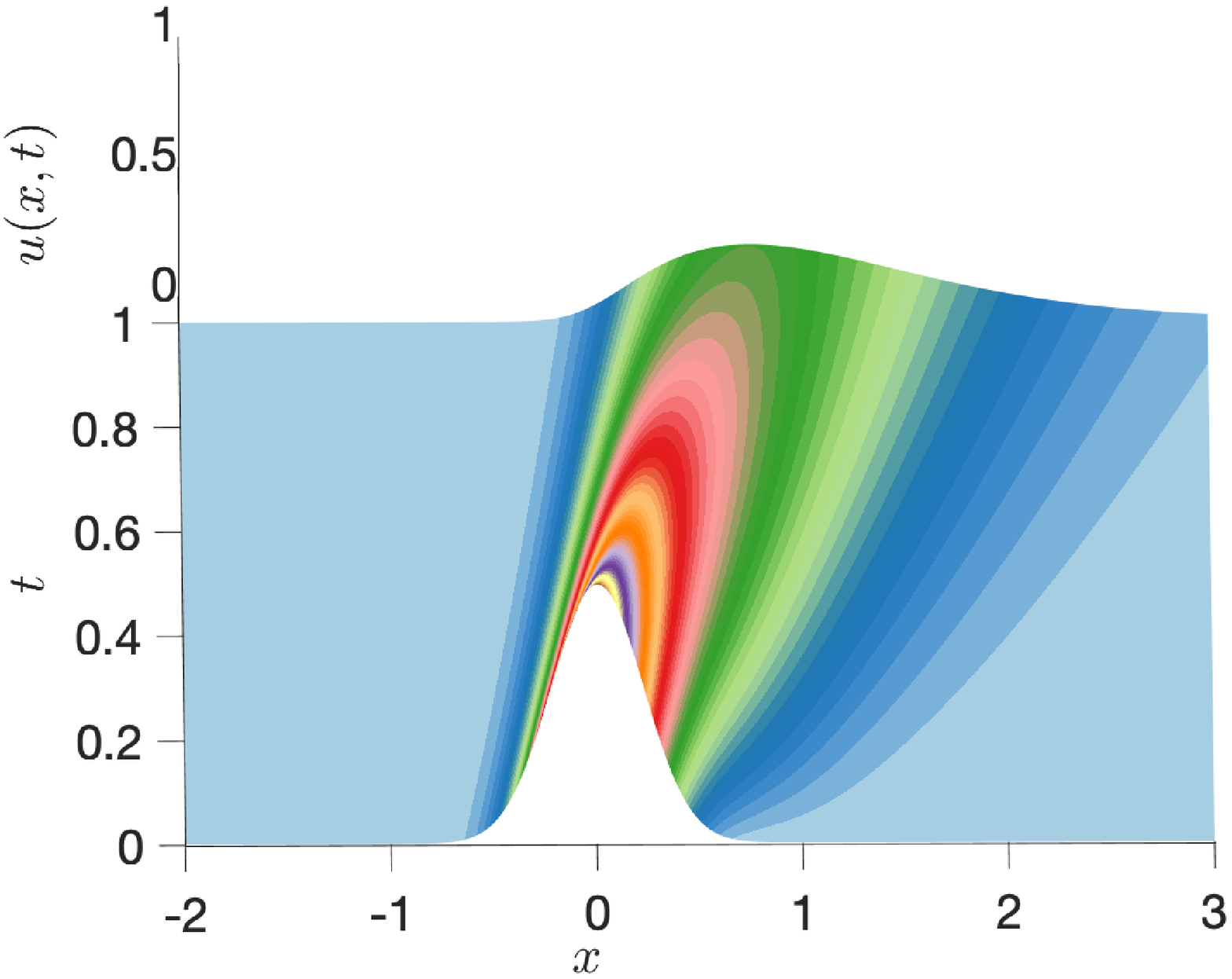}}
\subfigure[$\zeta(x)=\erfc(-2x)$]{
\includegraphics[width=4.3cm]{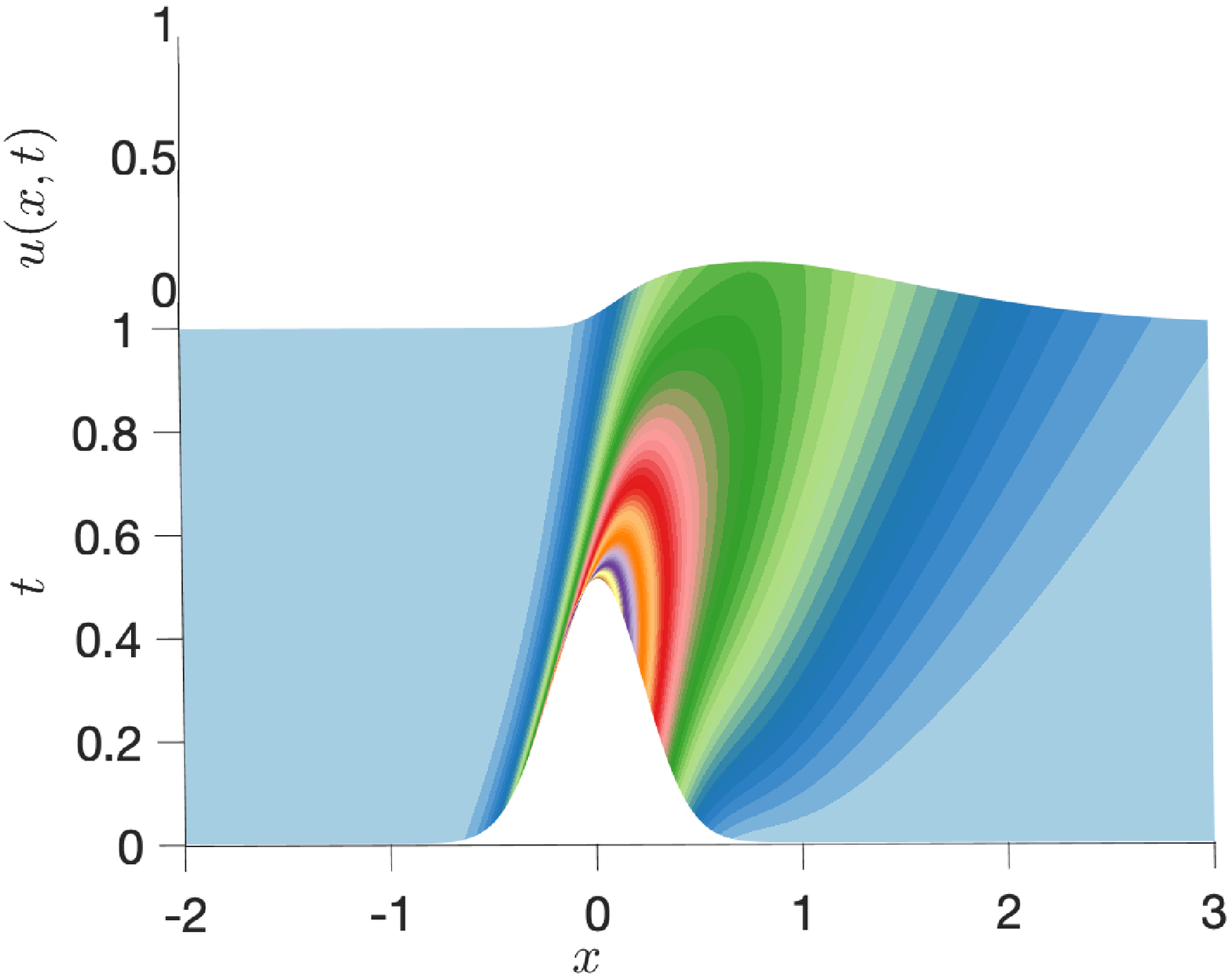}}
\caption{Wave propagation corresponding to different choices of  smooth $\zeta=\zeta(x)$ with a smooth initial data: zoom-in 3D view.}
\label{fig3} 
\end{figure}

\subsubsection{$\psi_0$ discontinuous, $\zeta$ smooth}
Now, we consider the case where horizon parameter  function $\zeta$ is smooth but the  initial data $\psi_0$ is discontinuous, corresponding to the 
situation given in theorem \ref{thm:jump-u}. We can use the same numerical method as equation \eqref{eq1} to simulate the numerical solution of $[u]$.

Here we choose
$$H(s) = 20 e^{-10s^2}, \quad \zeta(x)=\erfc(-x)$$
and
$$\psi_0(x)=\left\{\begin{array}{cc}
\frac{1}{p}, & -p<x<0 \\
-\frac{1}{p}, & 0 \leqslant x<p \\
0, & \text { otherwise, }
\end{array}\right.$$
for some given parameter $p>0$.

Even though the simulations are carried out on a much larger spatial domain to ensure that the accuracy of the solutions are maintained with the domain truncation,  we only plot part of the solutions for $x\in (-2, 4)$ and  $t\in (0, 2)$ or  $t\in (0, 1)$.  Likewise, we also carried out
 simulations for various values of $p$ but we present the case with $p=1$ mostly as the representative case.
As comparisons, we also plot the solutions corresponding to the nonlocal model with a constant horizon parameter $\zeta(x)\equiv 0.1$
and the local limit (with $\zeta(x)\equiv 0$ effectively). 

Figure \ref{fig17} shows the top view of the numerical simulation results of  $u$ for $p=1.0$ 
obtained using the numerical scheme \eqref{sc2}
 for  $t\in (0, 2)$, while Figure \ref{fig18} gives the 3D view of the solutions for $t\in (0,1)$.
 The discontinuities of the initial data are at $x=0$ and $x=\pm 1$.
 For the nonlocal horizon function $\zeta(x)=\erfc(-x)$, we note that $\zeta(-1)$ is around $0.1573$, $\zeta(0)=1$ while $\zeta(1)$ is  about $1.8427$.
 So the nonlocal effects are much stronger at $x=0$ and $x=1$. Indeed, 
the presence of stationary-in-time spatial discontinuities of the solutions can be captured by 
  abrupt and sharp changes in color densities  along vertical lines in the plots, see  in particular Figure \ref{fig17} (a) at $x=0$ and $x=1$ for most visible results.  Since the jumps decay exponentially in time, the discontinuities eventually become unnoticeable.
  Nonlocal effects are less evident for smaller horizon function values, so that the stationary-in-time discontinuities would only be visible for a very short period of time, as one can see for $x=-1$ in Figure \ref{fig17} (a). This can also be observed for all spatial positions in Figure \ref{fig17} (b). 
On the other hand, for the local model, the discontinuities of the solutions travel along characteristic lines, and they maintain
nearly the same magnitude, except for the dissipation due to numerical viscosity. 
  These  numerical results reported for this example are  consistent with our analytical study in the previous section.

 \begin{figure}[htb]
\centering
\subfigure[$\zeta(x)=\erfc(-x)$]{
\includegraphics[width=4.3cm]{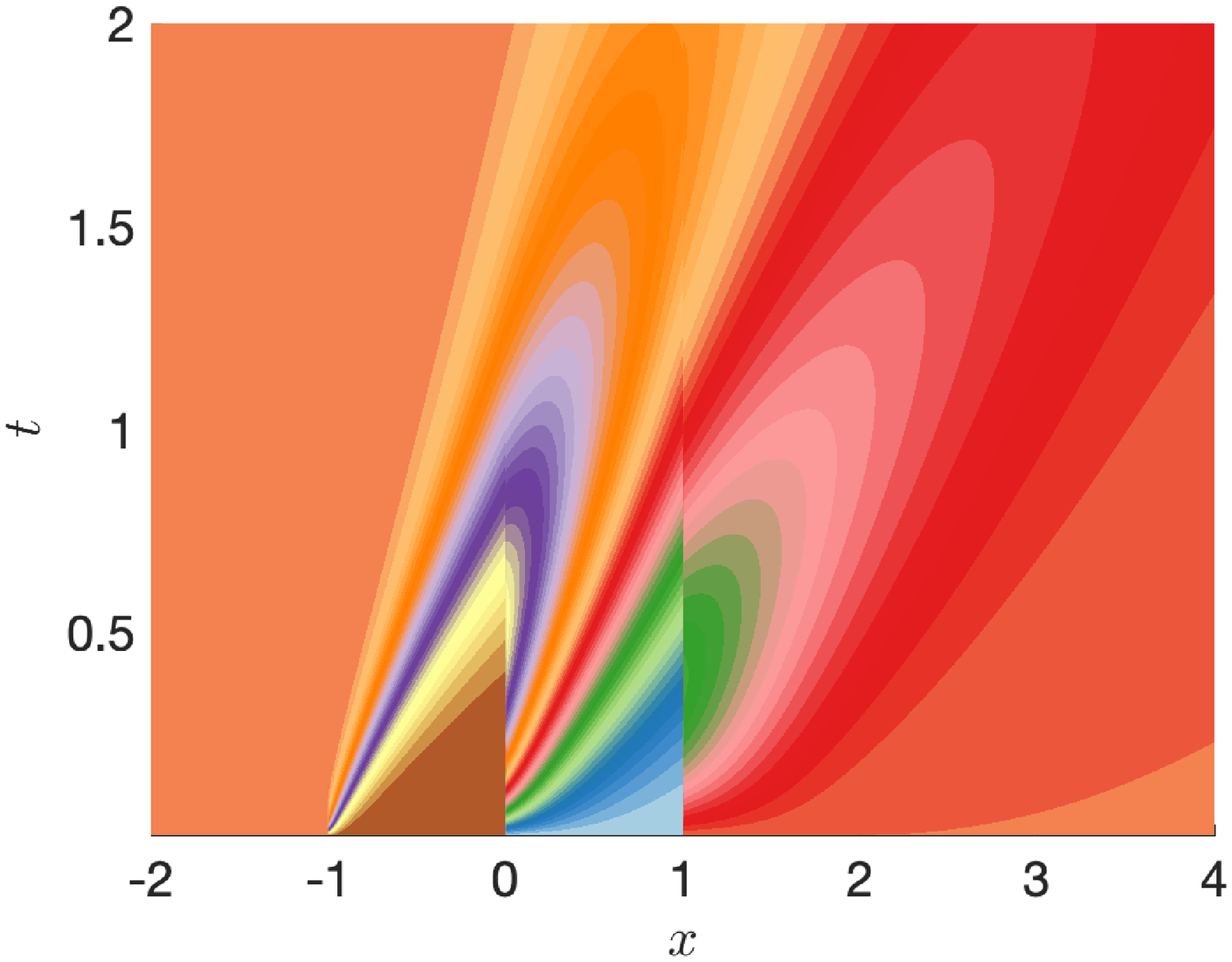}}
\subfigure[$\zeta(x)\equiv 0.1$]{
\includegraphics[width=4.3cm]{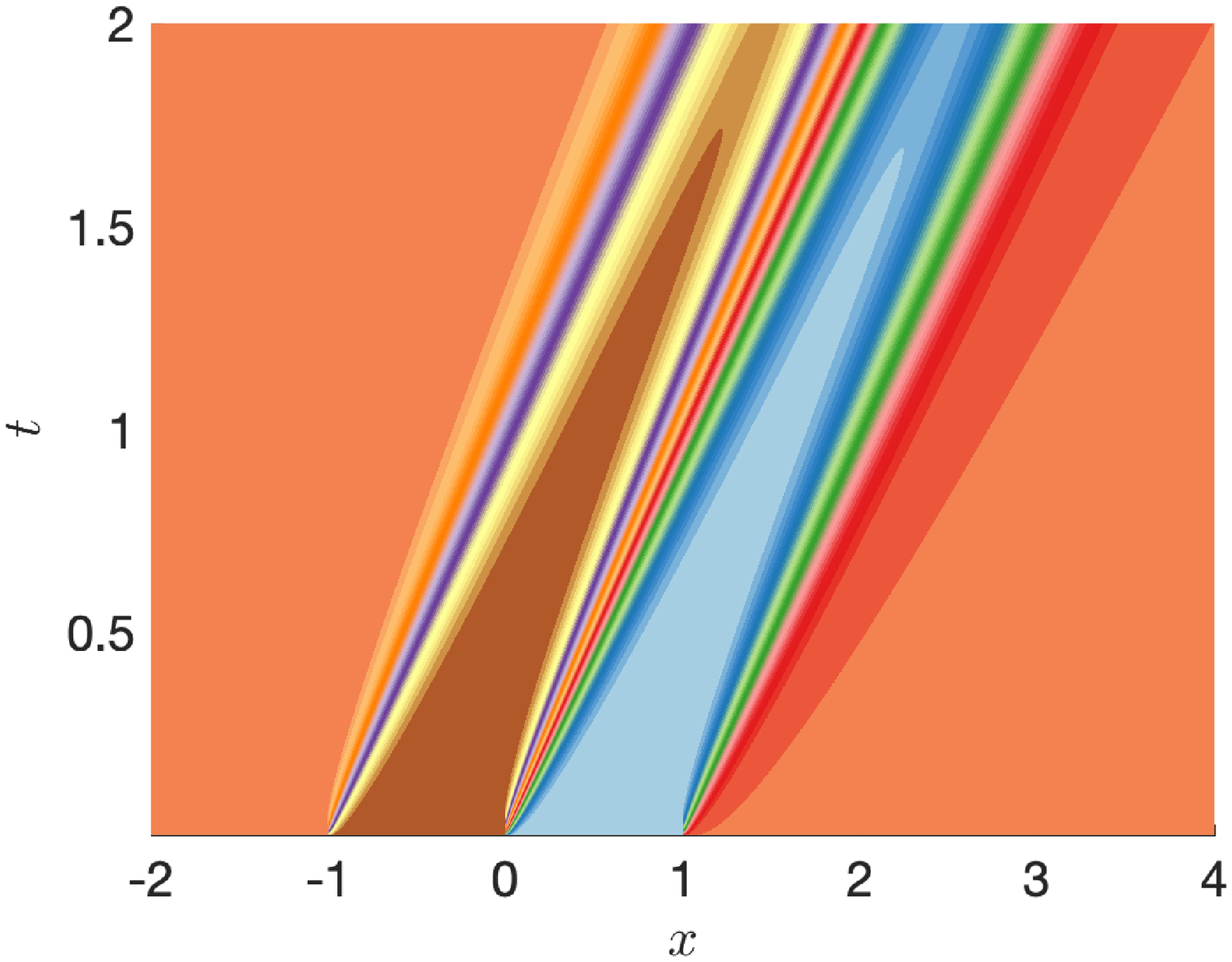}}
\subfigure[$\zeta(x)\equiv 0$]{
\includegraphics[width=4.3cm]{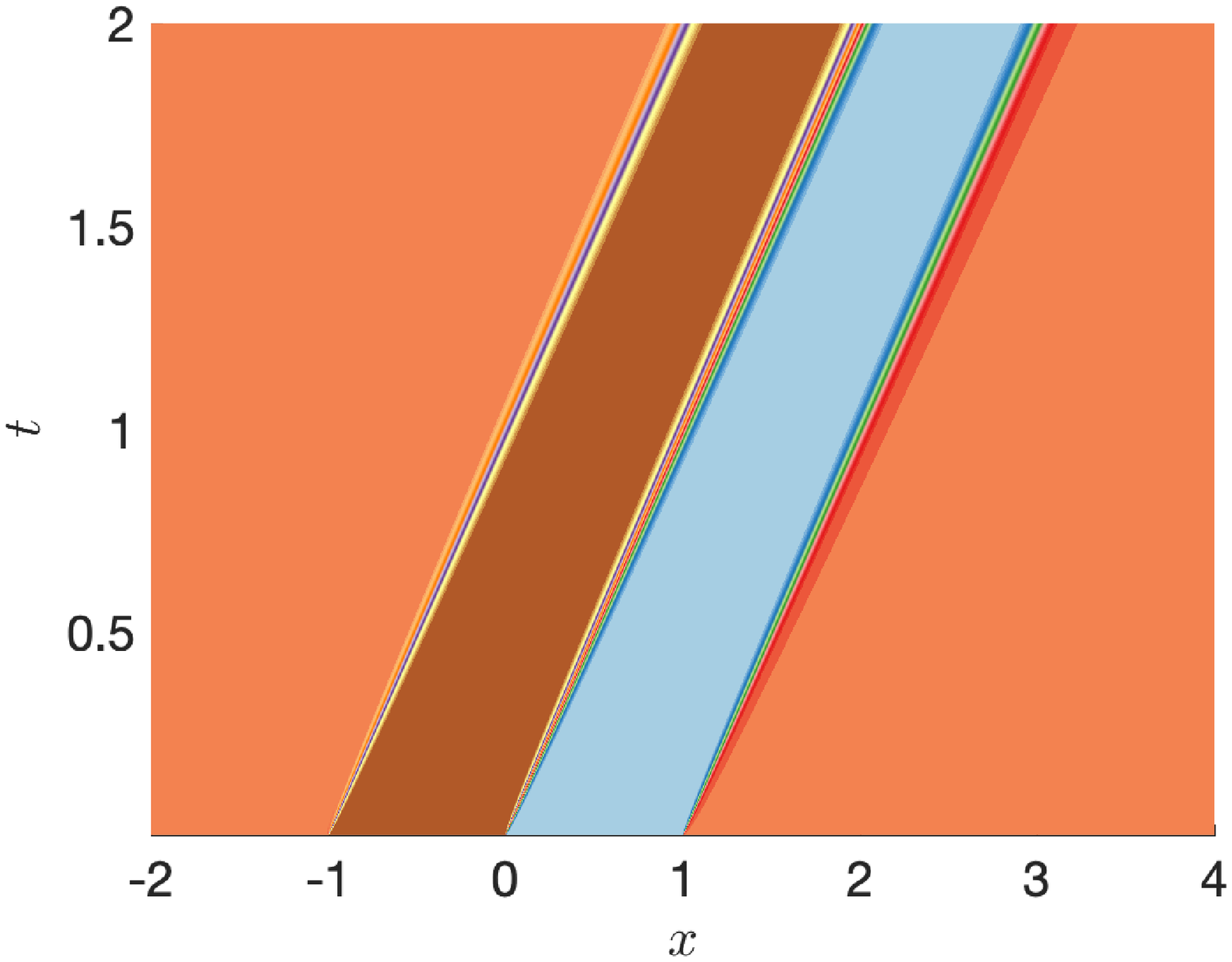}}
\caption{Wave propagation with  $\zeta(x)=\erfc(-x)$, $\zeta(x)\equiv 0.1$ and $\zeta(x)\equiv 0$ respectively: top view for $t\in (0,2)$.}
\label{fig17} 
\end{figure}

\begin{figure}[htb]
\centering
\subfigure[$\zeta(x)=\erfc(-x)$]{
\includegraphics[width=4.3cm]{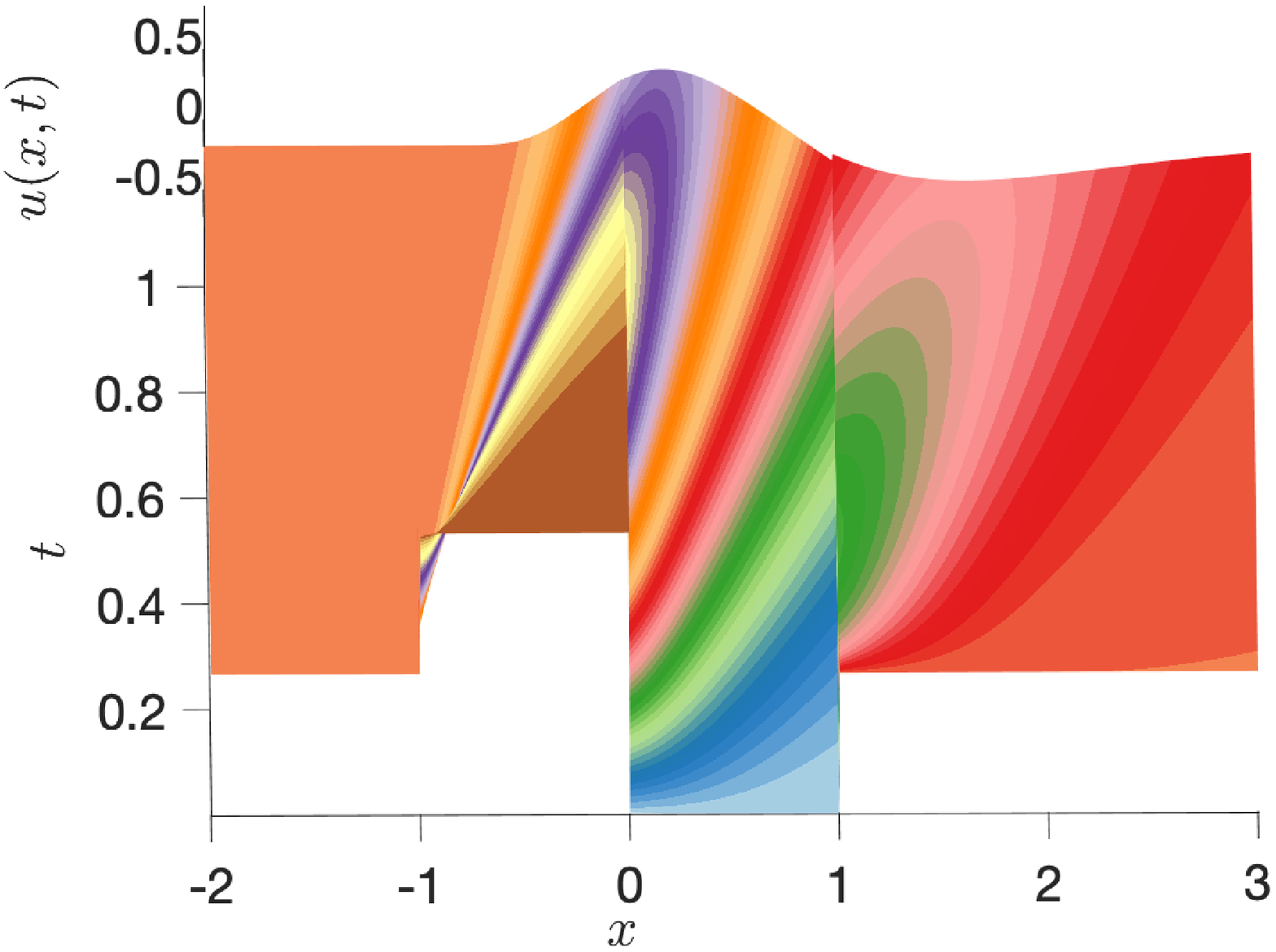}}
\subfigure[$\zeta(x)\equiv0.1$]{
\includegraphics[width=4.3cm]{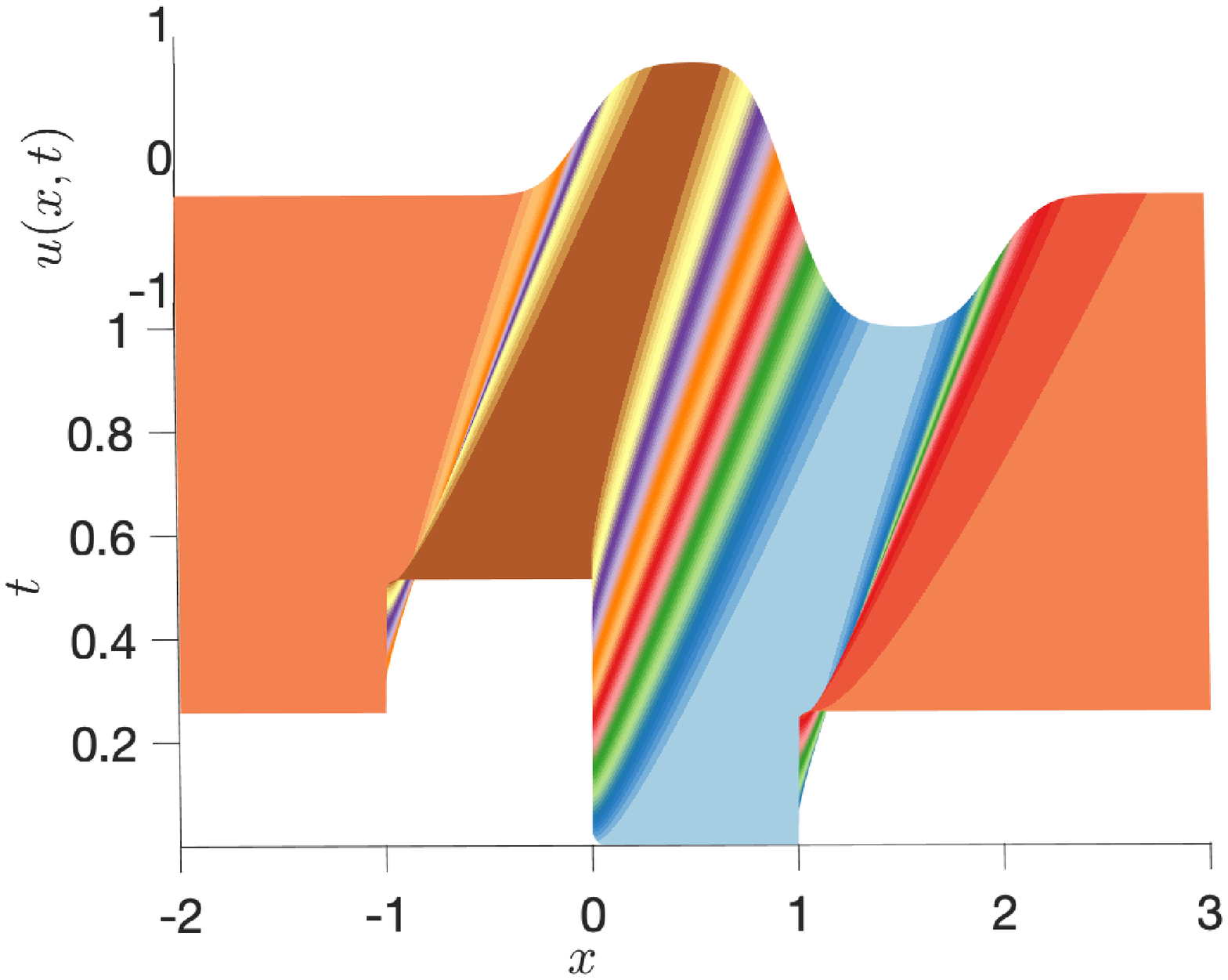}}
\subfigure[$\zeta(x)\equiv 0$]{
\includegraphics[width=4.3cm]{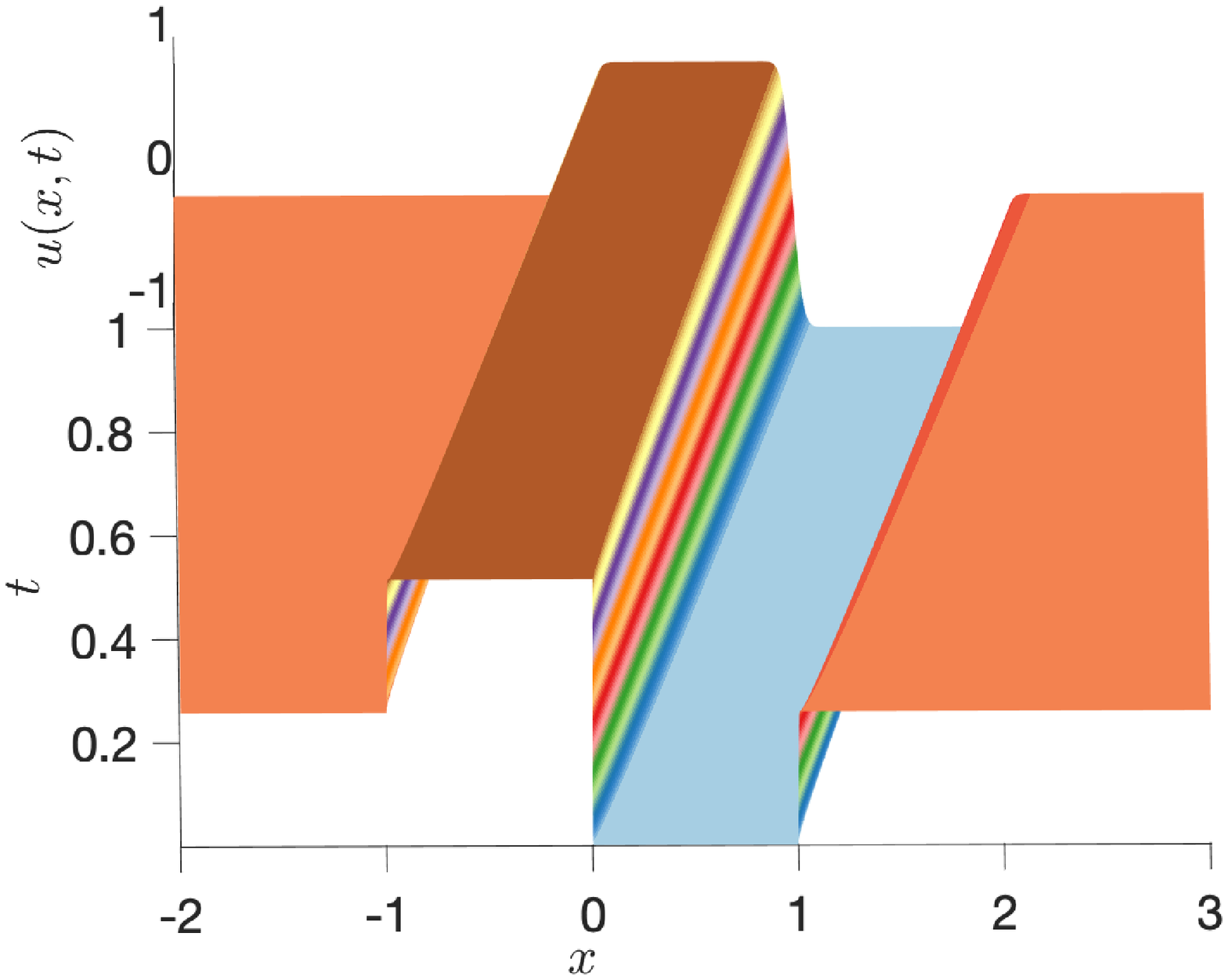}}
\caption{Wave propagation with  $\zeta(x)=\erfc(-x)$, $\zeta(x)\equiv 0.1$ and $\zeta(x)\equiv 0$ respectively: 3D view up to  $t\in (0,1)$.}
\label{fig18} 
\end{figure}

{To highlight the propagation of singularities, In Figure \ref{fig21} and \ref{fig22}, the time evolution of $ \bm{[}u\bm{]}(x,t)$  are plotted at a given positions $x$ in different time intervals $(0,t)$ for different choices of  horizon functions.  The results are obtained
by the two approaches mentioned in Section \ref{sec:4-1} are both presented to allow comparison and cross-validation.
The blue curves in Figure \ref{fig21} correspond to results obtained by the 
 numerical evaluation of \eqref{eq:jump-u1}, while the red curves are obtained from the numerical solutions of equation  \eqref{eq1}, shown in  the Figure \ref{fig22}. }
\begin{figure}[htb]
\centering
\subfigure[$\zeta(x)=\erfc(-x)$]{
\includegraphics[width=4.3cm]{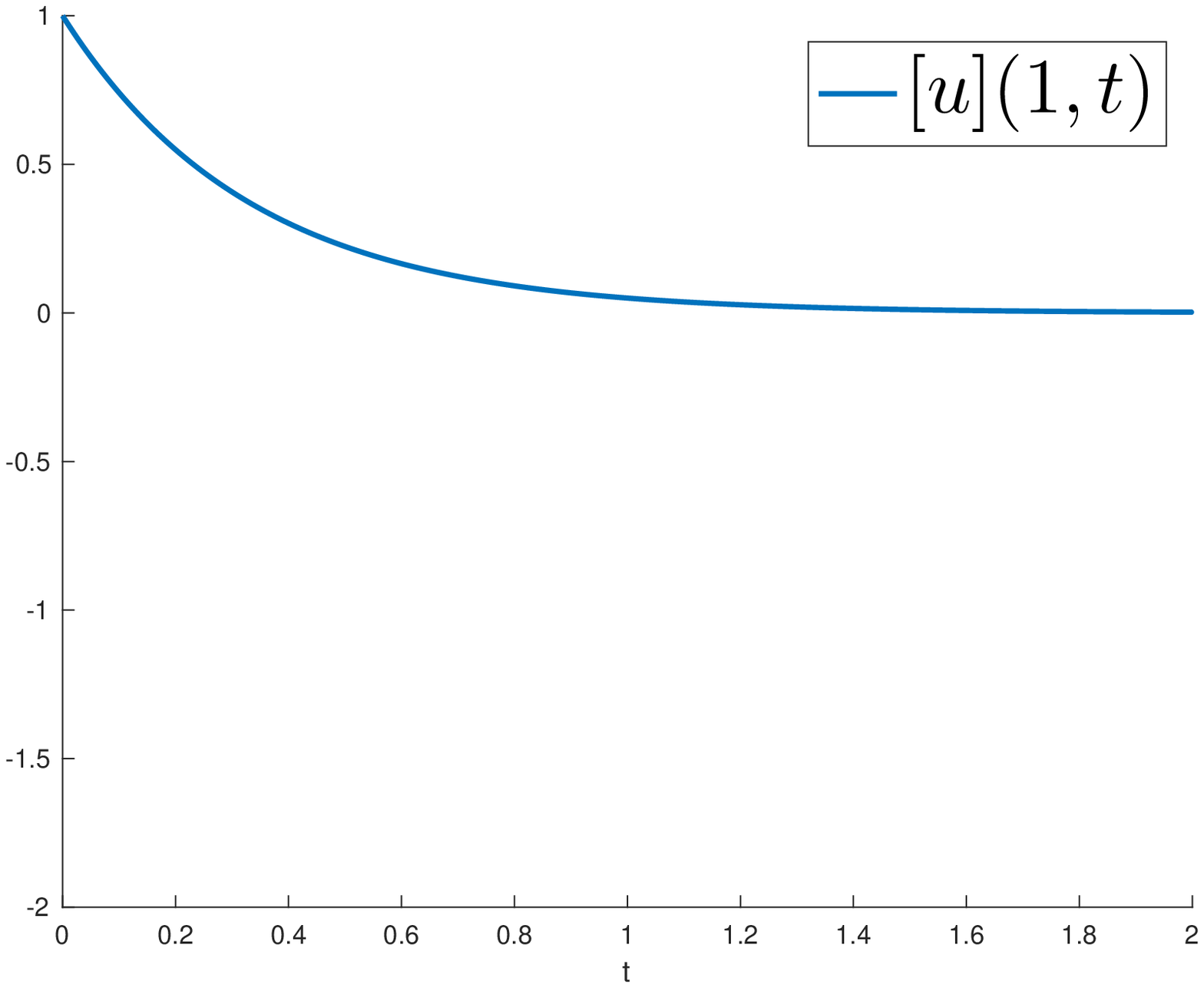}}
\subfigure[$\zeta(x)\equiv 0.1$]{
\includegraphics[width=4.3cm]{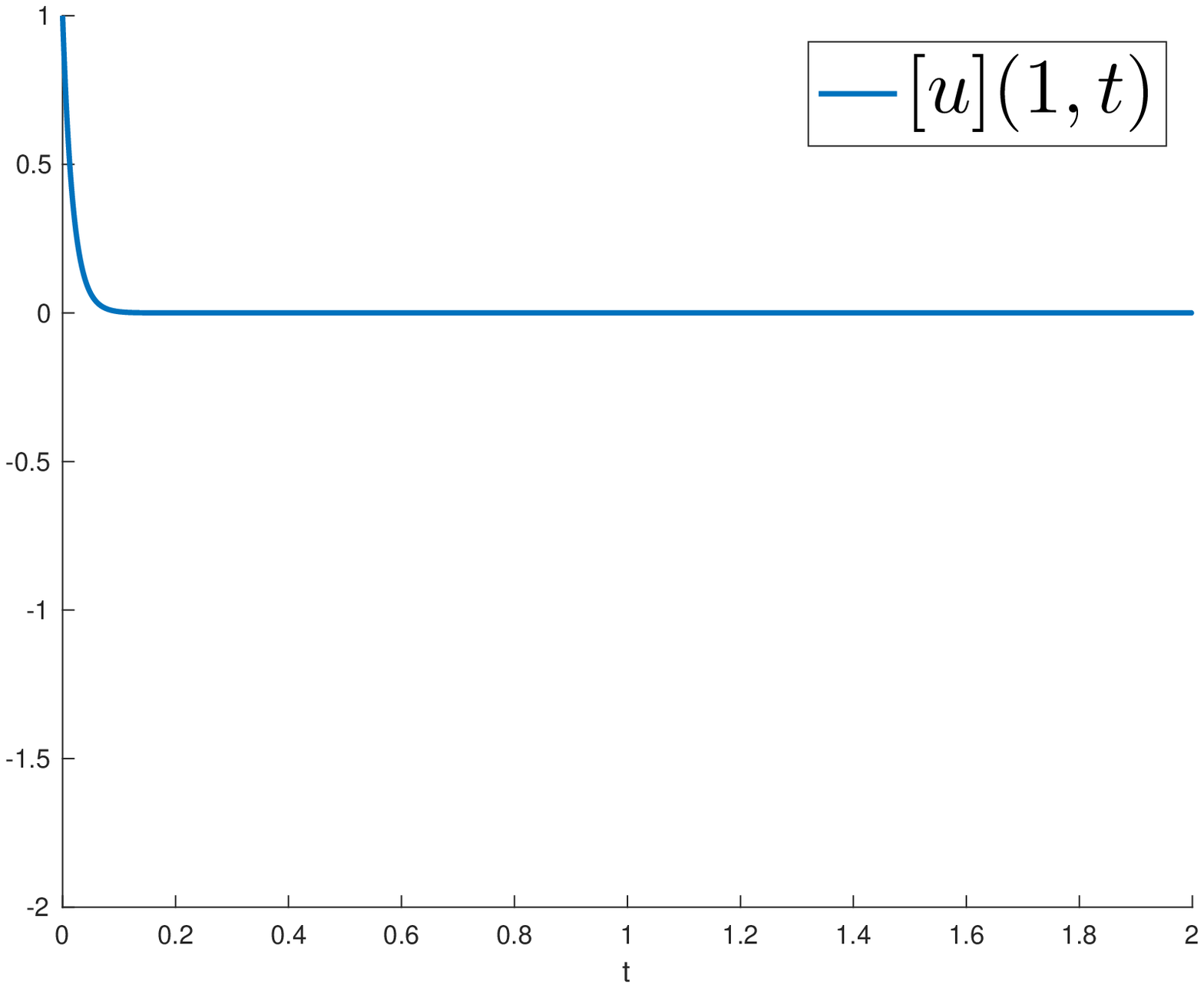}}
\subfigure[$\zeta(x)\equiv 0$]{
\includegraphics[width=4.3cm]{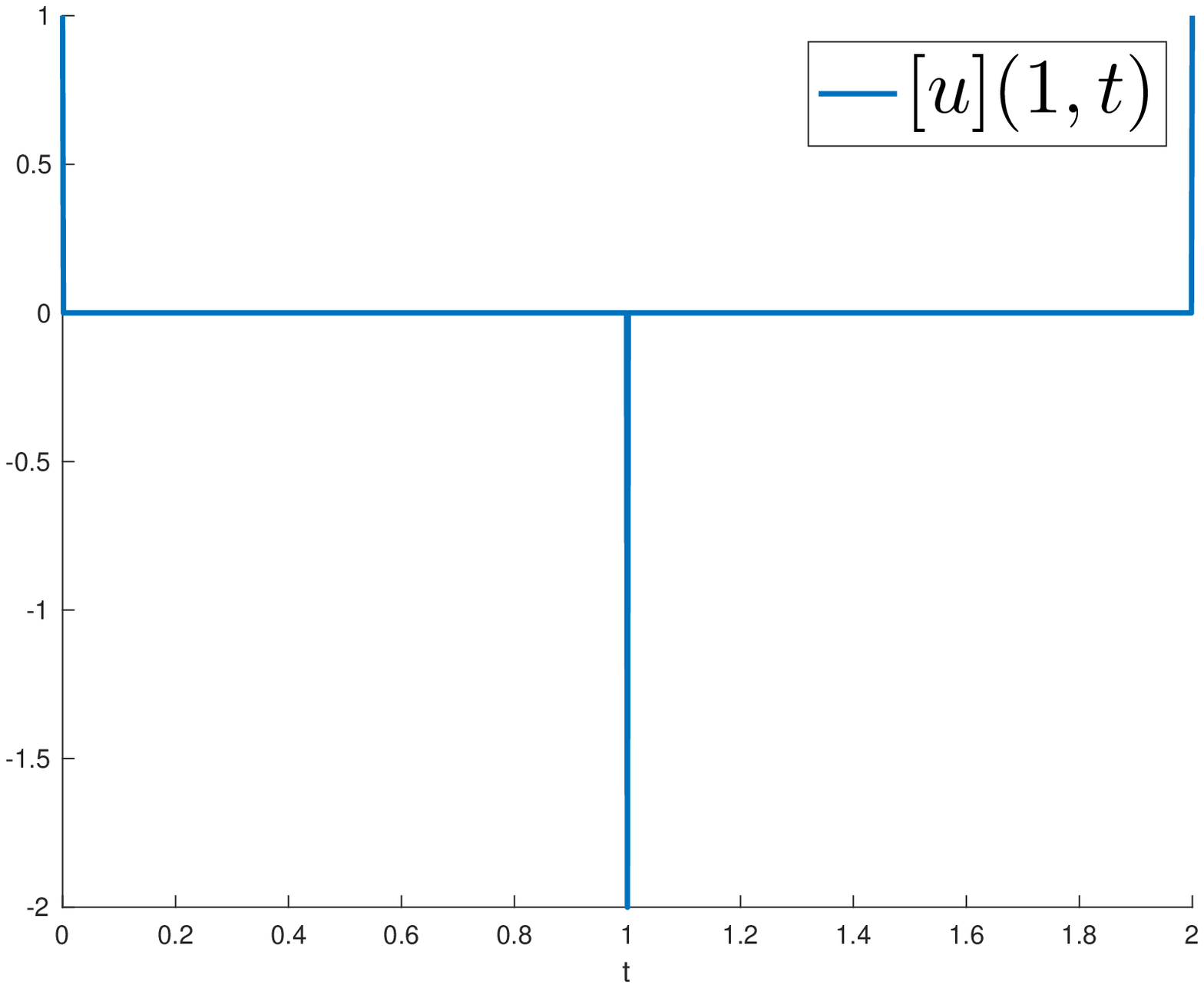}}
\caption{Plot of $ \bm{[}u\bm{]}(1,t)$  corresponding to different horizon up to $t \in (0, 2)$ by method \ref{method 2}}
\label{fig21} 
\end{figure}
\begin{figure}[htb]
\centering
\subfigure[$\zeta(x)=\erfc(-x)$]{
\includegraphics[width=4.3cm]{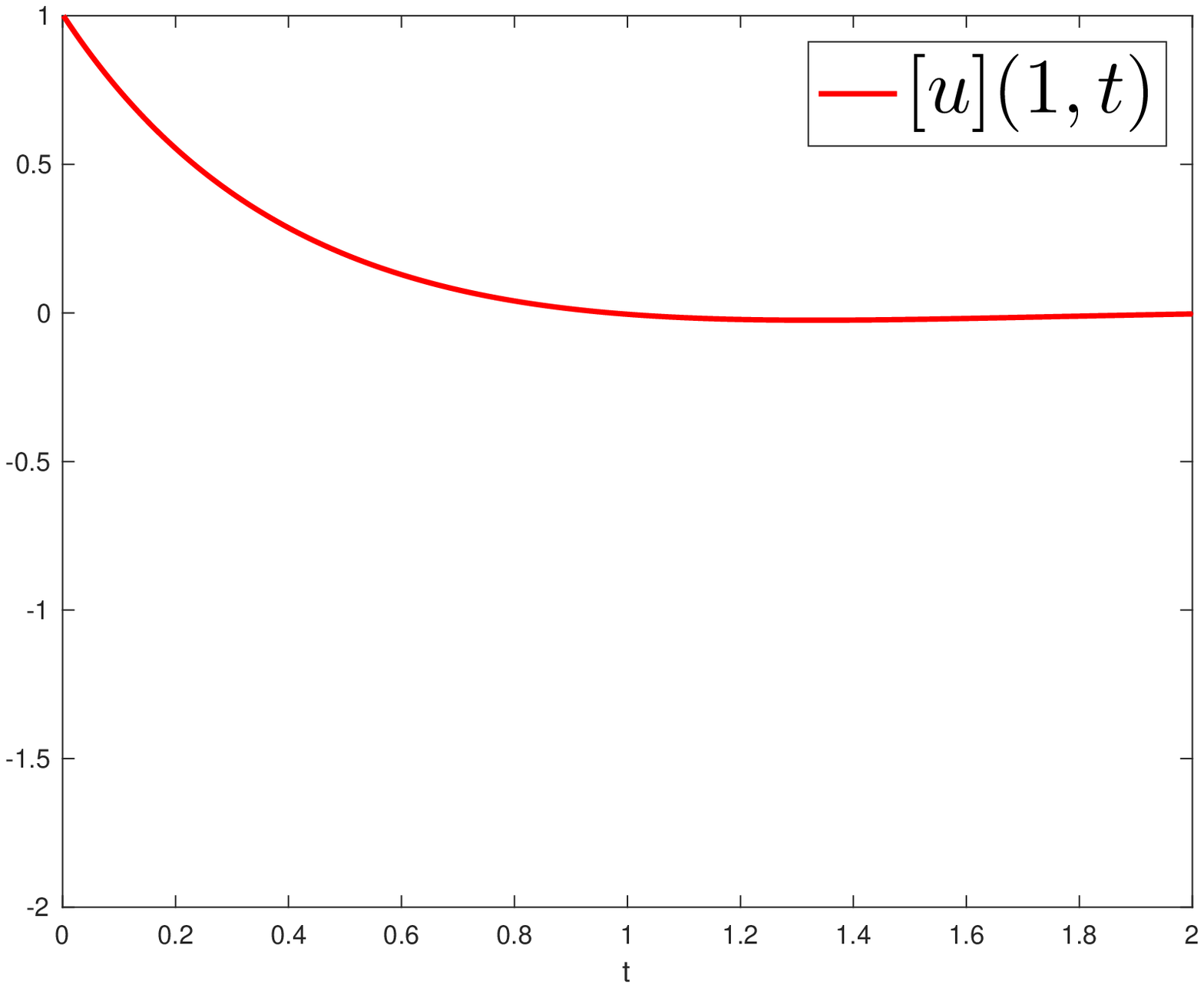}}
\subfigure[$\zeta(x)\equiv 0.1$]{
\includegraphics[width=4.3cm]{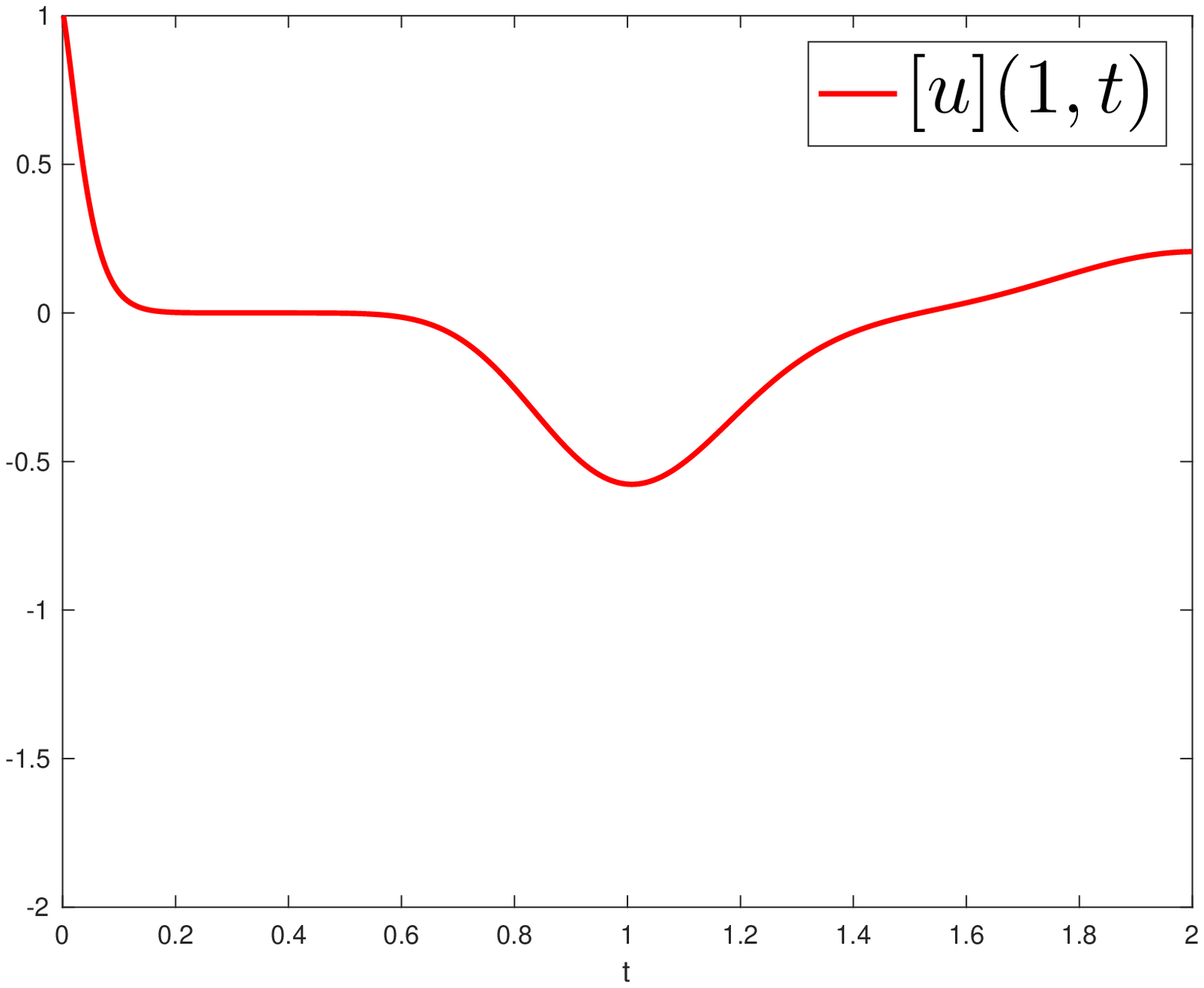}}
\subfigure[$\zeta(x)\equiv 0$]{
\includegraphics[width=4.3cm]{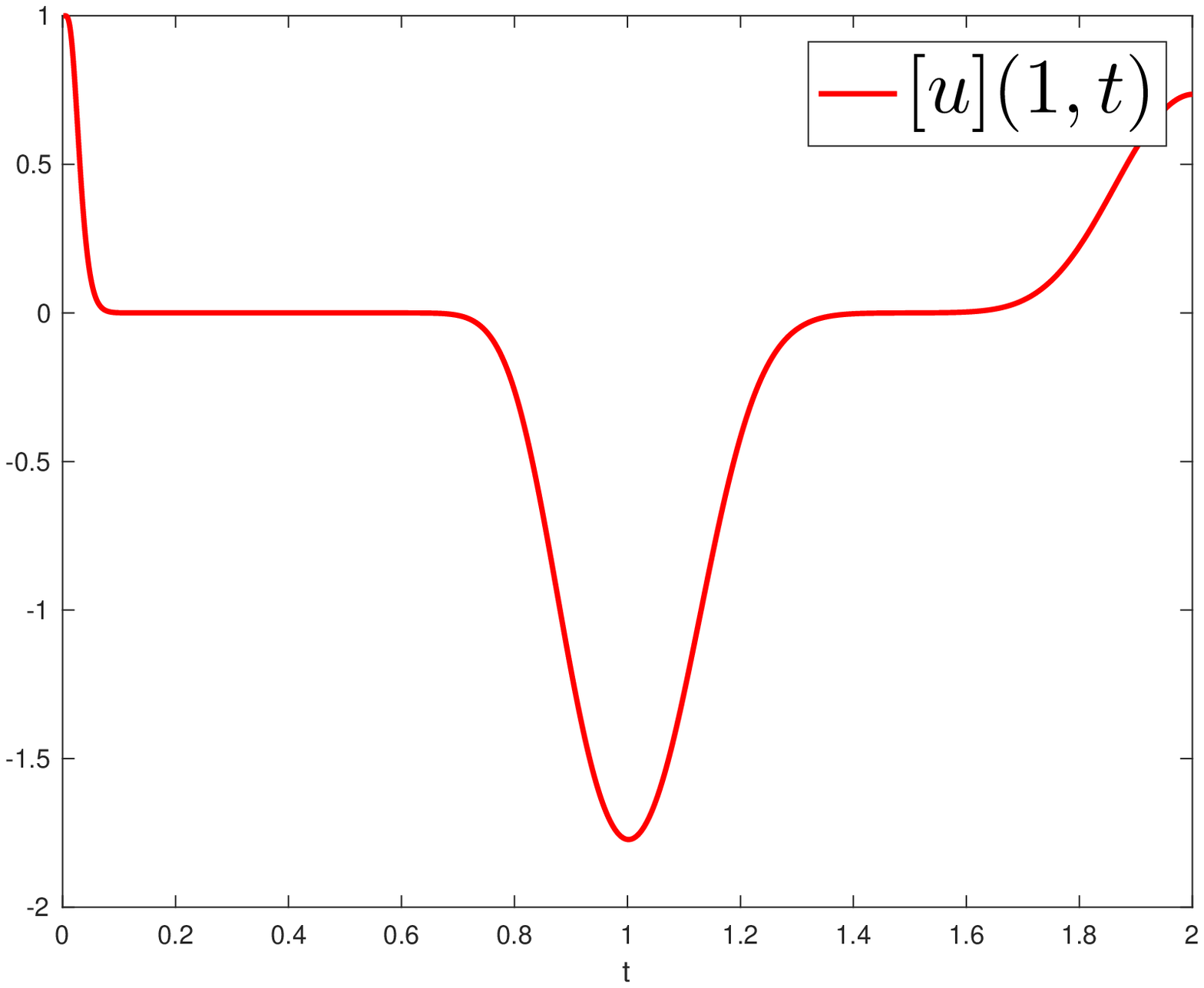}}
\caption{Plot of $ \bm{[}u\bm{]}(1,t)$ corresponding to different horizon  up to $t \in (0, 2)$ by method \ref{method 1}}
\label{fig22} 
\end{figure}

{From Figure \ref{fig21} and Figure \ref{fig22}, we can find that the results obtained by the two methods are  generally consistent with each other, up to the numerical viscous effect,  and are consistent with the conclusion in the Theorem \ref{thm:jump-u}. In the nonlocal model, corresponding to $(a)$ in Figures \ref{fig21} and  \ref{fig22},  $\bm{[}u\bm{]}$ monotonically changes with $t$ to 0. 

The plots in $(c)$ of Figures \ref{fig21} and  \ref{fig22} correspond to the local model, 
It should be noted that (c) in Figures \ref{fig21} draws the direct evaluation of $[u$], as obtained from the propagation along the characteristic line, while and the values of $[u$] in (c) of Figure \ref{fig22}  are obtained by calculating the jump given by the numerical solution, which is affected by the numerical dissipation. Still, we see that
$\bm{[}u\bm{]}$ decays rapidly to 0 at $x=1$, due to the  initial discontinuity at $x=1$  getting shifted over time, though a peak reappears later in time, due to the arrival of  the  initial discontinuity started at $x=0$ at time 0. 

In comparison, the plots  of Figures \ref{fig21}(b)  and  \ref{fig22}(b) correspond to the solution of nonlocal model with a small constant horizon parameter.
We can see the behavior of $\bm{[}u\bm{]}$ looks reasonably similar to the results of local model, because the influence of nonlocal effect is relatively weak,  even though the nonlocal model shows more pronounced dissipation as expected. 
Indeed, the results reported in Figures \ref{fig21}(b) follow the theoretical estimate and reflect the fast decay of the initial discontinuity, while the results presented in Figures \ref{fig22}(b) are based on the numerical simulations and differentiations which get influenced later on in time by the reappearance of the sharper transitions  (approximating the arrival of the initial discontinuity from $x=0$). In comparison, the results  based on the evaluation of \eqref{eq:jump-u1} is less affected by such effects of the numerical discretization.
}

\subsubsection{$\psi_0$ piecewise smooth, $\zeta$ smooth}
Now, we consider the case where horizon $\zeta$ is smooth but the  initial data $\psi_0$ is piecewise smooth, corresponding to the situation given in theorem \ref{thm:jump-ux}. 

Here we choose
$$H(s) = 20 e^{-10s^2}, \quad \zeta(x)=\erfc(-x)$$
and
\begin{align}
\label{eq:initalhat}
\psi_0(x)=\left\{\begin{array}{cc}
1-{|x|}/{p}, & -p<x<p, \\
0, & \text { else, }
\end{array}\right. \quad\text{for a given parameter}\; p>0.
\end{align}

We pick $p=0.5$ and $p=1$ for illustration.
 In Figure \ref{fig10}, we can see that the initial data $\psi_0(x)$  selected here is  a continuous piecewise linear functions with different slopes. 
 Its derivative has three discontinuities at  $x=-1, 0$ and $1$ with the one at $x=-1$ being in the local region. Thus, we mostly focus on the latter ones.

 The results of the solutions on $x\in (-2,5)$ 
are shown in Figures \ref{fig11} for $t\in (0, 2)$ with a top view
 and Figures \ref{fig12}   for $t\in (0, 1)$ (a zoomed-in version of Figure \ref{fig11}, with a 3D view)
 to illustrate the behavior of the solution.  
The  visible "vertical lines" appearing in the contours could signal potential singularities. However, as the colors on the two sides of a "vertical line"  have a 
much less dramatic transition, we can see that the lines are capturing the "folds" in the solutions, that is, the discontinuities of $u_x$, 
as we expect. We can see the appearance of these "folding lines" at $x = 0$ and $x = p$.
The initial Hat function gets smeared as time goes on, due to the dissipative effect present in the nonlocal equations. Also,  as the changes of slopes in the initial data become larger (as $p$ gets smaller, from the right plot to the left plot), the folding line shows up more dramatically. 

 \begin{figure}[htb]
\centering
{
\includegraphics[width=7.5cm]{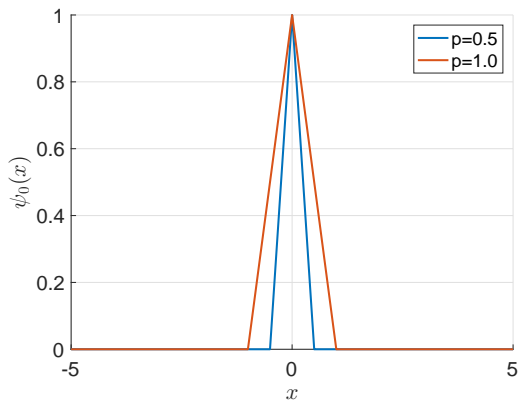}}
\caption{Plot of different choices of piecewise smoothly defined $\psi_0(x)$ for $x_0 = 0$}
\label{fig10} 
\end{figure}

\begin{figure}[htb]
\centering
\subfigure[$p = 0.5$]{
\includegraphics[width=6.5cm]{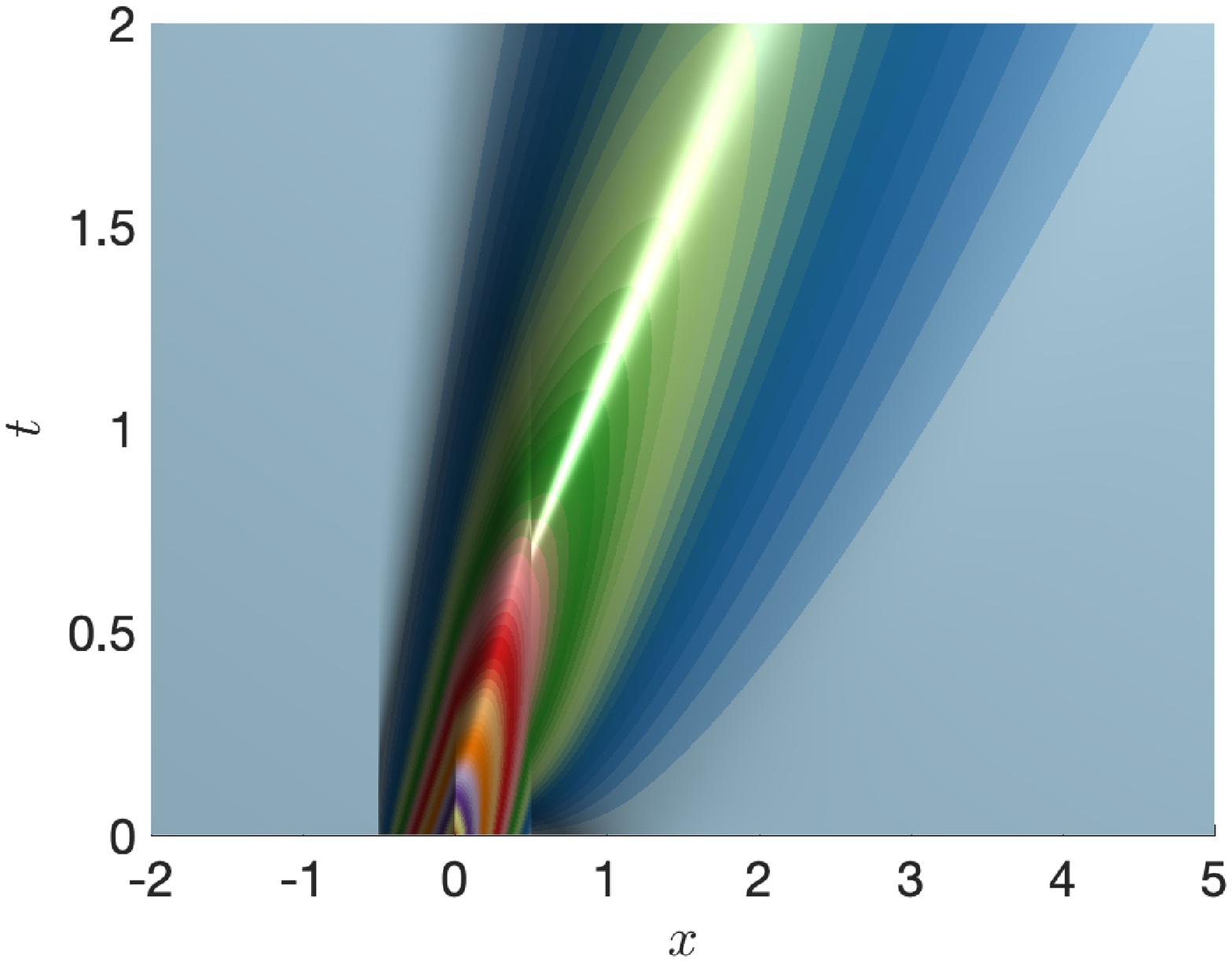}}
\subfigure[$p = 1.0$]{
\includegraphics[width=6.5cm]{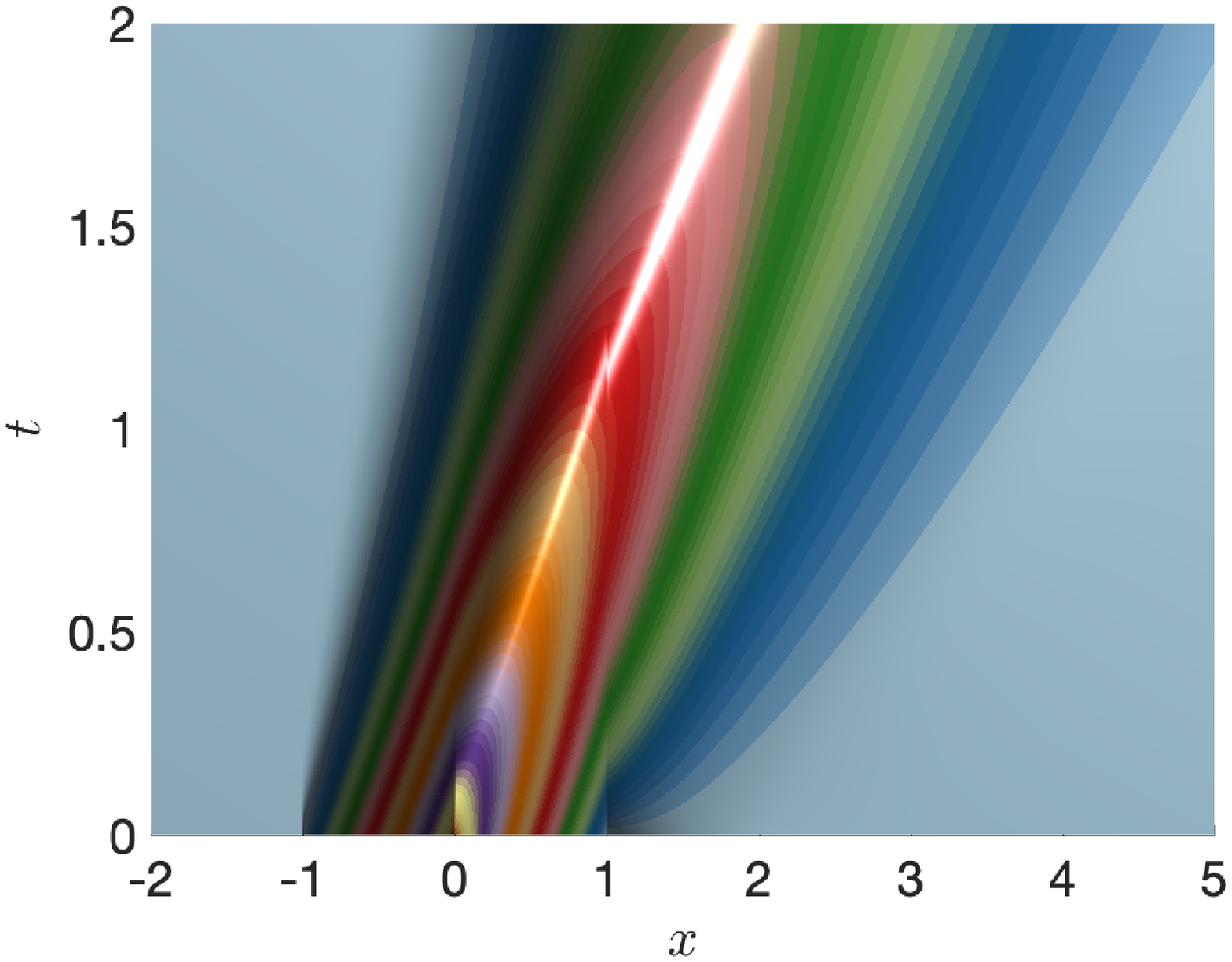}}
\caption{Wave propagation  corresponding to different choices of $p$ for piecewise smooth initial data $\psi_0(x)$ with a smooth horizon: top view}
\label{fig11} 
\end{figure}

\begin{figure}[htb]
\centering
\subfigure[$p = 0.5$]{
\includegraphics[width=6.5cm]{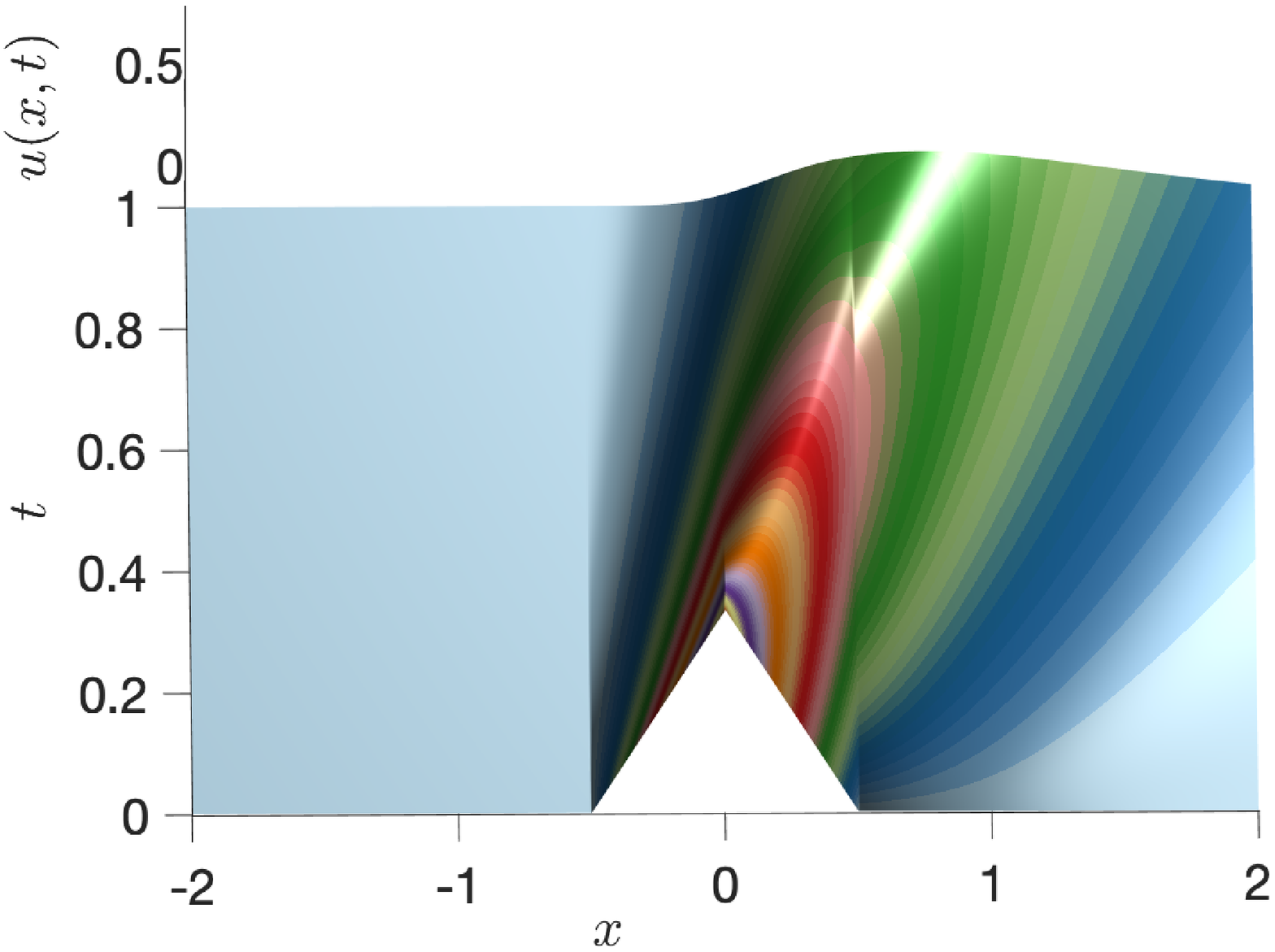}}
\subfigure[$p = 1.0$]{
\includegraphics[width=6.5cm]{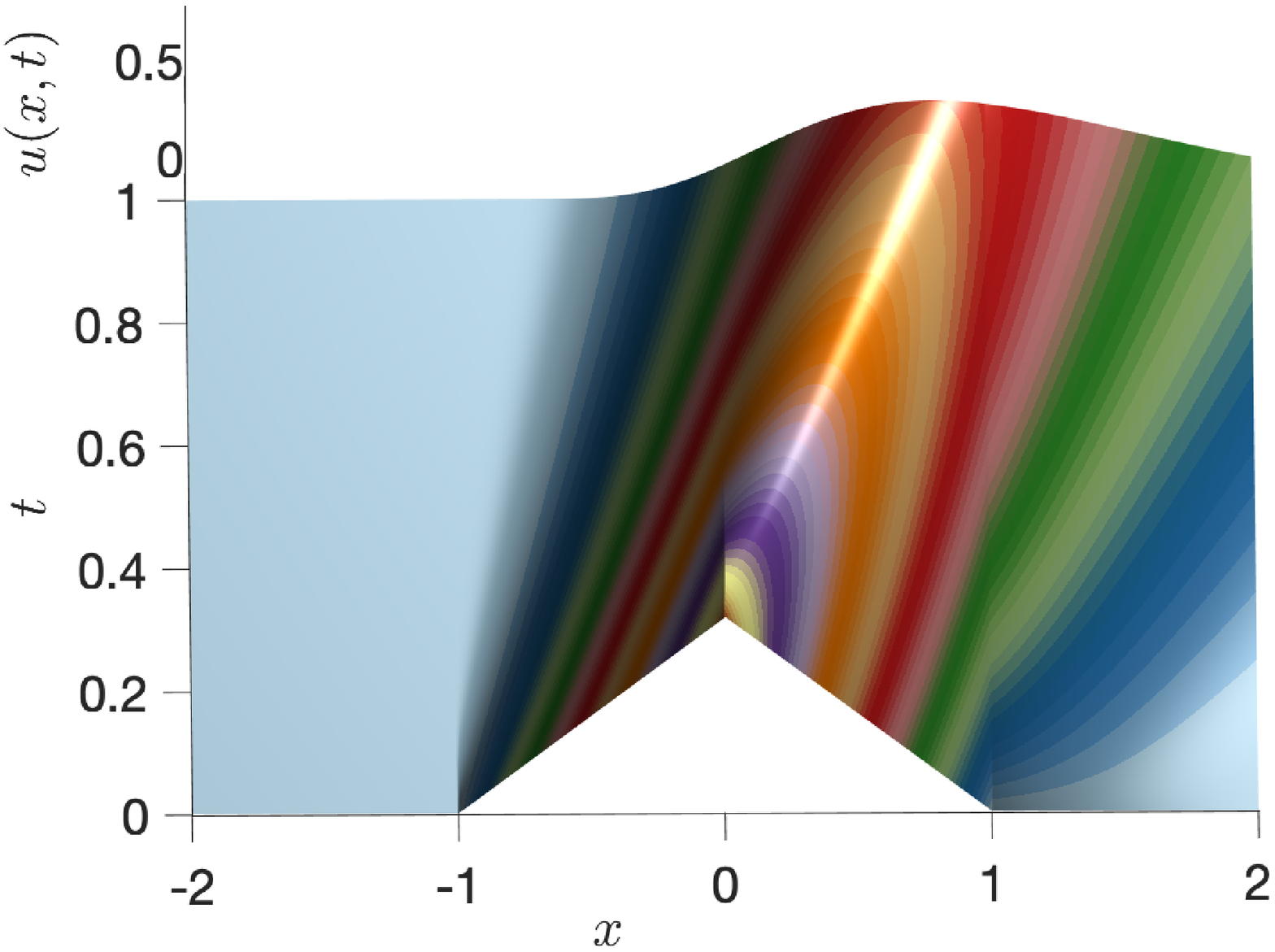}}
\caption{Wave propagation  corresponding to different choices of $p$ for piecewise smooth initial data $\psi_0(x)$ with a smooth horizon: zoom-in 3D view}
\label{fig12} 
\end{figure}

In Figures \ref{fig13} and \ref{fig14}, we plot the evolution of $ \bm{[}u_x\bm{]}(x,t)$ in time for different horizon by the two numerical methods mentioned before. Results obtained from \eqref{eq:jump-ux1} correspond to the blue curves while red curves are from the jumps of difference quotients of numerical solutions. From Figures \ref{fig13} and \ref{fig14}, we can find that $\bm{[}u_x\bm{]}$ monotonically diminishes to zero at $x = 0$ and  $x = p$ as described by equation \eqref{eq:jump-ux}, leading to consistent results obtained in the analytical estimates and numerical solutions in Figures \ref{fig13} and \ref{fig14}. Moreover, comparing the plots in $(a)$ and $(c)$, and those in $(b)$ and $(d)$, among these figures, we can find that the results obtained by the two methods are basically the same, which is consistent with our expectations and helps to validate the numerical findings.
\begin{figure}[htb]
\centering
\subfigure[$\bm{[}u_x\bm{]}(0,t)$ by method \ref{method 1}]{
\includegraphics[width=5.3cm]{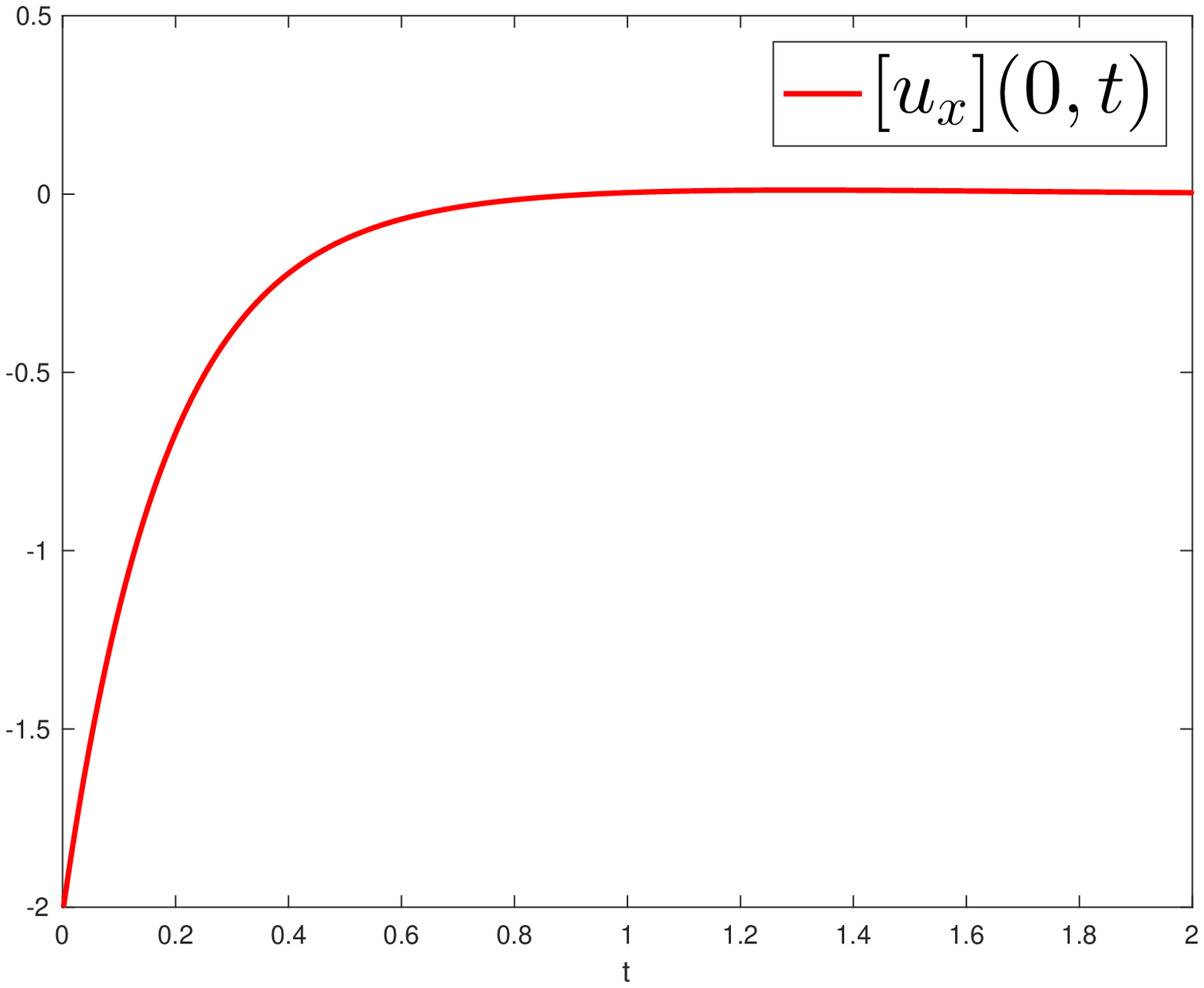}\qquad}
\subfigure[$\bm{[}u_x\bm{]}(0,t)$ by method \ref{method 2}]{
\includegraphics[width=5.3cm]{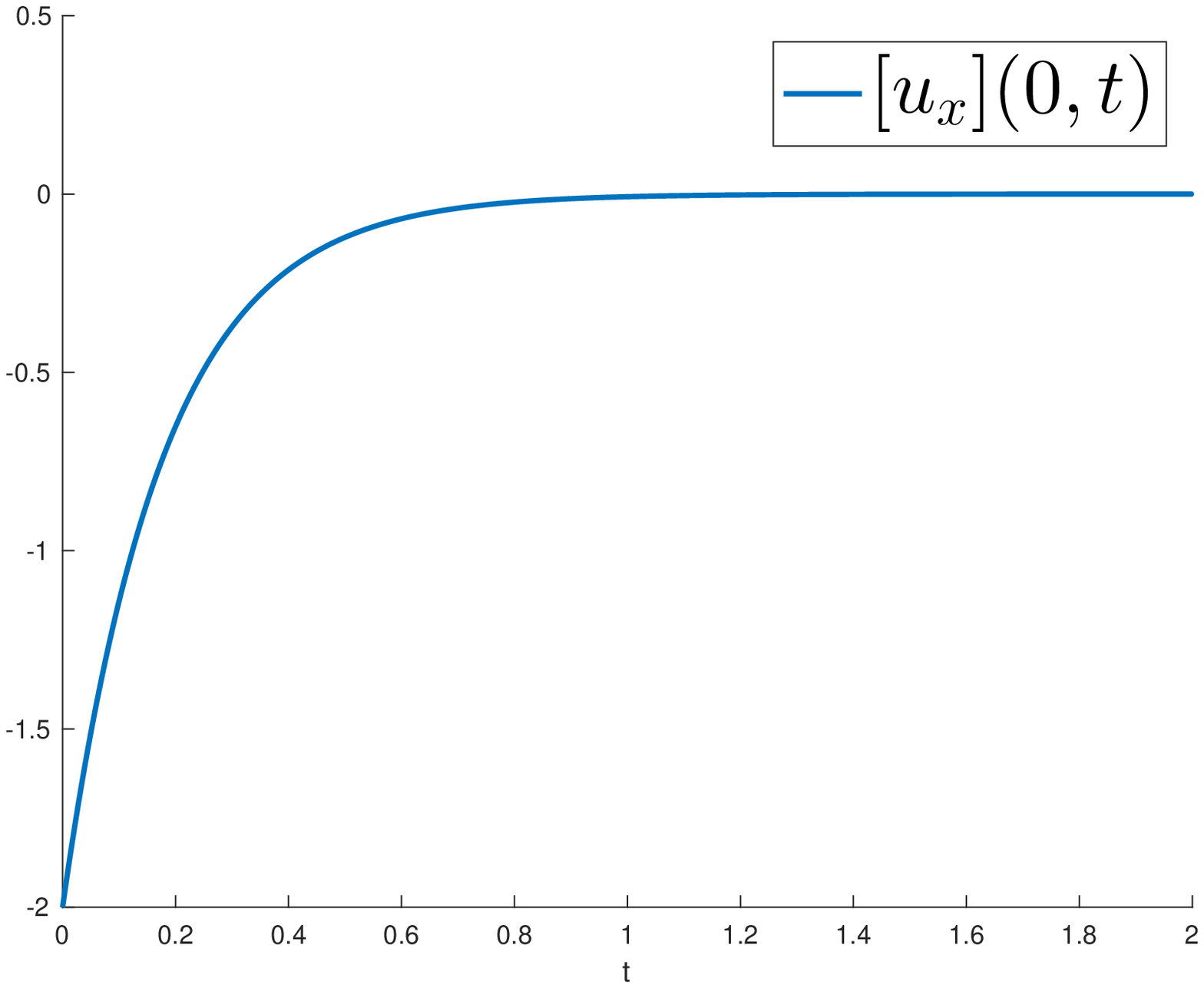}\qquad}\\
\subfigure[$\bm{[}u_x\bm{]}(1,t)$ by method \ref{method 1}]{
\includegraphics[width=5.3cm]{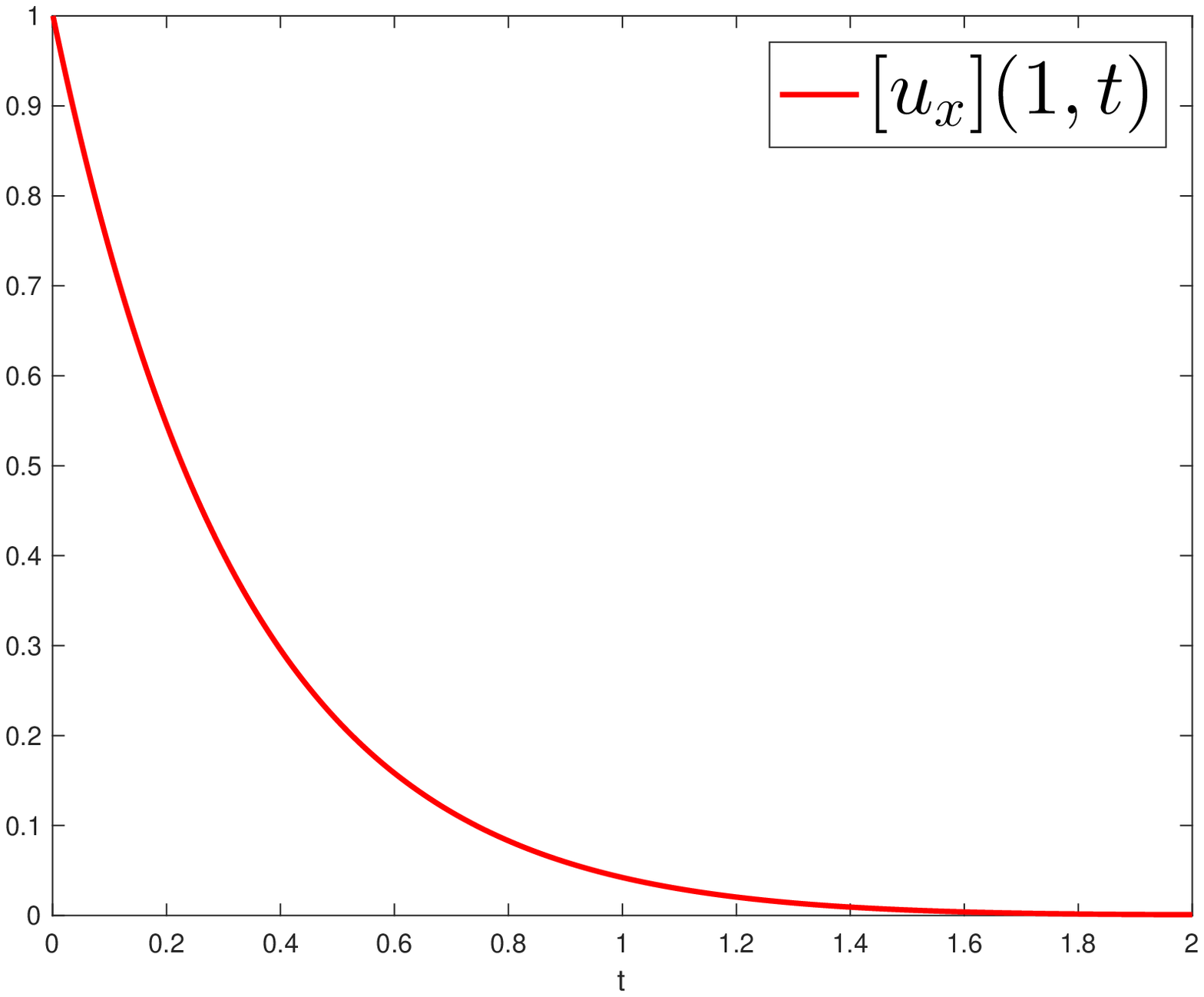}\qquad}
\subfigure[$\bm{[}u_x\bm{]}(1,t)$ by method \ref{method 2}]{
\includegraphics[width=5.3cm]{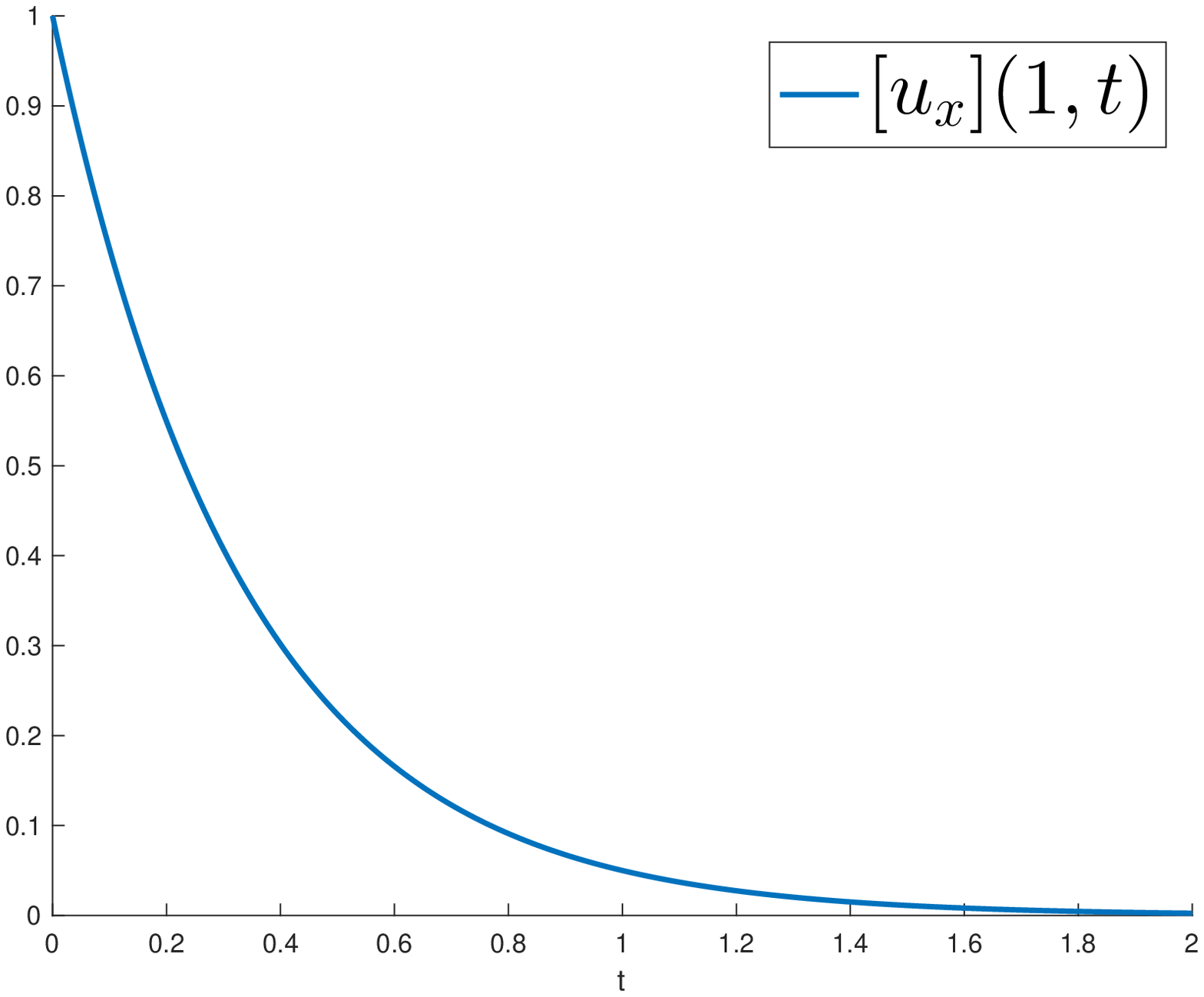}\qquad}
\caption{Plots of $ \bm{[}u_x\bm{]}(x,t)$ computed by different methods with $p = 1$ and $x = 0$ and $1$ for $t\in(0, 2)$.}
\label{fig13} 
\end{figure}

\begin{figure}[htb]
\centering
\subfigure[$\bm{[}u_x\bm{]}(0,t)$ by method \ref{method 1}]{
\includegraphics[width=5.3cm]{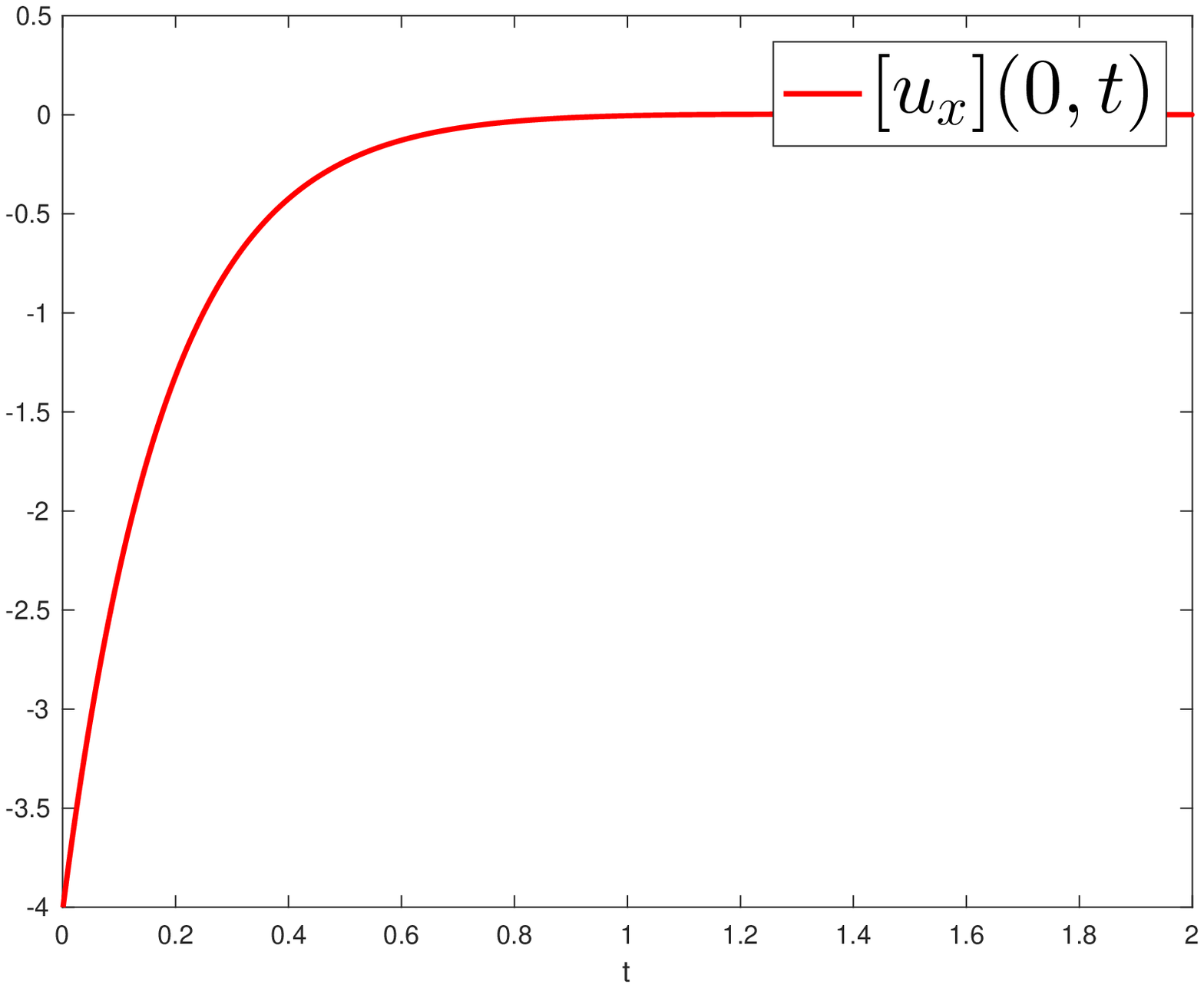}\qquad}
\subfigure[$\bm{[}u_x\bm{]}(0,t)$ by method \ref{method 2}]{
\includegraphics[width=5.3cm]{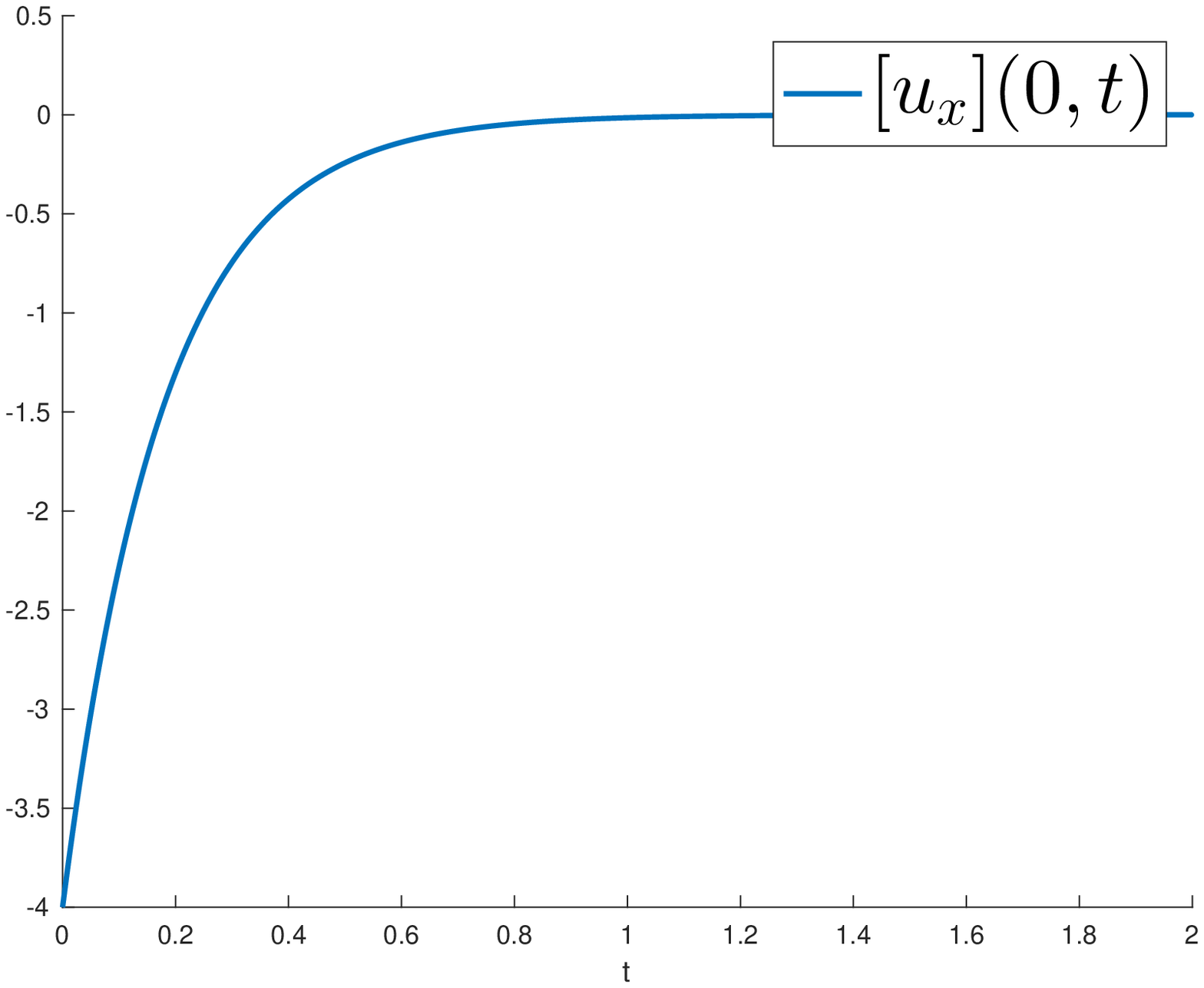}\qquad}\\
\subfigure[$\bm{[}u_x\bm{]}(0.5,t)$ by method \ref{method 1}]{
\includegraphics[width=5.3cm]{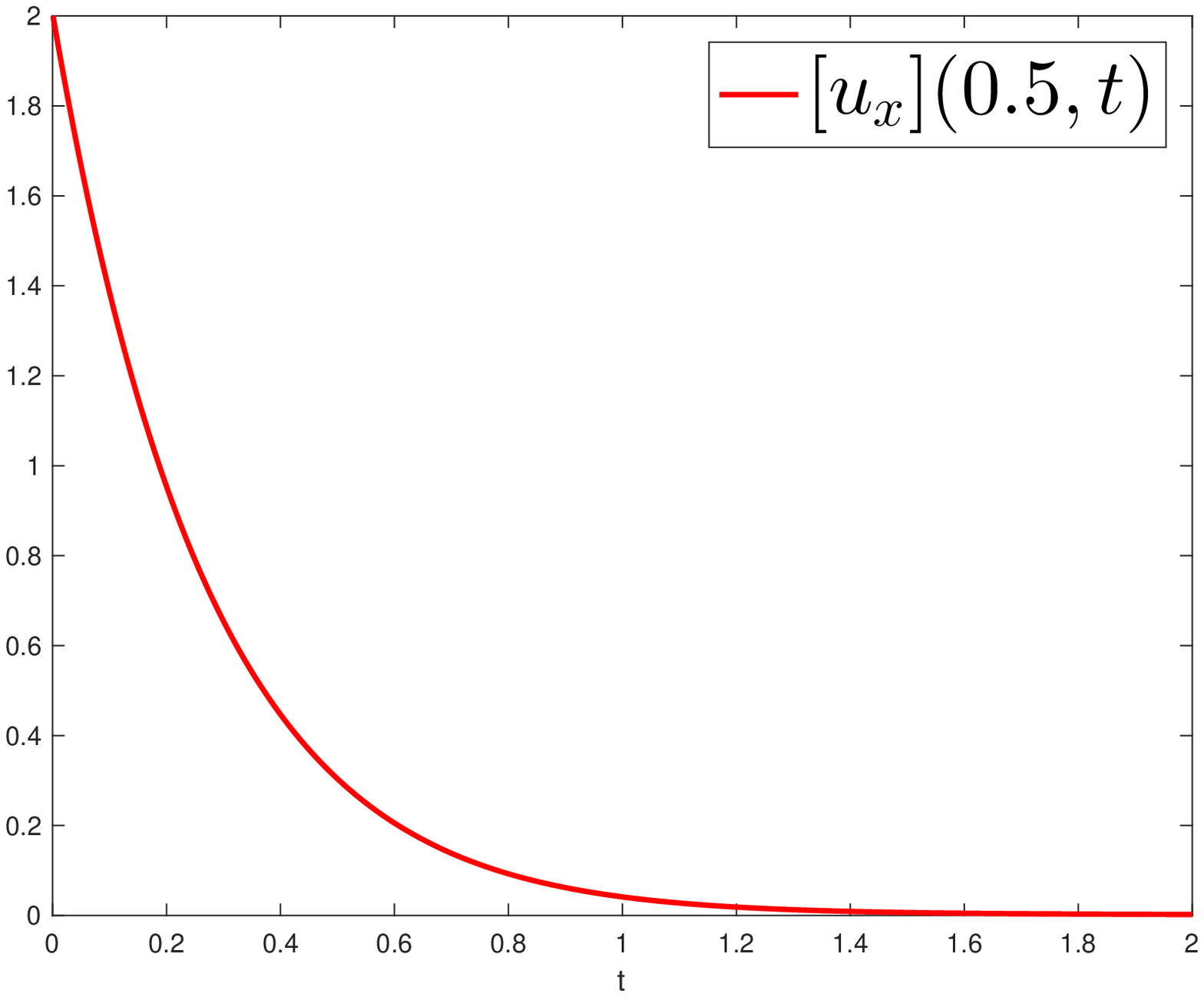}\qquad}
\subfigure[$\bm{[}u_x\bm{]}(0.5,t)$ by method \ref{method 2}]{
\includegraphics[width=5.3cm]{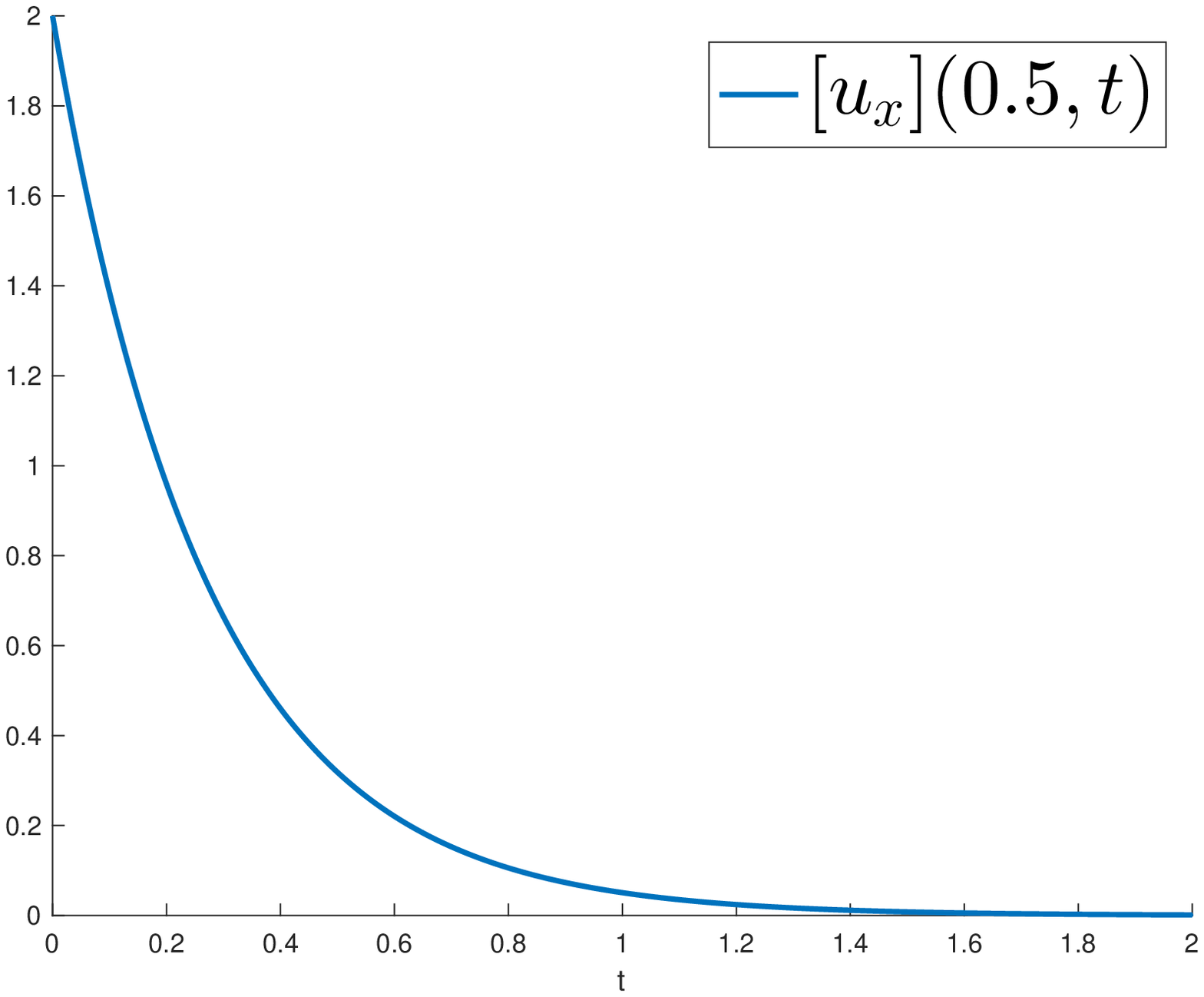}\qquad}
\caption{Plots of $ \bm{[}u_x\bm{]}(x,t)$  computed by different methods with $p = 0.5$ and $x = 0$ and $0.5$ for $t\in(0, 2)$}
\label{fig14} 
\end{figure}

\subsubsection{$\psi_0$ smooth, $\zeta$ piecewise smooth}
Now, we consider the case where initial data $\psi_0$ is smooth but the horizon $\zeta$ is piecewise smooth, corresponding to the situation given in Theorem \ref{thm:jump-ux-dis}. 
Here we choose
$$H(s) = 20 e^{-10s^2}, \quad\psi_0(x)=e^{-10x^2}$$
and
 \beq\label{zeta1}
 \zeta(x) = \max\{\min\{kx,6\},0\}, \quad \text{for} \; k = 1, \; 2\;  \text{and} \; 3 \; \text{respectively}.
 \eeq
 
 In Figure \ref{fig4}, we can see that the horizon $\zeta$ selected here is three piecewise smooth functions with transition layers becoming narrower as the slope $k$ changes from 1 to 2 then to 3.  The jumps in the derivative $\zeta'$ are at $x=0$ and $x=6/k$. While the latter is within the nonlocal region, the former is a local-nonlocal transition point.  Thus, we may expect more evident jumps in the space derivative of the solution at $x=6/k$.
 
 The results of the solutions on $x\in (-4,6)$ 
are shown in Figure \ref{fig5} for $t\in (0, 2)$ with a top view
 and Figure \ref{fig6}   for   $x\in (-4,4)$  and $t\in (0, 1)$ for a zoomed-in 3D view of Figure \ref{fig5}.
The folding "vertical line" appearing in the contours again are capturing the discontinuities of $u_x$, which in this case, are generated by the discontinuities in the horizon functions.
The Gaussian peak in the initial data gets smeared as time goes on, due to the dissipative effect present in the nonlocal equations. Still, one can see that the smeared peak still travel largely along the local characteristics. A closer look shows that the smearing (dissipation) due to the nonlocal interactions gets more evident for larger $k$ (from the left plot to the right plot), which is consistent to the fact the nonlocal region expands with larger $k$,  Moreover, the 
 folding lines also become more evident (as indicated by more dramatic color changes shown in the plots).

 \begin{figure}[htb]\label{ex21}
\centering
{
\includegraphics[width=7.5cm]{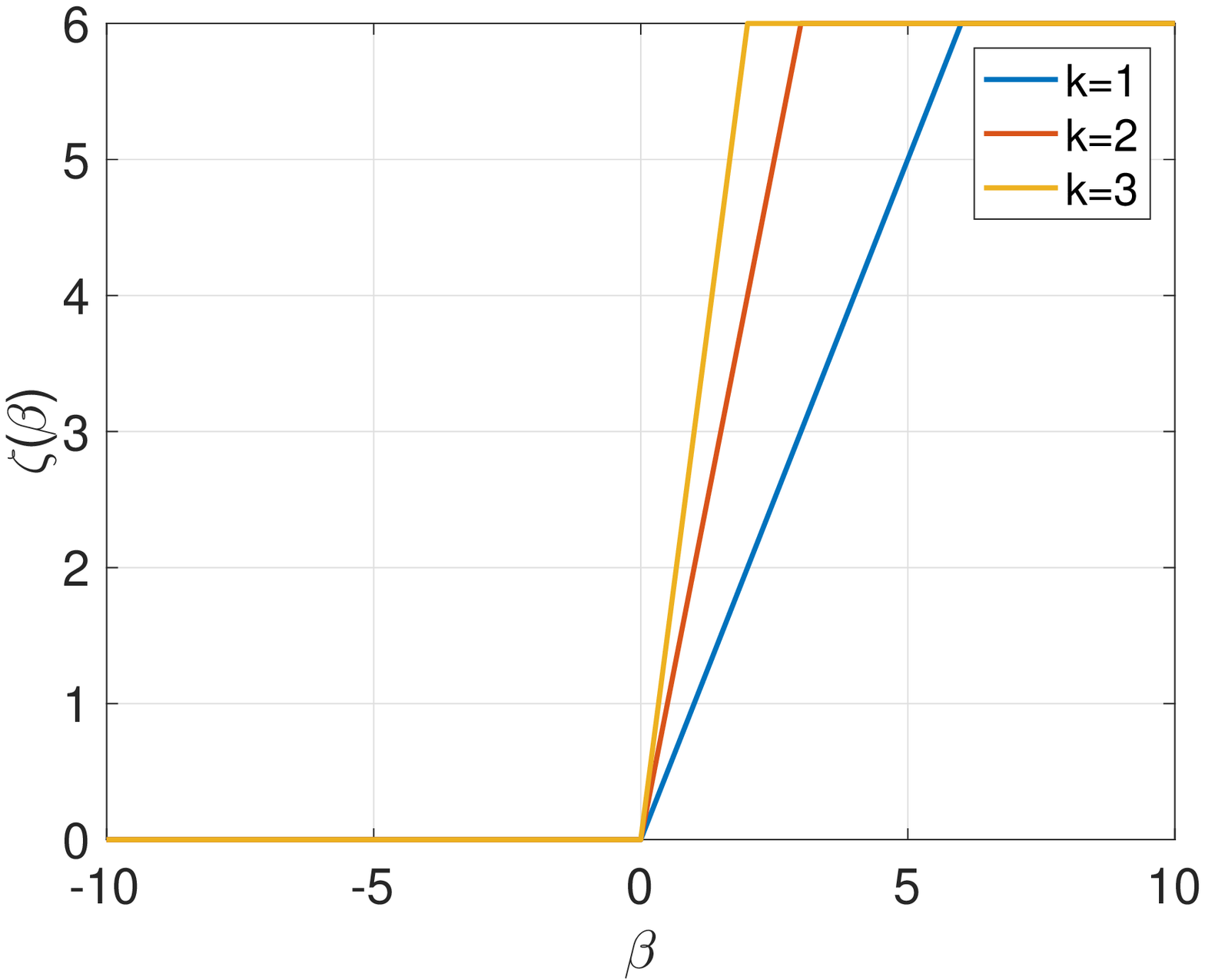}}
\caption{Plot of different choices of piecewise smoothly defined $\zeta = \zeta(\beta)$}
\label{fig4} 
\end{figure}

\begin{figure}[htb]
\centering
\subfigure[$k=1$]{
\includegraphics[width=4.3cm]{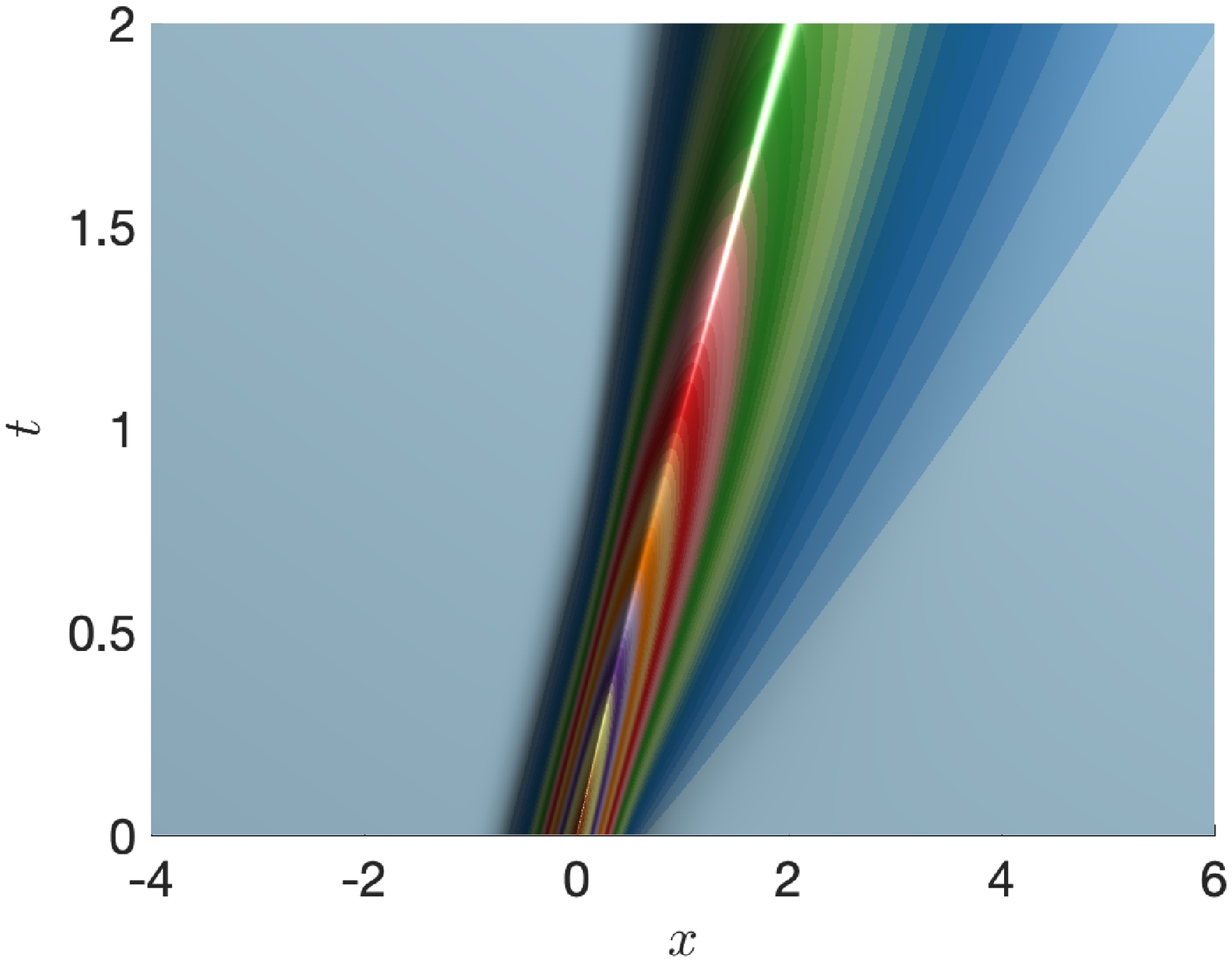}}
\subfigure[$k=2$]{
\includegraphics[width=4.3cm]{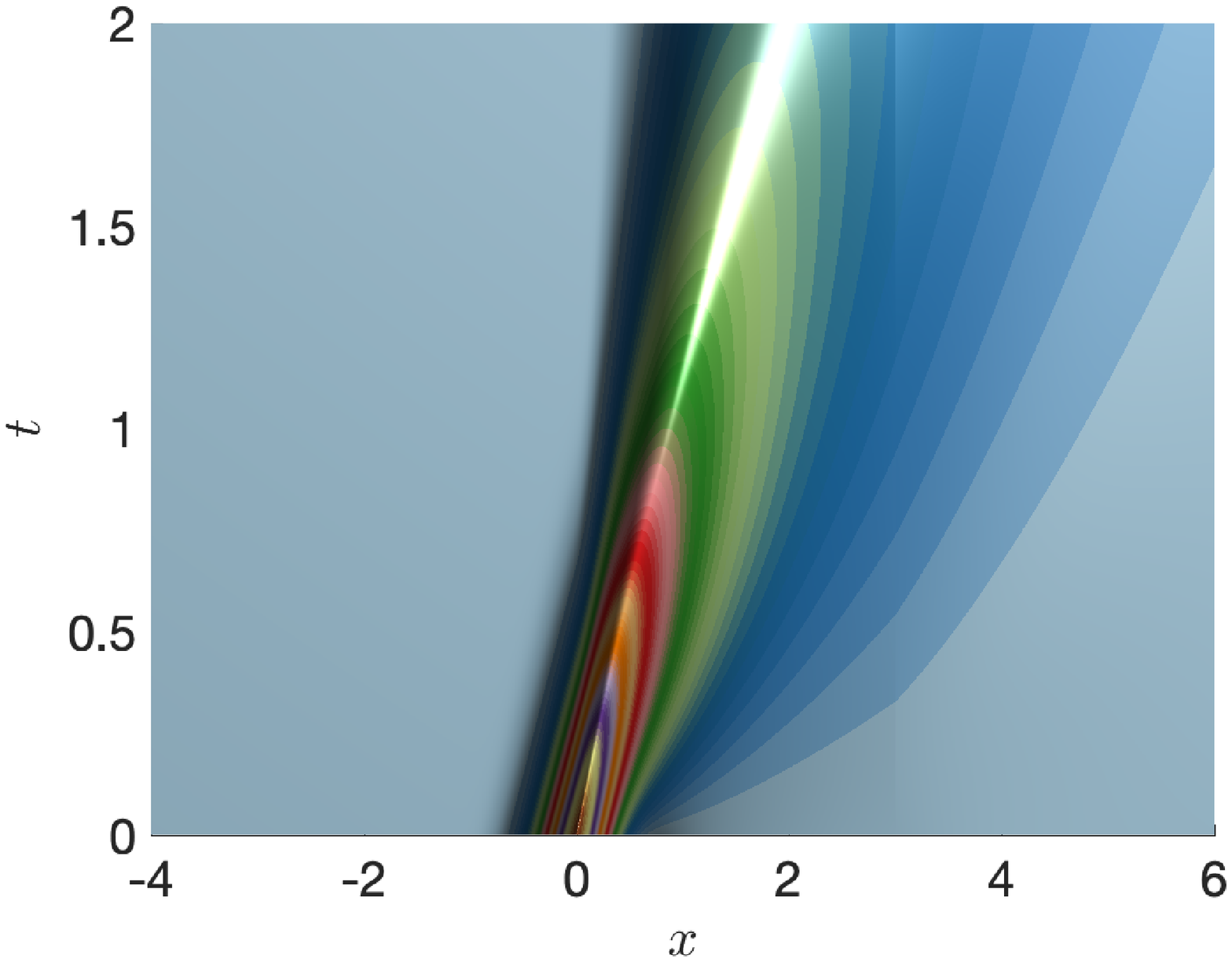}}
\subfigure[$k=3$ ]{
\includegraphics[width=4.3cm]{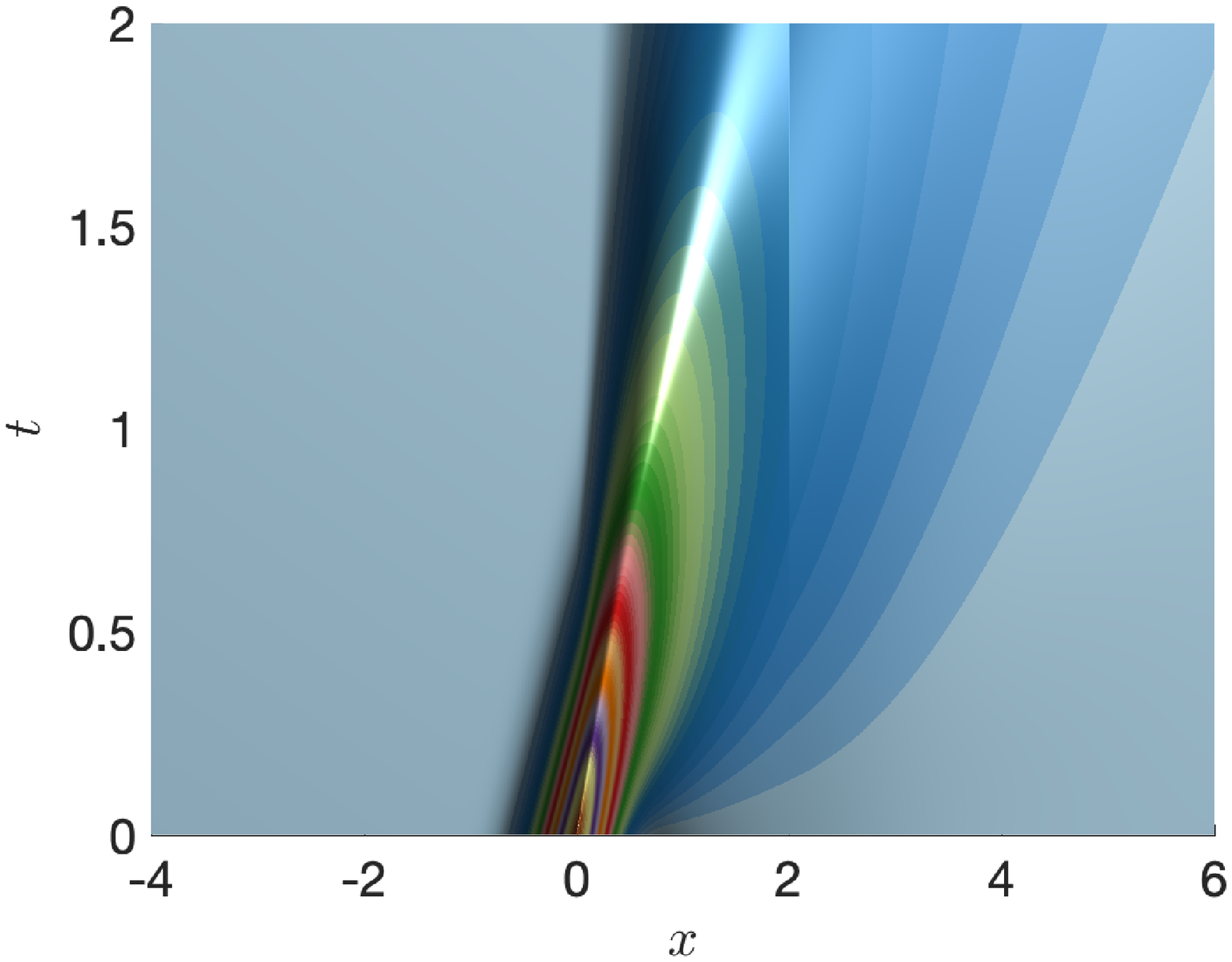}}
\caption{Wave propagation  corresponding to different choices of $k$ for piecewise smooth $\zeta = \zeta(x)$ with a smooth initial data: top view}
\label{fig5} 
\end{figure}

\begin{figure}[htb]
\centering
\subfigure[$k=1$]{
\includegraphics[width=4.3cm]{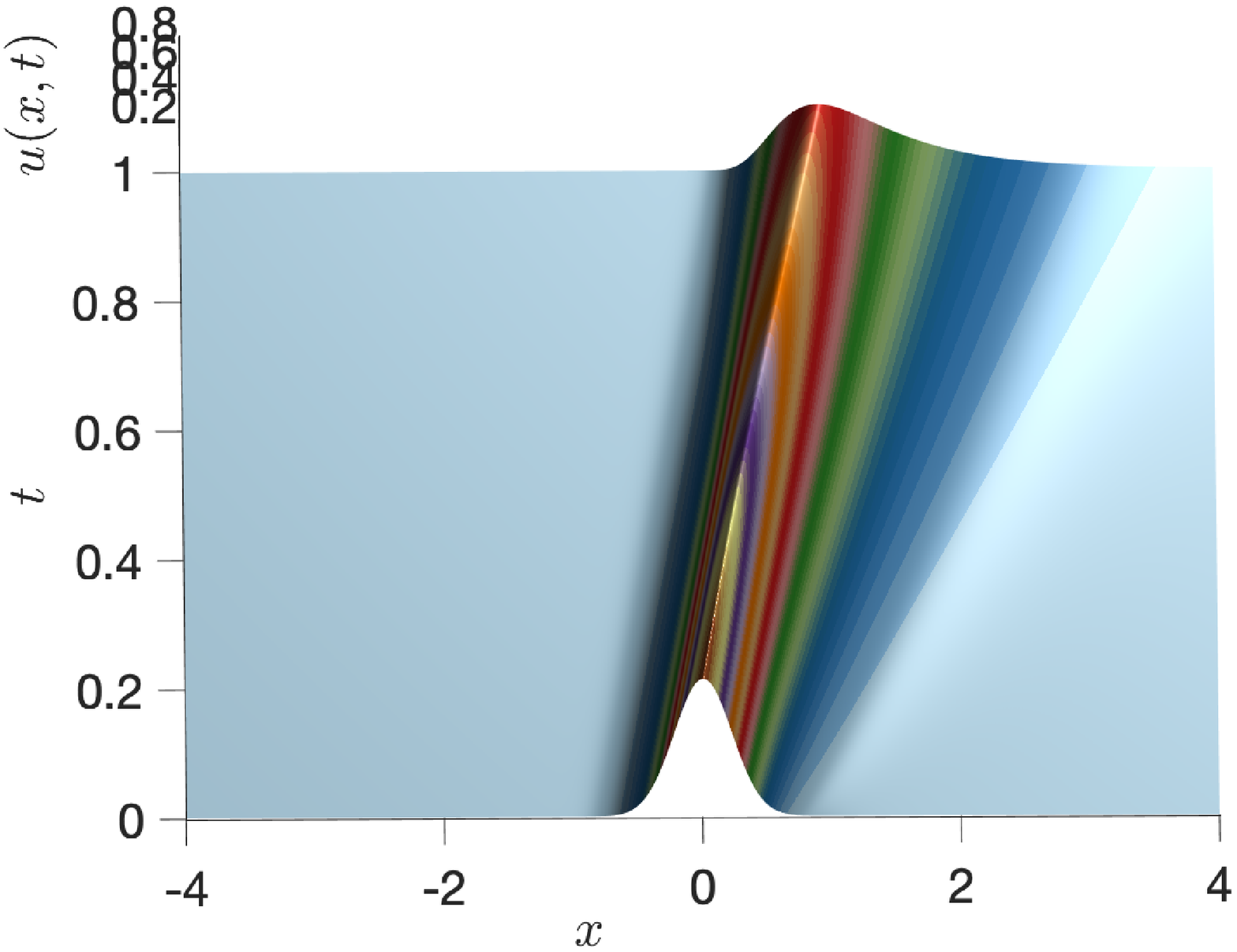}}
\subfigure[$k=2$]{
\includegraphics[width=4.3cm]{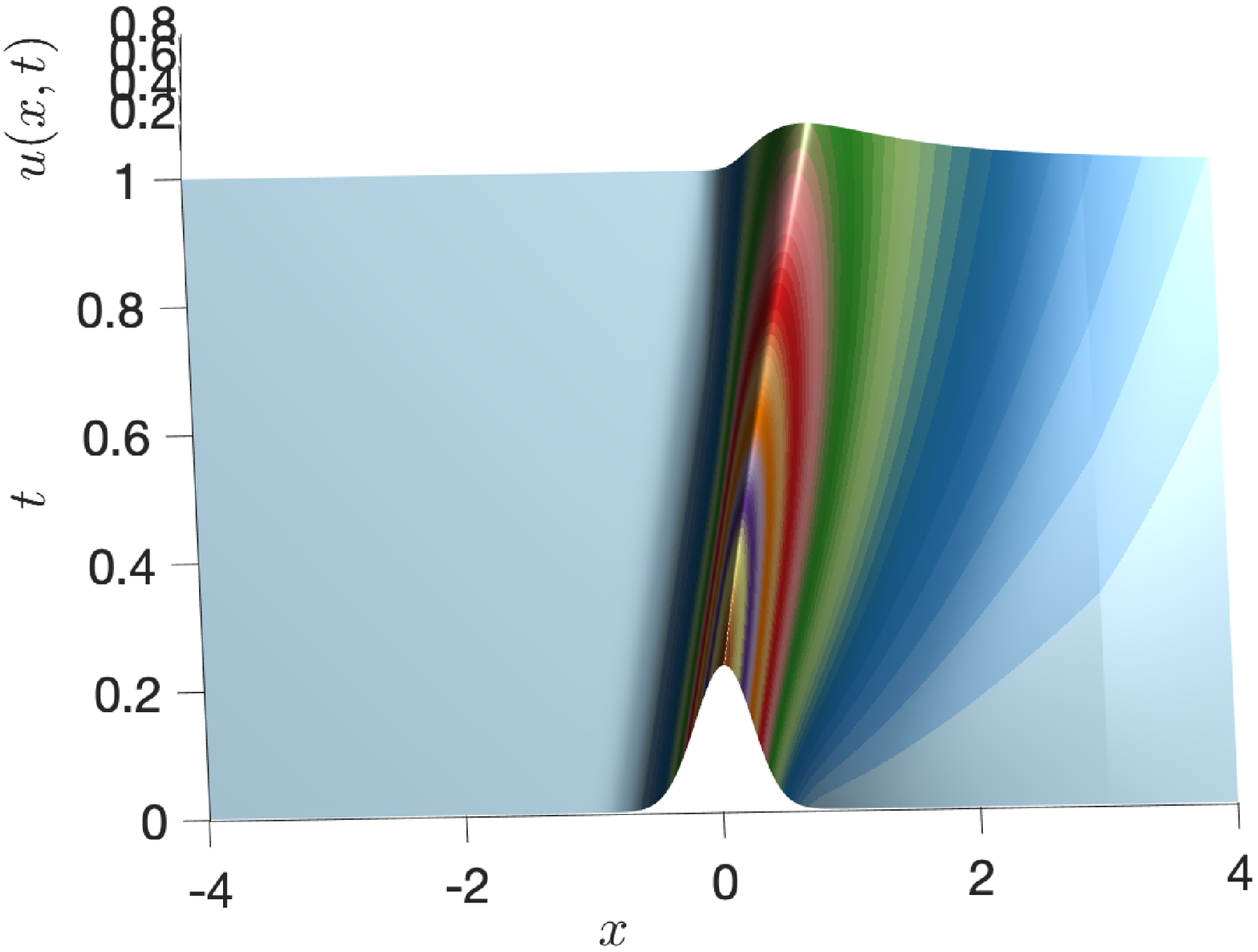}}
\subfigure[$k=3$]{
\includegraphics[width=4.3cm]{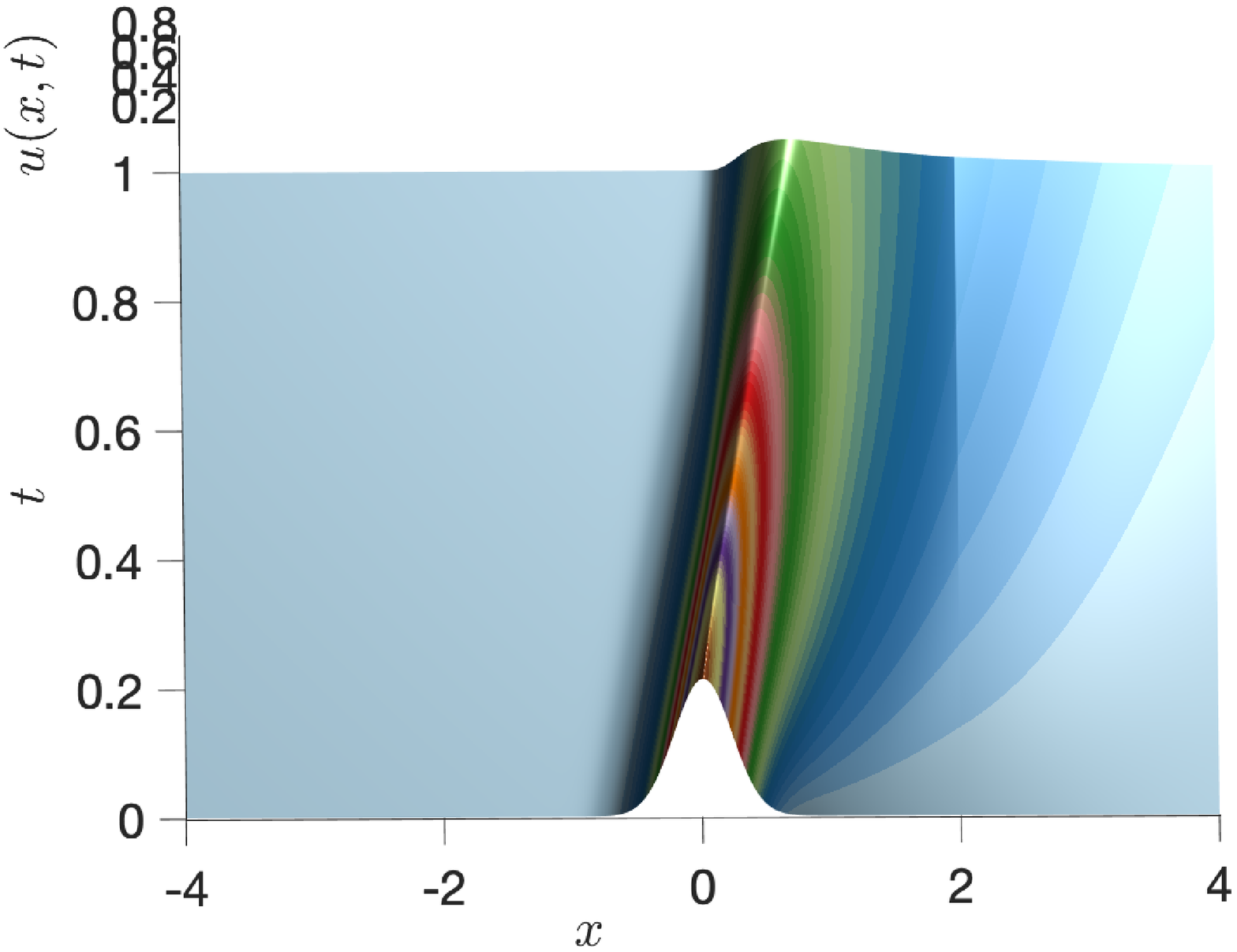}}
\caption{Wave propagation  corresponding to different choices of $k$ for piecewise smooth $\zeta = \zeta(x)$ with a smooth initial data: zoom-in 3D view}
\label{fig6} 
\end{figure}

In Figure \ref{fig7}, we examine more closely the time evolution of $ \bm{[}u_x\bm{]}(x,t)$ at $x = \frac6k$ for different  horizon choices. We present the results computed by the direct evaluation of $ \bm{[}u_x\bm{]}(x,t)$ using the
 method  \ref{method 2}.

For different values of $k$, we can observe some consistent patterns.  First, as predicted by
 the equation \eqref{eq:jump-ux-dis}, with the smooth initial data, we can see that  there is no jump  initially in  $u_x(\frac6k,0)$. As $t>0$, 
  $\bm{[}u_x\bm{]}(\frac6k,t)$ starts to change with the overall magnitude gets larger for the case with a larger value of $k$. This can be seen from
  the equation \eqref{eq:jump-ux-dis}, with the smooth initial data, we can see that $ \bm{[}u_x\bm{]}(x,t)$ is proportional to the jump in 
$\zeta'$. The latter is measured by $k$ for our example. As the initial smooth Gaussian peak gets smeared, it still travels along the characteristic lines 
which contributes to the more dramatic change in   $\bm{[}u_x\bm{]}(\frac6k,t)$ during the time when the peak passes through  $x = \frac6k$ later in time,
Overall, these observations are consistent to the patterns of the computed numerical solutions presented in Figures \ref{fig5} 
 and Figure \ref{fig6}.

\begin{figure}[htb]
\centering
\subfigure[$k = 1$ and $x = 6$]{
\includegraphics[width=4.3cm]{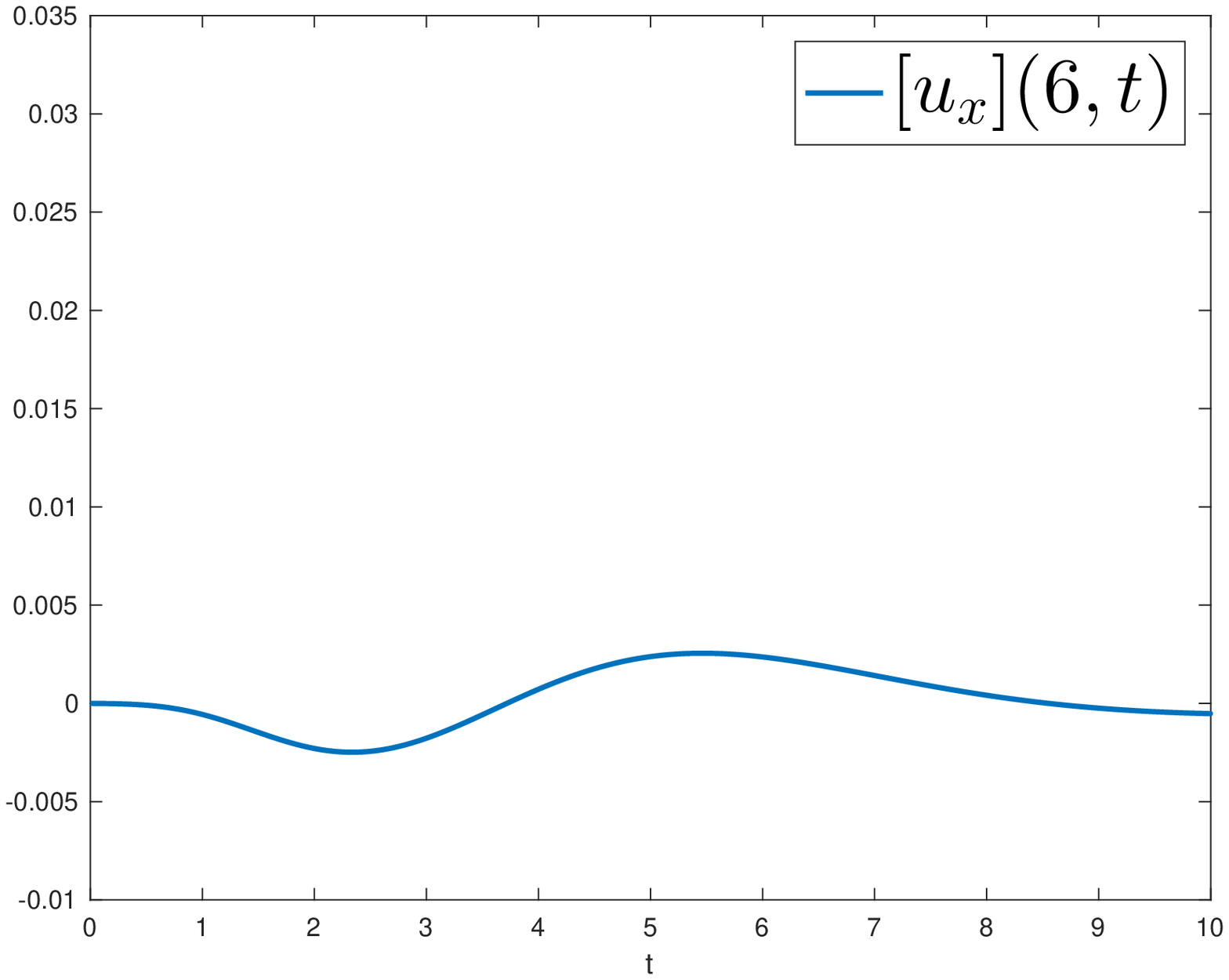}}
\subfigure[$k = 2$ and $x = 3$]{
\includegraphics[width=4.3cm]{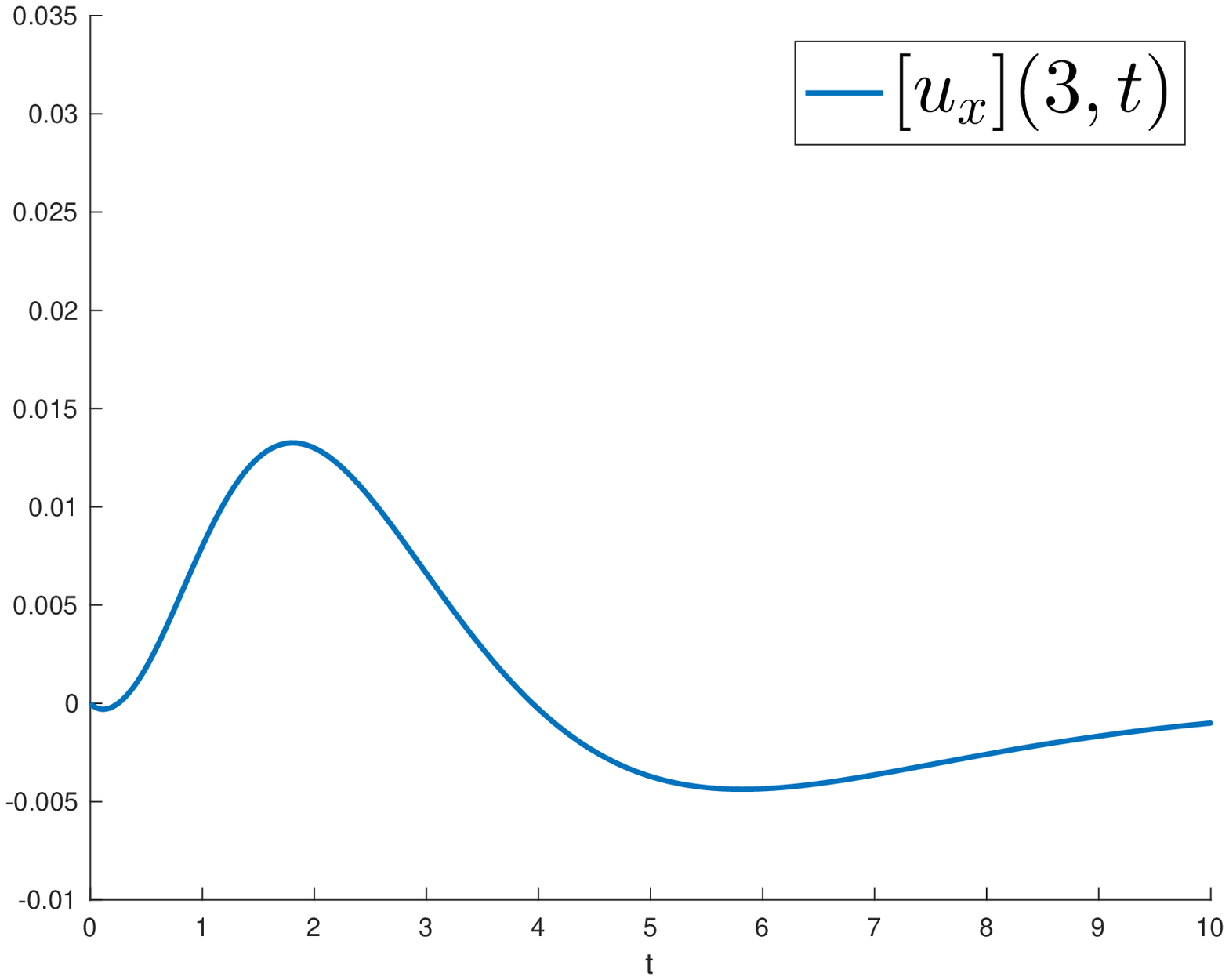}}
\subfigure[$k = 3$ and $x = 2$]{
\includegraphics[width=4.3cm]{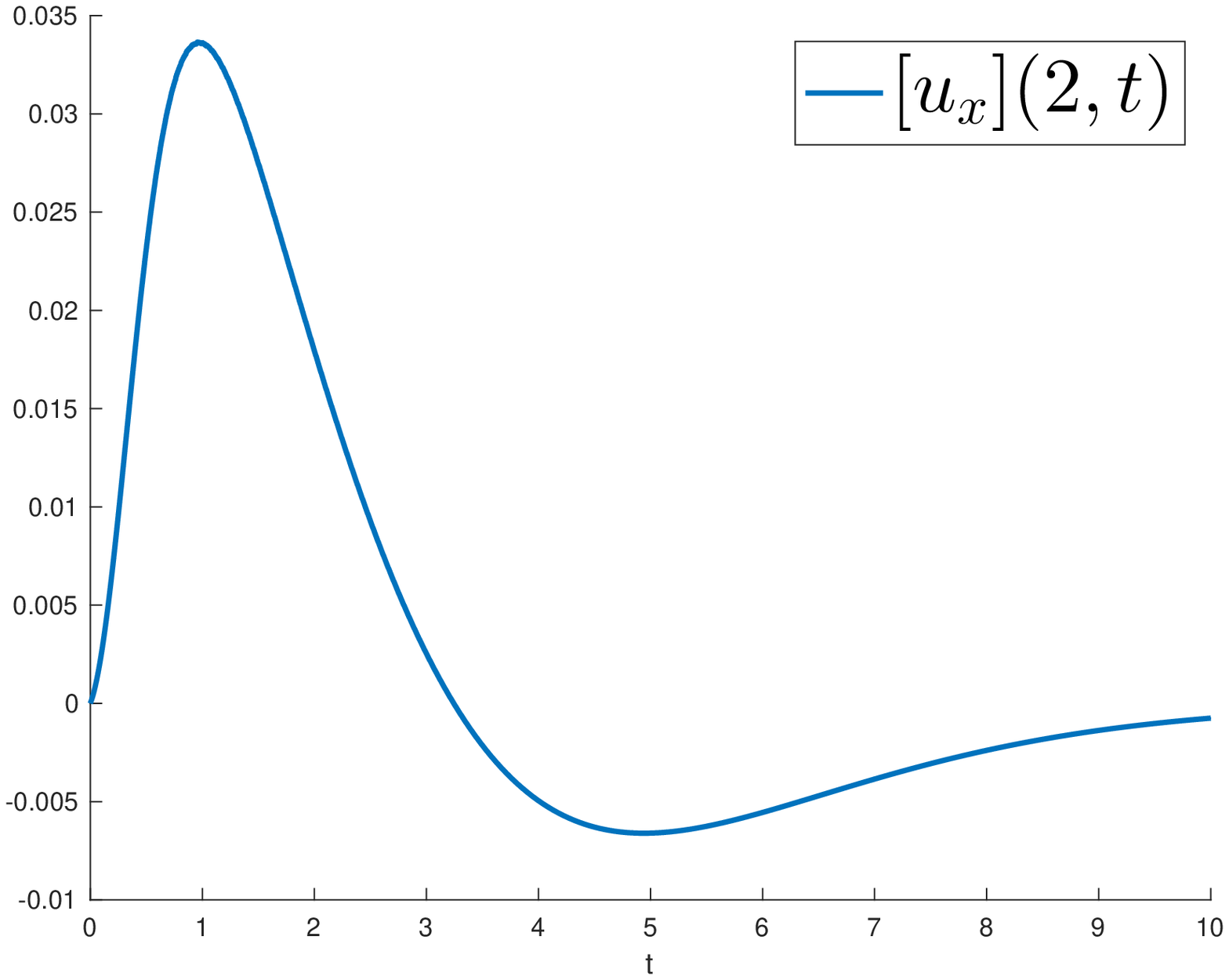}}
\caption{Plots of $ \bm{[}u_x\bm{]}(x,t)$ computed by method  \ref{method 2} corresponding to different $k$ and $x$ up to $t\in(0, 10)$}
\label{fig7} 
\end{figure}

\subsubsection{$\psi_0$ piecewise smooth, $\zeta$ piecewise smooth}
Now, we consider the case where both the horizon $\zeta$ and the  initial data $\psi_0$ are piecewise smooth. corresponding to the situation given in theorem \ref{thm:jump-ux-dis}. 

Here we choose $H(s) = 20 e^{-10s^2}$, 
$\psi_0(x)$ is given by \eqref{eq:initalhat} with $p =0.5$, $1$ and $2$ respectively,
and $ \zeta$ is taken to be the same as \eqref{zeta1}, with $k = 1$, $2$ and $3$ respectively.
 
 In this case, according to the previous theorem, we expect that to see discontinuities  in $u_x$ at locations where the derivative of either $\psi_0$ or $\zeta$ is discontinuous.
The results of the solutions  for $k=1$ and $p=0.5$ and $p=1.0$ on $x\in (-2,4)$ 
are shown in Figure \ref{fig15} for $t\in (0, 2)$ with a top view
 and Figure \ref{fig16}  for $t\in (0, 1)$  with a 3D view
 to illustrate solutions.  Similar results for $k=3$ and $p=2.0$ are presented in Figure \ref{fig16a} for $t\in (0, 3)$ with both top and 3D views.
Some folding "vertical lines", again for discontinuities of $u_x$ at $x = 0$, $x = p$ and $x = \frac6{k}$, can be observed, which represent the stationary discontinuities in the spatial derivatives of solutions at those locations. Also, relatively speaking,  larger $k$ and smaller $p$ would all make the jumps in the spatial derivatives more visible.

\begin{figure}[htb]
\centering
\subfigure[$k=1,  p = 0.5$]{
\includegraphics[width=6.5cm]{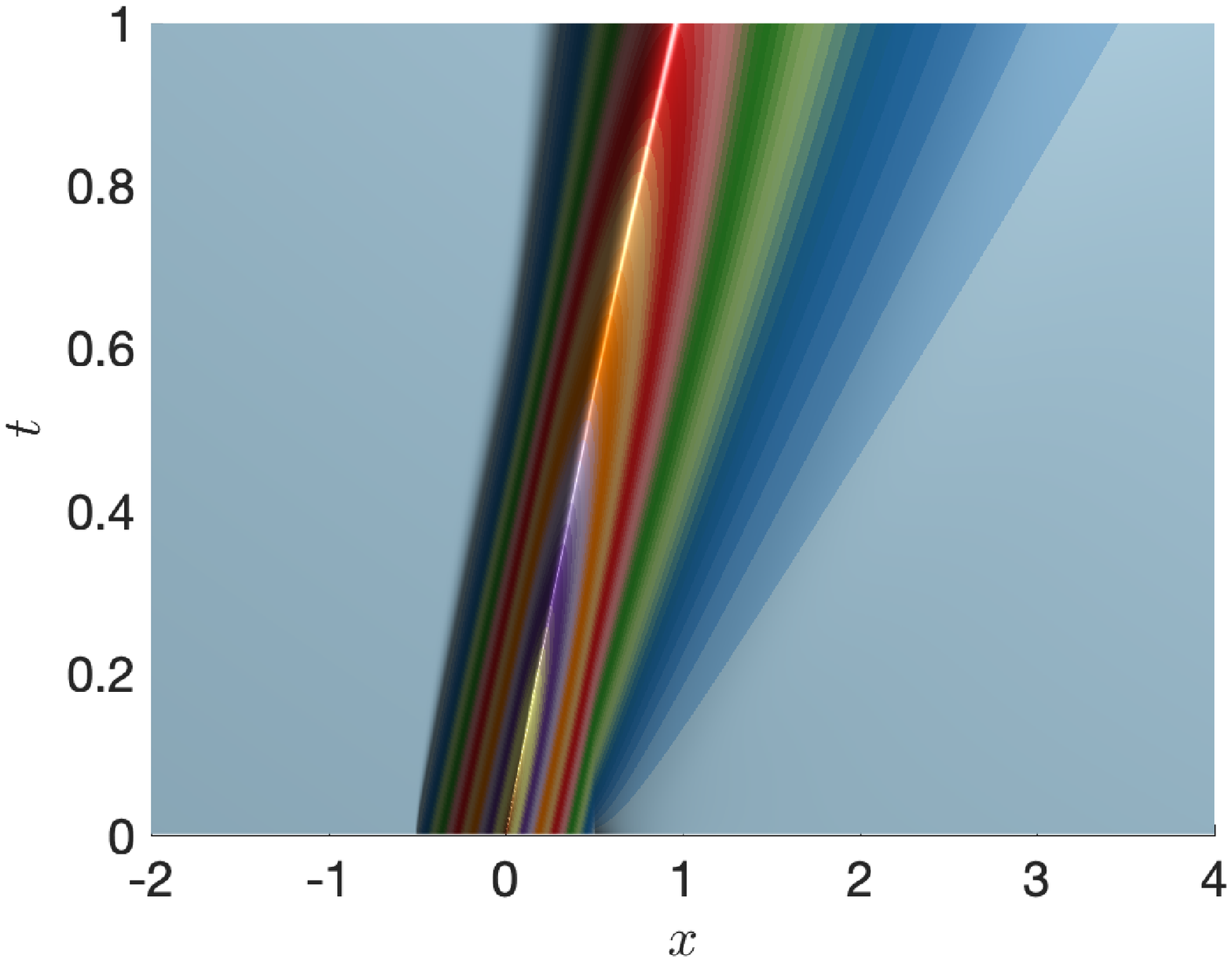}}
\subfigure[$k=1,  p = 1.0$]{
\includegraphics[width=6.5cm]{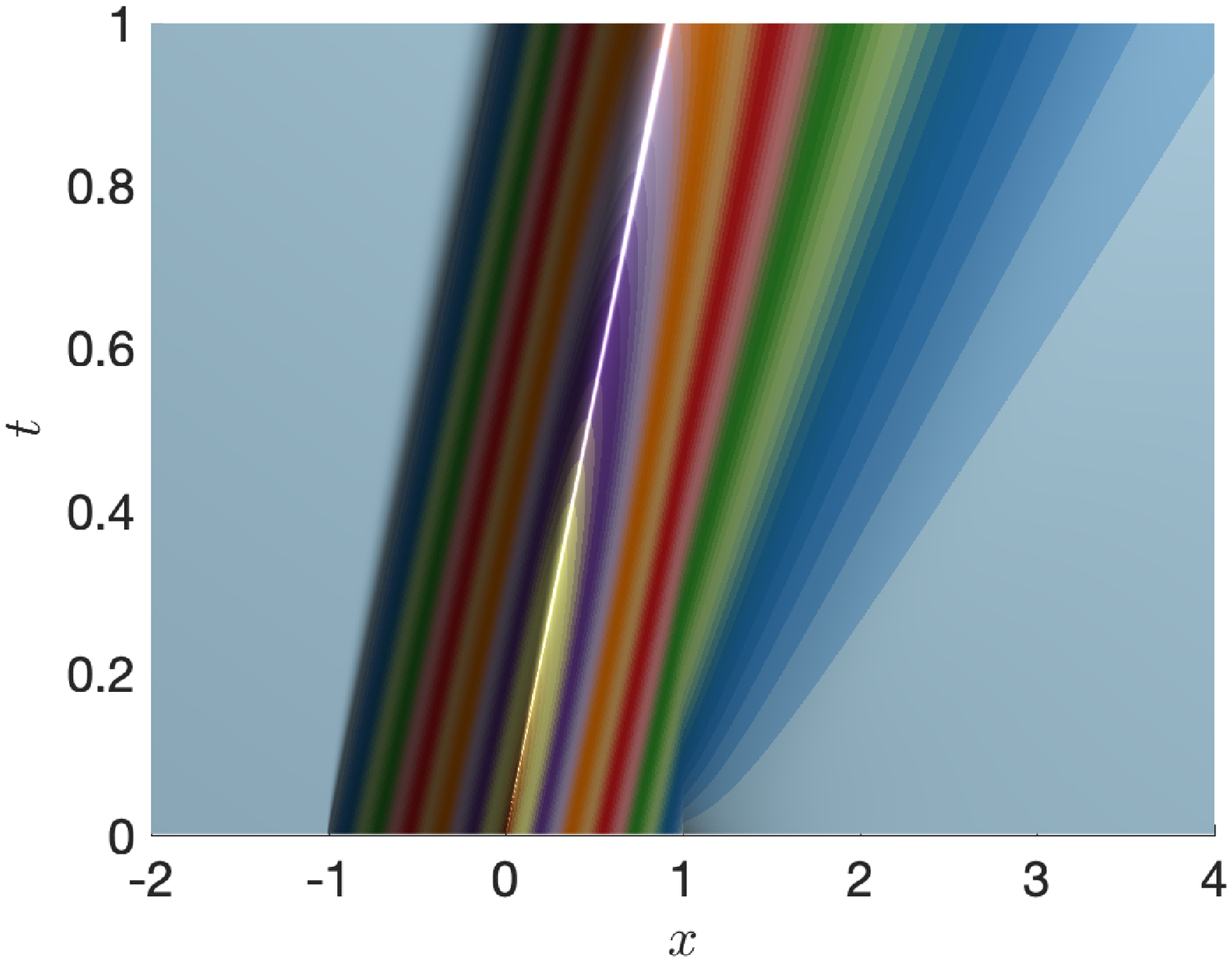}}
\caption{Wave propagation corresponding to different choices of  piecewise smooth $\zeta$ with a piecewise smooth initial data $\psi_0$: top view.}
\label{fig15} 
\end{figure}
\begin{figure}[htb]
\centering
\subfigure[$k=2,  p = 0.5$]{
\includegraphics[width=6.5cm]{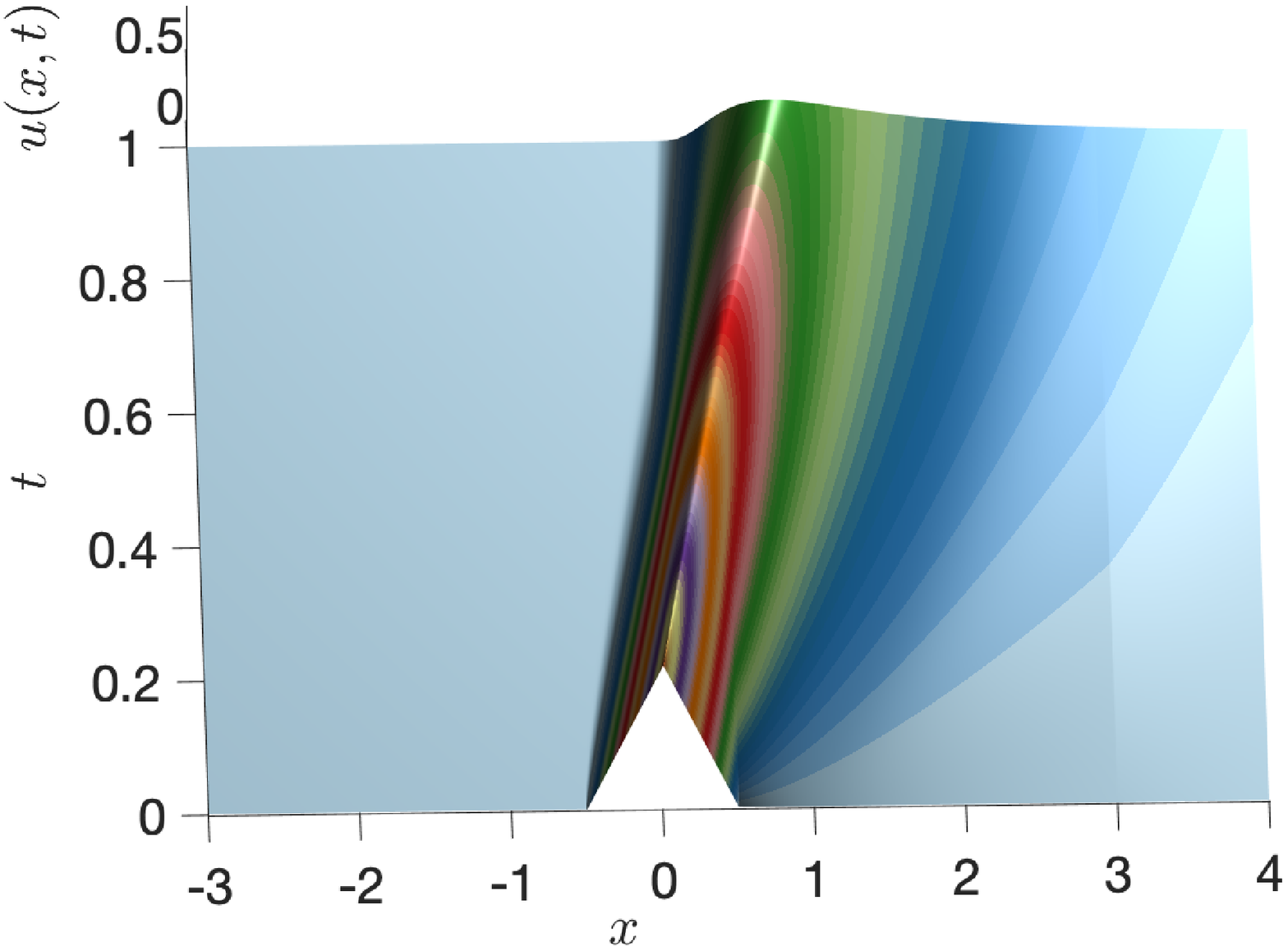}}
\subfigure[$k=2,  p = 1.0$]{
\includegraphics[width=6.5cm]{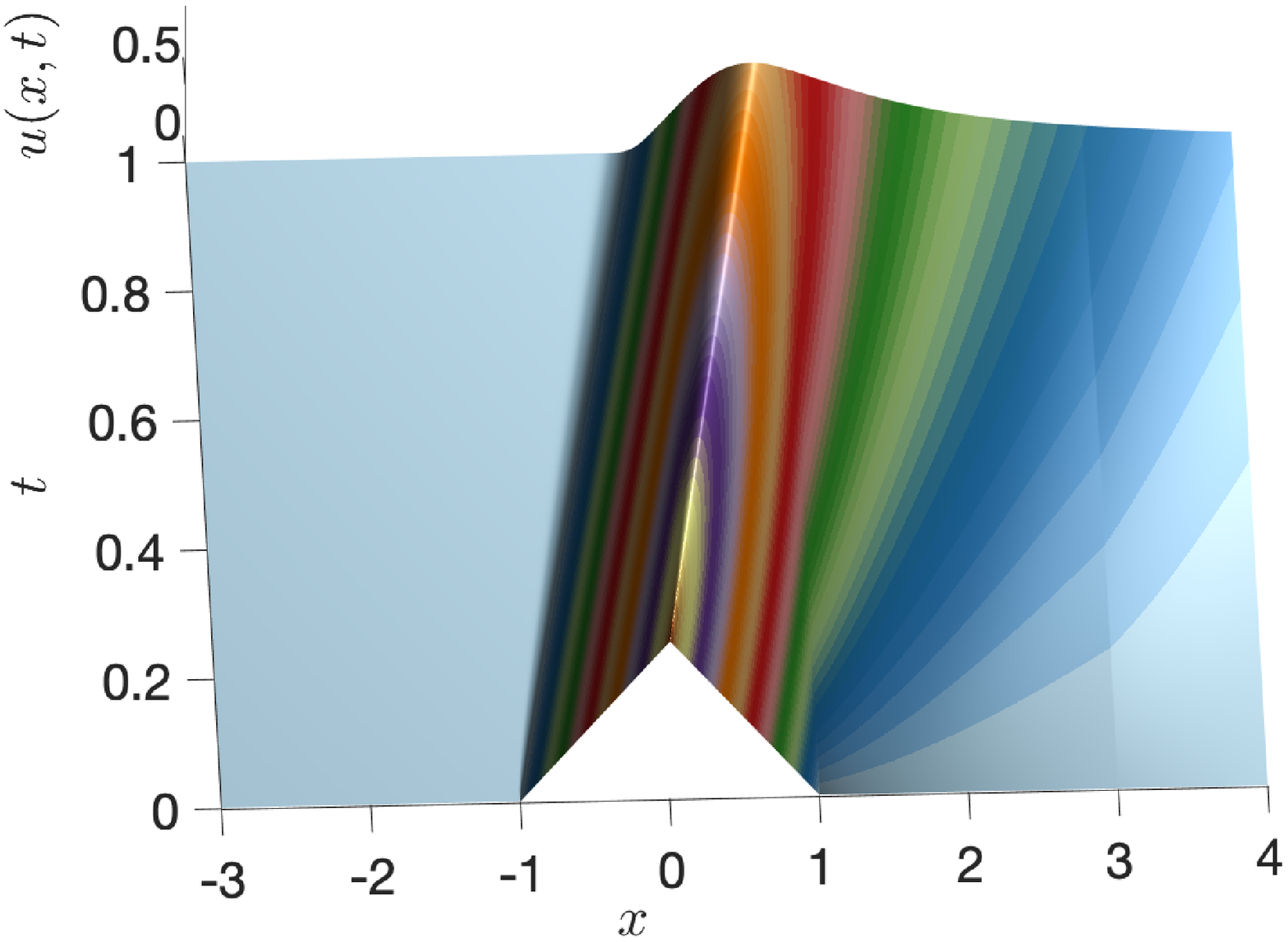}}
\caption{Wave propagation corresponding to different choices of  piecewise smooth $\zeta$ with a piecewise smooth initial data $\psi_0$: zoom-in 3D view.}
\label{fig16} 
\end{figure}
\begin{figure}[htb]
\centering
\subfigure[top view]{
\includegraphics[width=6.2cm]{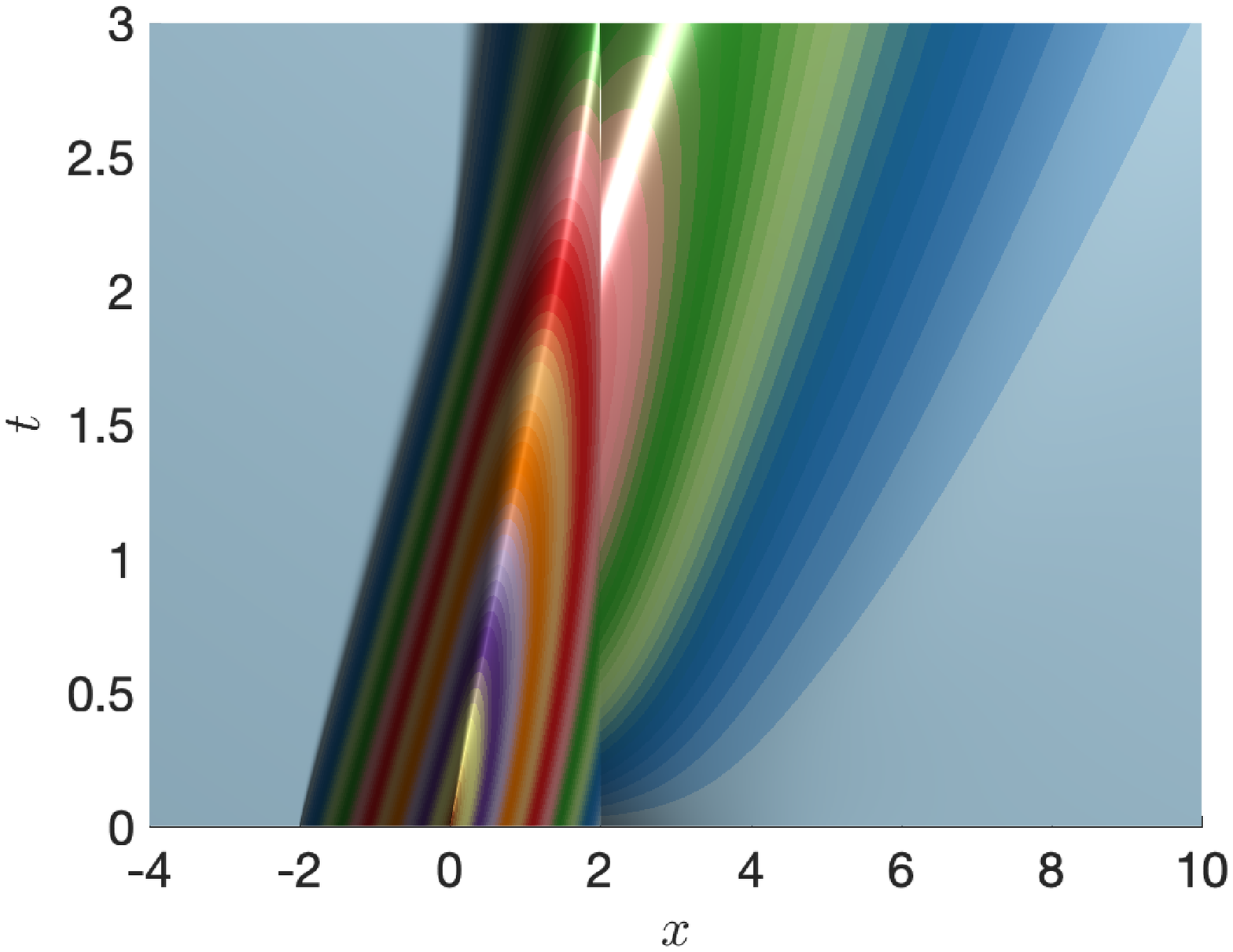}}
\subfigure[3D view]{
\includegraphics[width=7cm]{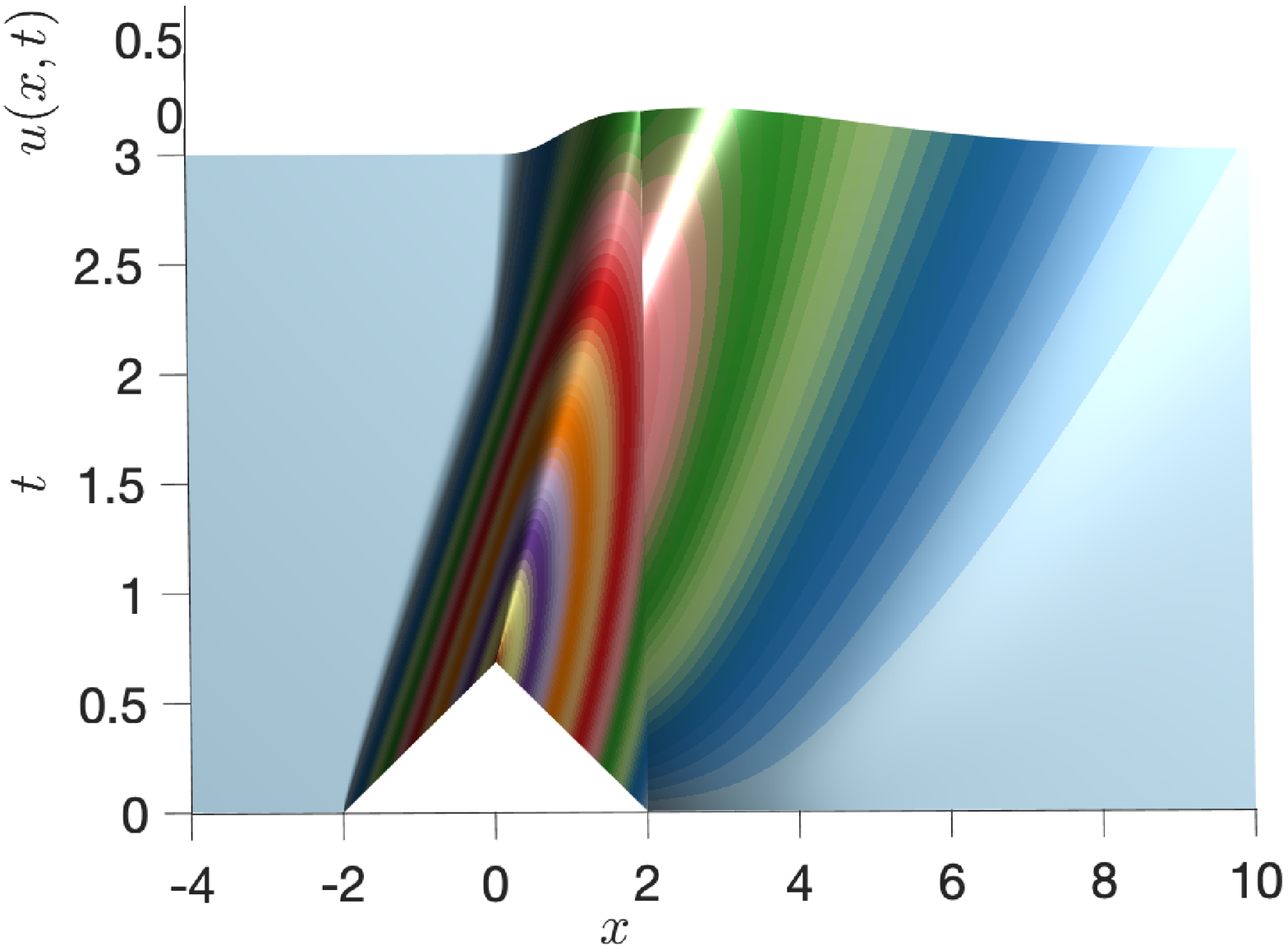}}
\caption{Wave propagation corresponding to different views of  piecewise smooth $\zeta$ with a piecewise smooth initial data $\psi_0$ with $k = 3$ and $p=2.0$.}
\label{fig16a} 
\end{figure}

To see the evolution of the jumps in $u_x$ in more details, we plot $\bm{[}u_x\bm{]}(x,t)$ in time for 
 $x=p$ and $x=6/k$
 corresponding to  the cases with  $x_0=0$, $p=1 $ and $2$ 
for the initial data and $k=2$ and $3$ for the horizon function.  We note that the behavior of discontinuity of $u_x$ at $x = 0$ is more subtle, as it is caused by discontinuities of $\psi_{0x}$ and  $\zeta'$
 and the additional fact that there is a transition from local ($\zeta=0$) to nonlocal ($\zeta>0$) regions at $x=0$.
 Note that with $\zeta(0)=0$, so \eqref{eq:jump-ux2} is technically not applicable at $x=0$. In the fully nonlocal region, we can use the method mentioned in Section \ref{method 2} for the evaluation of the equation \eqref{eq:jump-ux2}.  The results are presented in the plots in in Figure \ref{fig20},

 From the plots in in Figure \ref{fig20}, we first look at the case $x=p=1$ with $k=2$, Here, $\bm{[}u_x\bm{]}(1,t)$ decays monotonically to 0, which is purely due to the discontinuity of $\psi_{0x}$,  which is similar to the behavior observed for discontinuous initial data with a smooth horizon function,

Meanwhile for $p=1$ and $k=2$, at $x = \frac6k=3$, there is no  jump of the derivative of the initial data at this location, so
the generation of discontinuity in $u_x$  is controlled by $\zeta'$. We see that  $ \bm{[}u_x\bm{]}(x,t)$ starts from zero and varies in time, and is influenced by the passing of the wave fronts, and the behavior is similar to those observed before in plots in Figures \ref{fig7}. 
 
As for the case $k=3$ and $p=2$,  we have the co-existence of jumps in $\psi_{0x}$ and  $\zeta'$ at both $x=0$ and $x=6/k=p=2$. 
Only $x=2$ is in the nonlocal region. 
 We can observe the behavior from Figure \ref{fig20} (c) is largely dominated by the  jump of  $\psi_{0x}$, yet the  fast exponential decay is no longer observed due to the contributions of jumps of  in  $\zeta'$, which makes the evolution different from ones observed earlier.
  
 To summarize,   similar to what are demonstrated in previous examples, 
results for this more example involving  jumps in both $\psi_{0x}$ and  $\zeta'$
based on  the  estimates in Section \ref{method 2} are consistent to those presented by the numerical solutions to the nonlocal models and also match well with the theoretical understanding.


 \begin{figure}[htb]
\centering
\subfigure[$x =p= 1.0$, $k = 2$]{
\includegraphics[width=4.3cm]{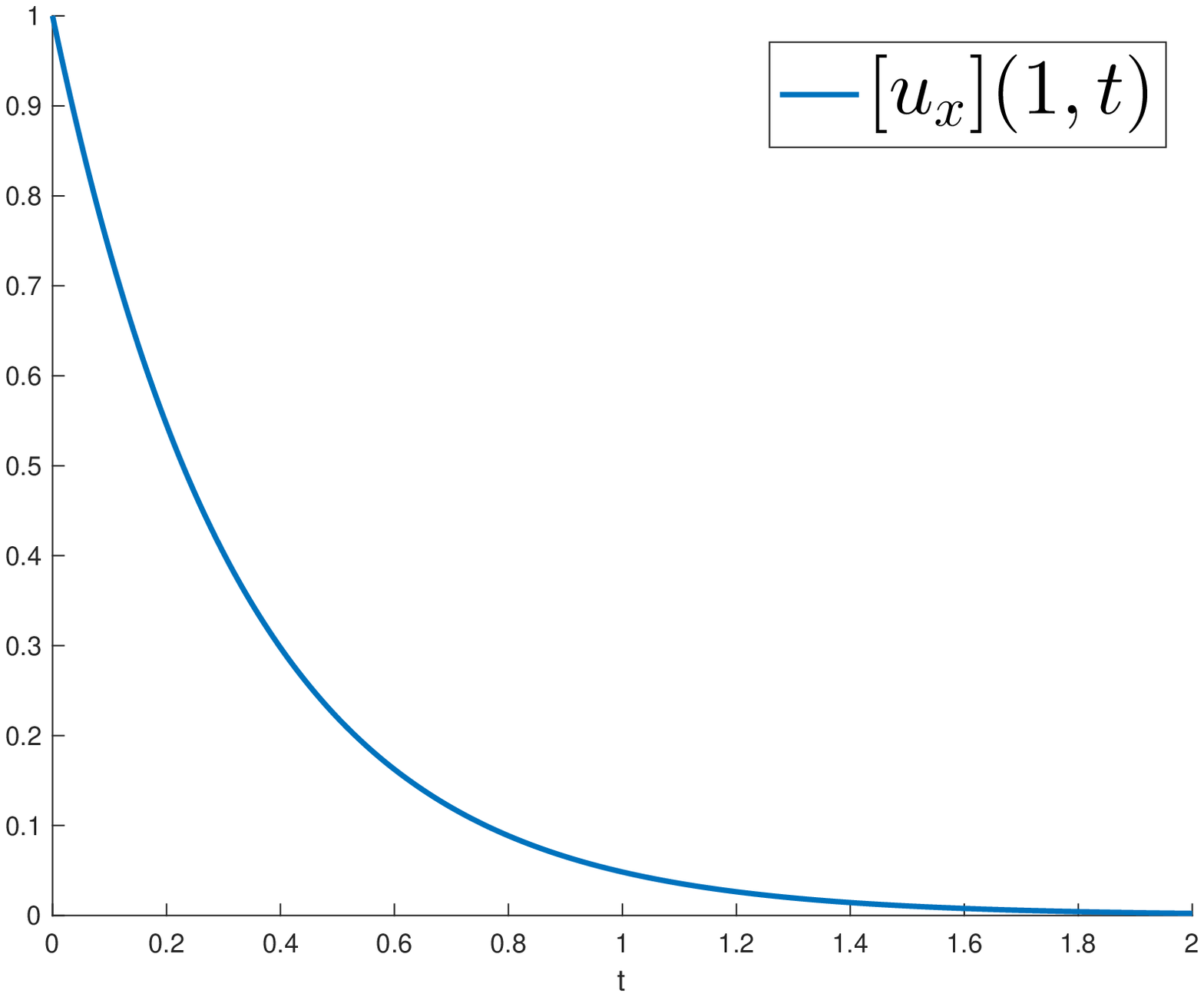}}
\subfigure[$x = \frac6k$,  $p= 1.0$, $k = 2$]{
\includegraphics[width=4.3cm]{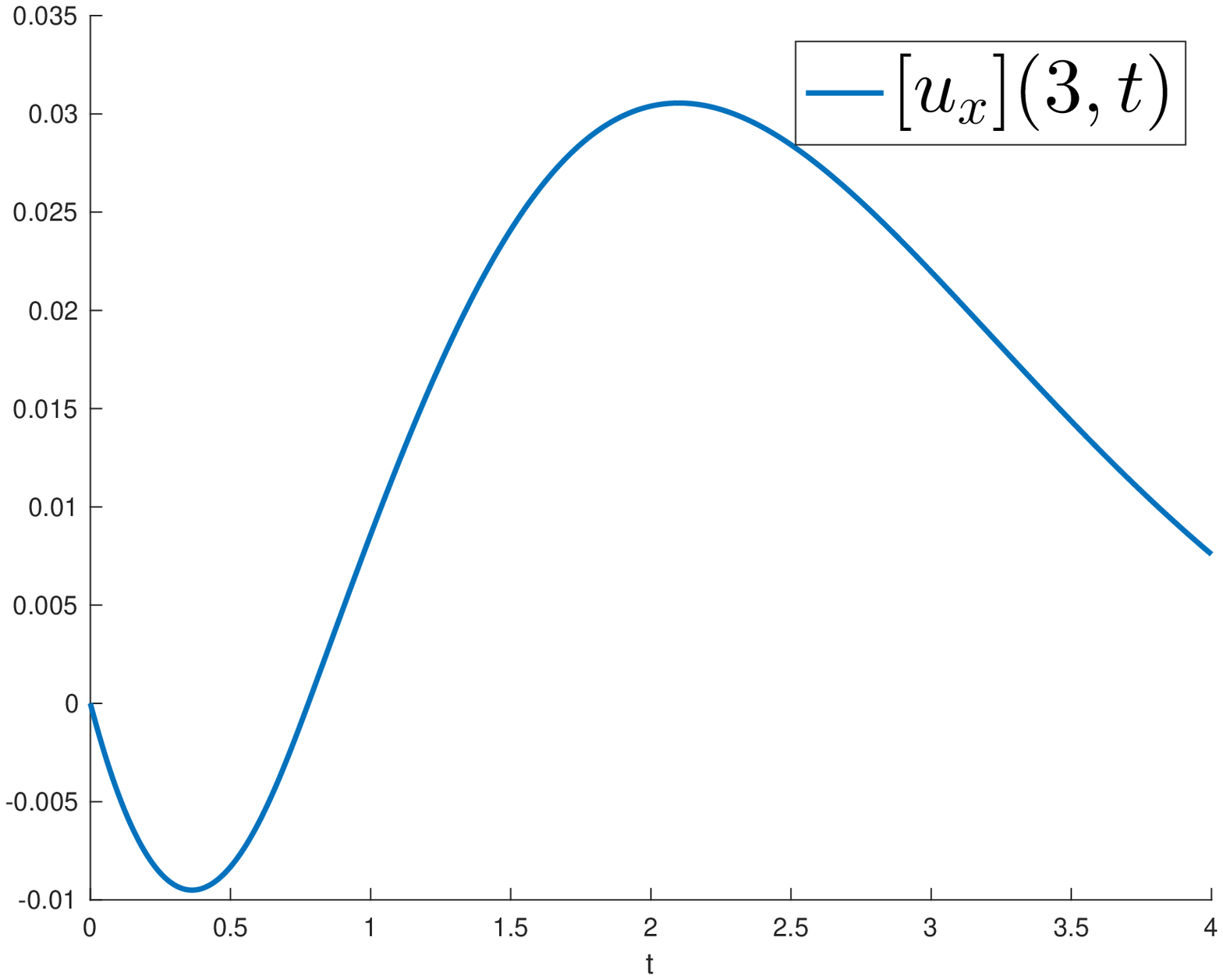}}
\subfigure[$x =p= 2.0=\frac6k$, $k = 3$ ]{
\includegraphics[width=4.3cm]{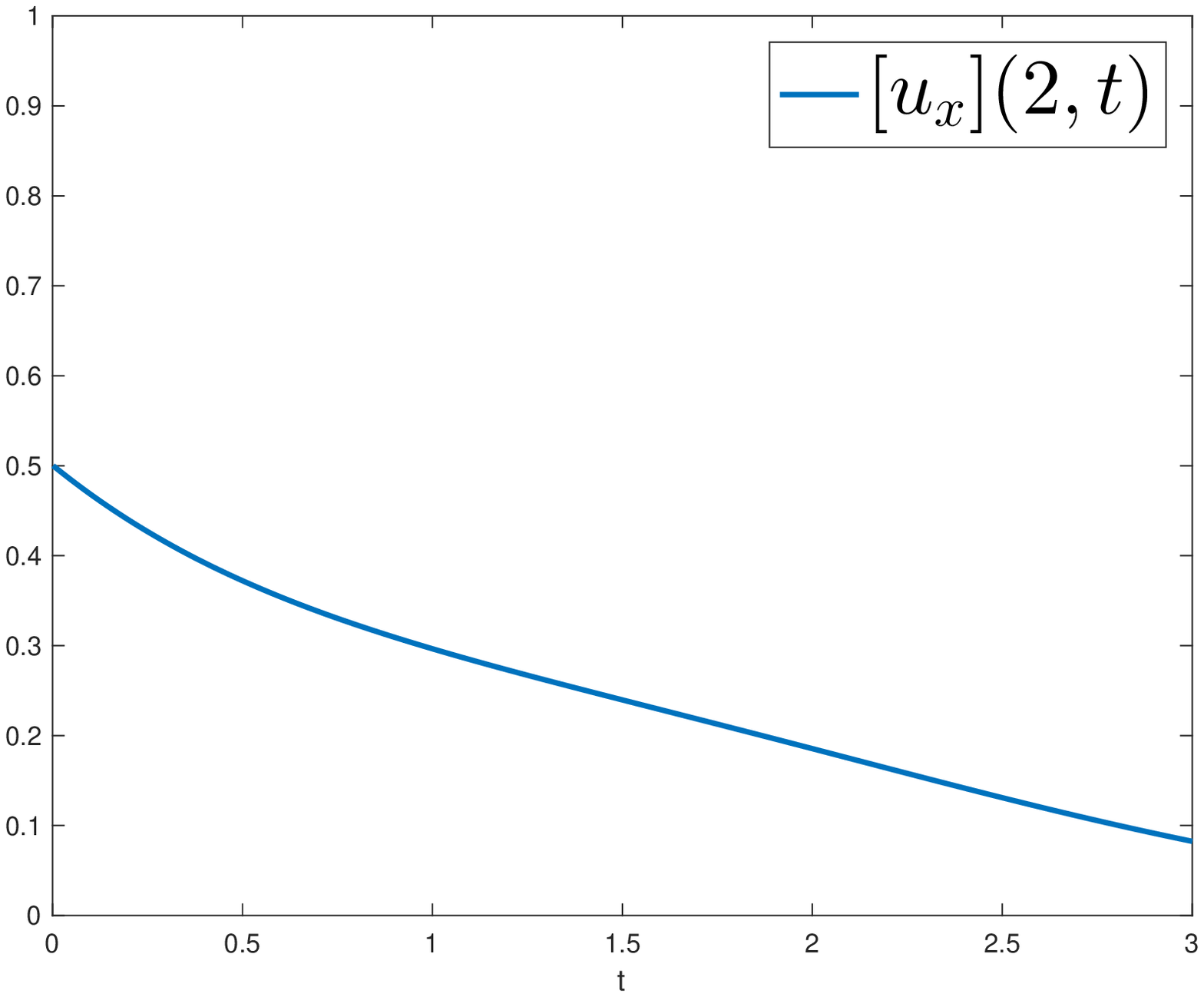}}
\caption{Plots of $ \bm{[}u_x\bm{]}(x,t)$ at different spatial locations for various values of $p$ and $k$.}
\label{fig20} 
\end{figure}

\section{Conclusion}
Understanding the phenomenon of singularity propagation is of particular interests in mechanics and other physical processes. In this work, we focus on the study of singularity propagation
 in waves modeled by a nonlocal convection equation with variable nonlocal interactions.  This new study  complements an earlier study in \cite{ac_2021} and  gives a deeper understanding  both theoretically and computationally. There are still many issues to be further investigated in the future. For example, 
 one may conduct more careful studies related to the wave dispersion and reflection due to the variations in the horizon parameter, as pursued in other recent studies \cite{helenli,dispersion}. The effect of boundary conditions is also an interesting problem worthy further investigation. Naturally, the one-dimensional linear problem is just one of the simplest examples, and for more practical and complex applications, we need to consider nonlinear and high-dimensional problems in future studies.

\section*{Declarations}

\subsection*{Data Availability}
The datasets generated during the current study are available from the corresponding author on reasonable request. They support 
our published claims and comply with field standards.

\subsection*{Competing of interest}
The authors declare that they have no known competing financial interests or personal relationships that could have appeared to 
influence the work reported in this paper.

\subsection*{Funding}
Research of Yan Xu is supported by  NSFC grant No. 12071455. Research of Qiang Du  is supported by NSF DMS-2012562.

\subsection*{Authors' contributions}
All authors contributed to the study conception and design. Material preparation, data collection and analysis were performed by Xiaoxuan Yu. The first draft of the manuscript was written by Xiaoxuan Yu and all authors commented on previous versions of the manuscript. All authors read and approved the final manuscript

\subsection*{Acknowledgements}
The authors would like to thanks Jiwei Zhang for helpful discussions on the subject.

 \subsection*{Financial disclosure}

None reported.

\bibliographystyle{plain}

\end{document}